\documentclass[12pt]{iopart}
\expandafter\let\csname 
equation*\endcsname=\relax
\expandafter\let\csname endequation*\endcsname=\relax
\usepackage{amsmath}
\usepackage{amssymb}
\usepackage{mathrsfs}
\usepackage{amsthm}
\usepackage{amsfonts}
\usepackage{algorithm}
\usepackage{epstopdf}
\usepackage{subfigure}
\usepackage{float}
\usepackage{graphicx}
\usepackage{color}
\usepackage{bm}
\usepackage{amsopn}
\usepackage{stmaryrd}
\DeclareMathOperator*{\argmin}{argmin}

\newtheorem{theorem}{Theorem}[section]
\newtheorem{lemma}[theorem]{Lemma}
\newtheorem{proposition}[theorem]{Proposition}
\newtheorem{remark}[theorem]{Remark}
\newtheorem{corollary}[theorem]{Corollary}

\begin{document}
\title[Inverse Problems]{Parameter Choices for Sparse Regularization with the $\ell_1$ Norm}

\author{Qianru Liu$^{1}$, Rui Wang$^{1}$, Yuesheng Xu$^{2}$ \footnote{Author to whom any correspondence should be addressed.} and Mingsong Yan$^{2}$}

\address
{$^{1}$ School of Mathematics, Jilin University, Changchun 130012, People’s Republic of China\\
$^{2}$ Department of Mathematics and Statistics, Old Dominion University, Norfolk, VA 23529, United States of America}
\eads{\mailto{liuqr19@mails.jlu.edu.cn}, \mailto{rwang11@jlu.edu.cn}, \mailto{y1xu@odu.edu} and \mailto{myan007@odu.edu}}

\vspace{10pt}
\begin{indented}
\item Dedicated to Professor Charles A. Micchelli on the occasion of his 80th birthday
with friendship and esteem
\end{indented}

\vspace{10pt}
\begin{abstract}
We consider a regularization problem whose objective function consists of a convex fidelity term and a regularization term determined by the $\ell_1$ norm composed with a linear transform. Empirical results show that the regularization with the $\ell_1$ norm can promote sparsity of a regularized solution. It is the goal of this paper to understand theoretically the effect of the regularization parameter on the sparsity of the regularized solutions. We establish a characterization of the sparsity under the transform matrix of the solution. When the fidelity term has a special structure and the transform matrix coincides with a identity matrix, the resulting characterization can be taken as a regularization parameter choice strategy with which the regularization problem has a solution having a sparsity of a certain level. We study choices of the regularization parameter so that the regularization term alleviates the ill-posedness and promote sparsity of the resulting regularized solution. Numerical experiments demonstrate that choices of the regularization parameters can balance the sparsity of the solutions of the regularization problem and its approximation to the minimizer of the fidelity function.
\end{abstract}

\section{Introduction}
Many practical problems may be modeled as learning a function from a finite number of observed data points. Learning a function from a finite number of observed data is an ill-posed problem. Such a problem cannot be solved directly as its solution is strongly sensitive to input data which are inevitably corrupted with noise. The ill-posedness was treated by the classical Tikhonov regularization which adds a regularization term to a data fidelity term constructed from the original ill-posed problem such that the resulting optimization problem is much less sensitive to disturbances. The added regularization term composes of a Hilbert space norm of the solution and a positive regularization parameter $\lambda$ which balances the noise suppression and the approximation error of the regularized solution. An estimate for the classical Tikhonov regularization expresses the regularization error in terms of a sum of the two terms: the approximation error proportional to $\lambda$  plus the error (caused by noise) proportional to the reciprocal of $\lambda$. The parameter $\lambda$ is then chosen to minimize the regularization error. For choices of the optimal regularization parameter, the readers are referred to  \cite{bauer2011comparingparameter, pereverzev2005adaptive, tikhonov1977solutions}.

Motivated by the big data nature of recent practical applications, sparse regularization in Banach spaces has attracted much attention in various fields, since a sparse representation for a learned function is essential to ease the computational burden for operations of the function as the amount of data increases. As a popular approach to achieve this, regularization in a Banach space with a sparsity promoting norm, such as the $\ell_1$ norm, is widely used in statistics, machine learning, signal processing, image processing and medical imaging. In statistics, the Lasso and its extensions \cite{ali2019generalized,tibshirani1996regression,tibshirani2005sparsity,tibshirani2011solution} apply an $\ell_1$ penalty to linear regression. The Lasso is also known in signal processing as basis pursuit \cite{chen2001atomic} which aims at decomposing a signal into an optimal superposition of dictionary elements in the sense that the resulting representation has the smallest $\ell_1$ norm of coefficients among all such decompositions. Image restoration using TV norms for regularization \cite{li2015multi,micchelli2011proximity,rudin1992nonlinear} leads to searching an optimization solution in the Euclidean space with the $\ell_1$ norm. Sparse learning models with the $\ell_1$ norm, such as $\ell_1$ SVM classification \cite{Li-Song-Xu2019, scholkopf2002learning, smola2004tutorial} and $\ell_1$ SVM regression \cite{bi2003dimensionality, Li-Song-Xu2018, smola2004tutorial}, have received increasing attention in machine learning. Motivated by the need of sparse learning algorithms, the notion of reproducing kernel Banach spaces (RKBSs) was introduced in \cite{zhang2009reproducing} and further developed in \cite{li2018sparse, lin2021multi, song2011reproducing, song2013reproducing, xu2019generalized}.  RKBSs with the $\ell_1$ norm \cite{lin2021multi, song2011reproducing, song2013reproducing} have been proven successful in promoting sparsity in representations for learned functions.

There were two crucial issues related to the choice of the regularization parameter in a regularization problem in a Banach space. The first one involves the error analysis to which considerable amount of work (for example,  \cite{grasmair2008sparse,grasmair2011necessary,lorenz2008convergence, schuster2012regularization}) has been devoted. In particular, for the regularization problem with a special fidelity term and the $\ell_p$ norm regularizer, a convergence rate of the regularized solutions has been derived in \cite{grasmair2008sparse,grasmair2011necessary,lorenz2008convergence} according to a noise level and a choice
the regularization parameter. The second issue concerns how a choice of the regularization parameter balances the sparsity of the regularized solution and its approximation accuracy. Empirical results \cite{Li-Song-Xu2018,song2013reproducing,rudin1992nonlinear,tibshirani1996regression,yuan2006model} showed that one can obtain a solution having sparsity of certain level under a given transform of the regularization problem by choosing appropriate regularization parameter. There also
exist some theoretical results \cite{bach2011optimization,koh2007interior,shi2011concentration, tibshirani2012degrees,zou2007degrees} for choices of the parameters in some special cases.
For several specific application models, there were attempts to understand how one can choose the regularization parameter so that the resulting learned function has sparsity of certain levels. For example, sparsity of the solution of the Lasso regularized model was studied in \cite{bach2011optimization}, where the relation between the sparsity of the regularized solution and the regularization parameter was characterized. Recent studies on the degrees of freedom of the Lasso regularized model \cite{tibshirani2012degrees,zou2007degrees} provided an objectively guided
choice of the regularization parameter in such a regularization problem through Stein’s unbiased risk estimation (SURE) framework.  For the $\ell_1$ regularized Logistic regression problem, the regularization parameter is given in \cite{koh2007interior} to ensure that the regularized solution has all components zero. 
A sufficient condition for vanishing of a coefficient in a solution representation of a regularized learning method with an $\ell_1$ regularizer was presented in \cite{shi2011concentration}. These theoretical results on the regularization parameter all depend on the learned solution and no practical choice strategy of the parameter was provided in these studies.

It remains to be understood from the theoretical viewpoint for a regularized learning problem with a general convex fidelity term how the choice of the regularization parameter balances sparsity of the learned solution and its approximation error. 
The aim of this paper is to reveal theoretically how the choice of the regularization  parameter can alleviate the ill-posedness and promoting sparsity of a regularized solution. To this end, we need to first study the relation between the choice of the regularization parameter and the sparsity of the regularized solutions in a Banach space with the $\ell_1$ norm.
This issue has been considered in \cite{xu2021} for the case when the regularization term is the $\ell_0$ ``norm'' composed with a linear transform. Since the regularization problem with the $\ell_0$ ``norm'' has nice geometric interpretation even though it is non-convex, it leads to a geometric approach to understand the issue. Since the $\ell_1$ norm regularization problem has less clear geometric meaning, the geometric approach introduced in \cite{xu2021} does not seem to be applicable directly. However, it provides us with useful insights of sparse solutions. Due to the convexity of the $\ell_1$ norm, we instead approach this problem by appealing to tools available in convex analysis.

In the regularization problem to be studied in this paper, the objective function consists of a convex fidelity term and a regularization term determined by the $\ell_1$ norm composed with a transform matrix. We first 
discuss the choices of the regularization parameters when the transform matrix that appears in the regularization term reduces to the identity matrix. We have paid special attention to the cases when fidelity terms have special structures such as additive separability or block separability. In such cases, we have established a complete characterization of the sparsity of the solution, which show how we can choose the regularization parameter so that the solution has certain levels of sparsity. For the case that the fidelity term is a general convex function, we also give a sparsity characterization 
of the solution. Although in this characterization the regularization parameter depends on the solution,  we still observe from it how the choice of the 
regularization parameter influences the sparsity of the solution. We then consider the case when the transform matrix is not the identity and has an arbitrary rank. 
In such a case, by making use of the singular value decomposition of the transform matrix, we transform the original minimization problem to an equivalent constrained optimization problem having a simple transform matrix which is a two block diagonal matrix with the diagonal blocks being an identity and a zero matrix. 
The equivalent constrained optimization problem is further reformulated as an unconstrained minimization problem by employing the indicator function of the 
constraint set. In this manner, we obtain a characterization of the sparsity under the transform of the regularized solution. Results obtained in this paper are 
applied to several practical examples. Moreover, we conduct numerical experiments, to test the obtained theoretical results, which show that the parameter choices provided by this study can balance the sparsity of the solution of the regularization problem and its approximation accuracy. 

Choosing a regularization parameter to balance the sparsity of the regularized solution and its approximation accuracy in a Banach space setting is a challenging issue. Unlike the counterpart in a Hilbert space setting where the parameter was chosen to balance two error terms 
(the approximation error and the error caused by noise) 
which have the same base quantity (dimension) \cite{Chen2015,Chen-Lu-Xu-Yang2008,Chen-Micchelli-Xu2015,Egger2018,Tautenhahn1996,Tautenhahn1998},  the sparsity measure and the accuracy measure in a Banach space setting are not the same base quantity. This raises difficulties in balancing them from a theoretical standpoint. We attempt in this paper to understand theoretically the effect of the regularization parameter to the sparsity of the regularized solution and as well as to approximation errors caused by noise. We demonstrate our idea by considering the Lasso regularized model. An error estimate for a solution of the model can be obtained by a general argument established in \cite{grasmair2008sparse,grasmair2011necessary}. By combining the sparsity characterization of the regularized solution and the error estimate, we obtain a choice strategy of the regularization parameter that yields a sparse regularized solution with an error bound.  

We organize this paper in seven sections. In section 2, we describe the regularization problem to be considered and review several practical examples of importance. We characterize in section 3 the relation between the regularization parameter and levels of sparsity of the regularized solution in the case that the transform matrix is the identity. The resulting characterizations provide regularization parameter choice strategies ensuring that the regularized solution with this parameter has sparsity of a desired level. Section 4 is devoted to studying choices of the regularization parameter that guarantee desired levels of sparsity under a transform of the regularized solution. In section 5, we discuss how the regularization parameter $\lambda$ can be chosen to alleviate the ill-posedness and promoting sparsity of the regularized solutions by considering a Lasso regularized model.  In section 6, we present numerical experiments to demonstrate the regularization parameter choice strategy established in this paper. Finally in section 7 we make conclusive remarks.


\section{Regularization with the $\ell_{1}$ norm}
In this section, we describe the regularization problem, whose objective function consists of a convex fidelity term and the $\ell_{1}$ norm regularization term, to be considered in this paper, and identify several 
optimization models of practical importance which can be formulated in this general form.

We begin with describing the regularization problem under investigation. For each $d\in\mathbb{N}$, let $\mathbb{N}_d:=\{1,2,\ldots,d\}$ and set $\mathbb{N}_0:=\emptyset.$ For $\mathbf{x}:=[x_j:j\in\mathbb{N}_d]\in\mathbb{R}^d$, we define its $\ell_1$ norm by $\|\mathbf{x}\|_1:=\sum_{j\in \mathbb{N}_d}|x_j|$. For $m,n\in\mathbb{N}$, suppose that $\bm{\psi}:\mathbb{R}^n\to\mathbb{R}_+:=[0,+\infty)$ is a convex function and $\mathbf{B}$ is an $m\times n$ real matrix. We consider the regularization problem
\begin{equation}\label{optimization_problem_nm}
\min
\left\{\bm{\psi}(\mathbf{u})+\lambda\|\mathbf{B}\mathbf{u}\|_{1}:\mathbf{u}\in\mathbb{R}^n\right\},
\end{equation}
where $\lambda$ is a positive regularization parameter. The regularization problem \eqref{optimization_problem_nm} covers many application problems. We present several examples below.

The Lasso regularized model \cite{tibshirani1996regression}  is a special case of the regularization problem \eqref{optimization_problem_nm}. Specifically, let $p\in\mathbb{N}$ and $\|\cdot\|_2$ denote the standard Euclidean norm on $\mathbb{R}^p$. Suppose that $\mathbf{x}\in\mathbb{R}^p$ is a response vector and $\mathbf{A}\in\mathbb{R}^{p\times n}$ is a predictor matrix. When the fidelity  term  $\bm{\psi}$ is chosen as
\begin{equation}\label{square_psi_u}
\bm{\psi}(\mathbf{u}):=\frac{1}{2}\|\mathbf{Au}-\mathbf{x}\|_2^2, \ \ \mathbf{u}\in\mathbb{R}^{n}, 
\end{equation} 
and the transform matrix $\mathbf{B}$ as the identity matrix $\mathbf{I}_n$ of order $n$, the regularization problem \eqref{optimization_problem_nm} reduces to the Lasso regularized model
\begin{equation}\label{lasso}
\min \left\{\frac{1}{2}\|\mathbf{Au}-\mathbf{x}\|_2^2+\lambda\|\mathbf{u}\|_1:\mathbf{u}\in\mathbb{R}^n\right\}.
\end{equation}
The generalized Lasso regularized model \cite{ali2019generalized,tibshirani2011solution} enforces certain structural constraints on the coefficients in a linear regression by defining the regularization term employing the $\ell_{1}$ norm composed with a transform matrix. That is, letting $\mathbf{B}$ be an $m\times n$ real matrix, the generalized Lasso regularized model has the form 
\begin{equation}\label{generalized_lasso}
\min \left\{\frac{1}{2}\|\mathbf{Au}-\mathbf{x}\|_2^2+\lambda\|\mathbf{Bu}\|_1:\mathbf{u}\in\mathbb{R}^n\right\}.
\end{equation}
The regularization problem \eqref{generalized_lasso} covers many important areas where  different choices of matrices $\mathbf{A}$ and $\mathbf{B}$ are taken. 

In signal or image denoising processes, the two positive integers $p$ and $n$ are equal and the matrix $\mathbf{A}$ is chosen as the identity matrix of order $n$. The transform matrix $\mathbf{B}$ is often chosen to reflect some believed structure or geometry in the signal or the image. For example, if the transform matrix $\mathbf{B}$ is chosen as the $(n-1)\times n$ first order difference matrix $\mathbf{D}^{(1)}:=[d_{ij}:i\in\mathbb{N}_{n-1}, j\in\mathbb{N}_{n}]$ with $d_{ii}=-1$, $d_{i,i+1}=1$ for $i\in\mathbb{N}_{n-1}$ and $0$ otherwise,
then the generalized Lasso regularized model \eqref{generalized_lasso} describes the one-dimensional  fused Lasso model \cite{tibshirani2005sparsity}, which is also called the one-dimensional total-variation denoising model \cite{condat2013direct}. 
If $\mathbf{B}$ is chosen as the two-dimensional difference matrix giving both the horizontal and vertical differences between pixels, then the generalized Lasso regularized model \eqref{generalized_lasso} coincides with the two-dimensional fused Lasso model \cite{tibshirani2005sparsity} or the ROF total-variation denoising model \cite{li2015multi,micchelli2011proximity,rudin1992nonlinear}. 

Another example that concerns the polynomial trend filtering is described below. For each $k\in\mathbb{N}$, let $\mathbf{D}^{(1,k)}$ denote the $(n-k-1)\times (n-k)$ first order difference matrix. Accordingly, the difference matrix of $k+1$ order is defined recursively by $\mathbf{D}^{(k+1)}
:=\mathbf{D}^{(1,k)}\mathbf{D}^{(k)}$, $k\in\mathbb{N}$.
The polynomial trend filtering of order $k$ has the form \eqref{generalized_lasso} with the matrix $\mathbf{A}:=\mathbf{I}_n$  and the transform matrix $\mathbf{B}:=\mathbf{D}^{(k+1)}$. In the special case that $k=1$, the generalized Lasso regularized model \eqref{generalized_lasso} reduces to the linear trend filtering \cite{kim2009ell_1,wang2016trend}. The transform matrix $\mathbf{B}$ may also be chosen as a discrete wavelet transform \cite{Chang2000adaptive,Daubechies1992,donoho1995adapting,micchelli1994using,micchelli1997,Tomassi2015}, a framelet transform \cite{Li-Shen2011,lian2011},
a discrete cosine transform \cite{Jafarpour2009,strang1999discrete} or a discrete Fourier transform \cite{Gasquet1999,McCann2020,Yang2011}, depending on specific applications in signal or imaging processing. The resulting regularized model aims at representing a signal or an image as a sparse linear combination of certain basis functions.

Data in many applications often carry a group structure where they are partitioned into disjoint pieces.  Structured sparsity approaches recently received considerable attention in statistics, machine learning and signal processing. A natural extension of the Lasso regularized model \eqref{lasso} is the group Lasso regularized model \cite{Boyd2011, Buhlmann2011,Jenatton2011,yuan2006model}. We now briefly review this model. For $d,n\in\mathbb{N}$ with $d\leq n$, we suppose that $\mathcal{S}:=\left\{S_{1},S_{2},\ldots, S_{d}\right\}$ is a partition of the index set $\mathbb{N}_n$ which satisfy $S_j\neq\emptyset$, for all $j\in\mathbb{N}_d$, $S_j\cap S_k=\emptyset$ if $j\neq k$, and $\cup_{j\in\mathbb{N}_d}S_j=\mathbb{N}_n$. For each $j\in \mathbb{N}_d$ we denote by $n_j$ the cardinality of $S_j$ and regard $S_j$ as an {\it ordered set} in the natural order of elements in $\mathbb{N}_n$. Specifically, we assume $S_j:=\{i(j)_1, i(j)_2, \dots, i(j)_{n_j}\}$, with $i(j)_l\in \mathbb{N}_n$, $l\in \mathbb{N}_{n_j}$ and $i(j)_1<i(j)_2<\dots<i(j)_{n_j}$. Associated with the partition $\mathcal{S}$, we decompose a vector $\mathbf{u}:=[u_k:k\in\mathbb{N}_n]\in\mathbb{R}^n$
into $d$ sub-vectors by setting $\mathbf{u}_j:=[u_{{i(j)}_1},u_{{i(j)}_2}, \dots, u_{{i(j)}_{n_j}}]\in\mathbb{R}^{n_j}$ for each $j\in\mathbb{N}_d$. With this notation, the group Lasso regularized model can be described as  
\begin{equation}\label{group_lasso}
\min\left\{\frac{1}{2}\left\|\mathbf{A}\mathbf{u}-\mathbf{x}\right\|_2^2+\lambda\sum_{j\in\mathbb{N}_d}\sqrt{n_j}\|\mathbf{u}_j\|_2:\mathbf{u}\in\mathbb{R}^{n}\right\}.
\end{equation}
The regularizer in the above regularization problem could be viewed as a group-wise $\ell_1$ norm. If the partition $\mathcal{S}:=\{S_1,S_2,\ldots,S_n\}$ is chosen as the nature partition defined by $S_j:=\{j\}$ for all $j\in\mathbb{N}_n$, then the group Lasso regularized model \eqref{group_lasso} reduces to Lasso regularized model \eqref{lasso}. It is known \cite{Juditsky2012,yuan2006model} that the model \eqref{group_lasso} performs better than the lasso regularized model when the optimal variable has the group structure. 

Support vector machines (SVMs) for both classification and regression with the $\ell_1$ norm can be reformulated in the form \eqref{optimization_problem_nm}. We first present the $\ell_1$ SVM
classification model  \cite{scholkopf2002learning,smola2004tutorial}. Given training data $D:=\{(\mathbf{x}_j,y_j): j\in\mathbb{N}_n\}$ composed of input data points $X:=\{\mathbf{x}_j:j\in\mathbb{N}_n\}\subset\mathbb{R}^d$ and output data values $Y:=\{y_j:j\in\mathbb{N}_n\}\subset\{1,-1\}$. A hyperplane determined by $\bm{\alpha}\in\mathbb{R}^n$ and $b\in\mathbb{R}$ is wished to separate the data $D$ into two groups for $y_j=1$ and $y_j=-1$. By introducing a loss function  ${L}_{D}:\mathbb{R}^n\times\mathbb{R}\to \mathbb{R}_{+}$, the parameters $\bm{\alpha},b$ are obtained by the $\ell_1$ SVM classification model
\begin{equation}\label{l1_SVM}
\min
\left\{{L}_{D}(\bm{\alpha},b)+\lambda\|\bm{\alpha}\|_1:\bm{\alpha}\in\mathbb{R}^n, b\in\mathbb{R}\right\}.
\end{equation}
Let $K:\mathbb{R}^d\times\mathbb{R}^d\to\mathbb{R}$ be a given reproducing kernel.  
A commonly used loss function ${L}_{D}$ in the $\ell_1$ SVM classification model \eqref{l1_SVM} is the hinge loss function defined by
\begin{equation}\label{Loss_LD}
{L}_{D}(\bm{\alpha},b)
:=\sum_{j\in\mathbb{N}_n}\mathrm{max}\left\{1-y_j\left(\sum_{k\in\mathbb{N}_n}
\alpha_{k}K(\mathbf{x}_k,\mathbf{x}_j)+b\right),0\right\}, \end{equation} 
for all $\bm{\alpha}:=[\alpha_j:j\in\mathbb{N}_n]\in\mathbb{R}^n$ and all $b\in\mathbb{R}$. 
We next rewrite the $\ell_1$ SVM classification model \eqref{l1_SVM} with the hinge loss function \eqref{Loss_LD} in the form  \eqref{optimization_problem_nm}.
Throughout this paper, we write $\llbracket\mathbf{a},\mathbf{b}\rrbracket$ to denote the vector
$\begin{bmatrix}
\mathbf{a}\\
\mathbf{b}
\end{bmatrix}\in\mathbb{R}^{d_1+d_2}$ for all $\mathbf{a}\in\mathbb{R}^{d_1}$ and $\mathbf{b}\in\mathbb{R}^{d_2}$.
Associated with $\bm{\alpha}\in\mathbb{R}^n$ and $b\in\mathbb{R}$, we introduce the vector 
$\mathbf{u}:=\llbracket\bm{\alpha},b\rrbracket\in \mathbb{R}^{n+1}.
$
We define the kernel matrix 
$\mathbf{K}:=[K(\mathbf{x}_j,\mathbf{x}_k):j,k\in\mathbb{N}_n]$ 
and augment it to $\mathbf{K}':=[\mathbf{K} \ \mathbf{1}_n]$ with $
\mathbf{1}_n:=[1,1,\ldots,1]^\top\in\mathbb{R}^n$.
We also define the diagonal matrix $\mathbf{Y}:=\mathrm{diag}(y_j: j\in\mathbb{N}_n)$ and the function 
\begin{equation}\label{sum of hinge loss}    \bm{\phi}(\mathbf{z}):=\sum_{j\in\mathbb{N}_n}\mathrm{max}\{1-z_j,0\},\ \mbox{for all}\ \mathbf{z}:=[z_j:j\in\mathbb{N}_n]\in\mathbb{R}^n.
\end{equation}
Then by introducing the fidelity term by 
$\bm{\psi}(\mathbf{u}):=\bm{\phi}(\mathbf{Y}\mathbf{K}'\mathbf{u})$, $\mathbf{u}\in\mathbb{R}^{n+1}$ and choose 
$\mathbf{B}:=[\mathbf{I}_n \ \mathbf{0}]\in\mathbb{R}^{n\times(n+1)}$, the $\ell_1$ SVM classification model \eqref{l1_SVM} can be rewritten in the form of \eqref{optimization_problem_nm}.

Another popular choice of the loss function ${L}_{D}$ in the $\ell_1$ SVM classification model \eqref{l1_SVM} is the squared loss function defined by
\begin{equation}\label{square_loss}
{L}_{D}(\bm{\alpha},b)
:=\frac{1}{2}\sum_{j\in\mathbb{N}_n}\left(\sum_{k\in\mathbb{N}_n}
\alpha_{k}K(\mathbf{x}_k,\mathbf{x}_j)+b-y_j\right)^2.
\end{equation}
By setting $\mathbf{y}:=[y_j:j\in\mathbb{N}_n]$, the $\ell_1$ SVM classification model \eqref{l1_SVM} with the squared loss function \eqref{square_loss} can be identified as the form \eqref{optimization_problem_nm} with 
$\bm{\psi}(\mathbf{u}):=\frac{1}{2}\|\mathbf{K}'\mathbf{u}-\mathbf{y}\|_2^2$, $ \mathbf{u}\in\mathbb{R}^{n+1}$, and $\mathbf{B}:=[\mathbf{I}_n \ \mathbf{0}]\in\mathbb{R}^{n\times(n+1)}.$ We note that this model is also the generalized Lasso model \eqref{generalized_lasso} with $\mathbf{x}:=\mathbf{y}$, $\mathbf{A}:=\mathbf{K}'$ and $\mathbf{B}:=[\mathbf{I}_n \ \mathbf{0}].$

When $L_{D}$ in the $\ell_1$ SVM classification model \eqref{l1_SVM} is chosen as the average logistic loss function defined for $\bm{\alpha}\in\mathbb{R}^d$ and $b\in\mathbb{R}$ by 
$L_{D}(\bm{\alpha},b):=\frac{1}{n}\sum_{j\in\mathbb{N}_n}
\mathrm{ln}\left(1+\mathrm{exp}
\left(-y_j\left(\bm{\alpha}^{\top}\mathbf{x}_j+b\right)\right)\right)$,
it is the $\ell_1$ regularized logistic regression model. It can be written in the form  \eqref{optimization_problem_nm} with the fidelity term 
$\bm{\psi}(\mathbf{u}):=\bm{\phi}(\mathbf{Y}\mathbf{X}'\mathbf{u})$, $ \mathbf{u}:=\llbracket\bm{\alpha},b\rrbracket\in\mathbb{R}^{d+1}$, and matrix $\mathbf{B}:=[\mathbf{I}_d\ \mathbf{0}]\in\mathbb{R}^{d\times(d+1)}$,
where $\mathbf{X}:=[\mathbf{x}_j:j\in \mathbb{N}_{n}]^{\top}$, $\mathbf{X}':=[\mathbf{X}\ \mathbf{1}_{n}]$ and
\begin{equation}\label{logistic_regression}
\bm{\phi}(\mathbf{z})
:=\frac{1}{n}\sum_{j\in\mathbb{N}_n}
\mathrm{ln}(1+\mathrm{exp}(-z_j)),
\ \mbox{for all}\ \mathbf{z}:=[z_j:j\in\mathbb{N}_{n}]
\in\mathbb{R}^{n}.
\end{equation}

We now turn to describing the $\ell_1$ SVM regression model \cite{bi2003dimensionality,smola2004tutorial} which aims at learning a function from the observed data $D:=\{(\mathbf{x}_j,y_j):j\in\mathbb{N}_n\}\subset\mathbb{R}^d\times\mathbb{R}$. Specifically, the $\ell_1$ SVM regression model has the same form as for the classification
model \eqref{l1_SVM}, with a different loss function ${L}_{D}$. A popular choice of  ${L}_{D}$ is the $\epsilon$-insensitive loss function \cite{Vladimir1998} in the form 
\begin{equation}\label{epsilon-insensitive}
{L}_{D}(\bm{\alpha},b)
:=\sum\limits_{j\in\mathbb{N}_n}
\mathrm{max}\left\{\left|\sum_{k\in\mathbb{N}_n}
\alpha_{k}K(\mathbf{x}_k,\mathbf{x}_j)+b-y_j\right|-\epsilon,0\right\},
\end{equation}
where $\epsilon$ is a positive parameter and $K$ is a given reproducing kernel on $\mathbb{R}^d$. 
We may rewrite this $\ell_1$ SVM regression model in the form of \eqref{optimization_problem_nm}. To this end, we define the vector
$\mathbf{u}$, the kernel matrix $\mathbf{K}$ and its augmentation matrix $\mathbf{K}'$ as in the classification model. Associated with the output data values  $\mathbf{y}:=[y_j:j\in\mathbb{N}_n]\in\mathbb{R}^n$ and the parameter $\epsilon>0$, we introduce the function $\bm{\phi}_{\mathbf{y},\epsilon}$ on $\mathbb{R}^n$ by
\begin{equation}\label{sum of epsilon-insensitive loss}
\bm{\phi}_{\mathbf{y},\epsilon}(\mathbf{z}):=\sum\limits_{j\in\mathbb{N}_n}
\mathrm{max}\{|z_j-y_j|-\epsilon,0\},\ \mbox{for all}\ \mathbf{z}:=[z_j:j\in\mathbb{N}_n]\in\mathbb{R}^n.
\end{equation}
In this notation, the $\ell_1$ SVM regression model with the $\epsilon$-insensitive loss function \eqref{epsilon-insensitive} can be rewritten in the form \eqref{optimization_problem_nm} with the fidelity term 
$\bm{\psi}(\mathbf{u}):=\bm{\phi}_{\mathbf{y},\epsilon}(\mathbf{K}'\mathbf{u})$, $ \mathbf{u}\in\mathbb{R}^{n+1}$, and matrix  $\mathbf{B}:=[\mathbf{I}_n \ \mathbf{0}]\in\mathbb{R}^{n\times(n+1)}$.
The squared loss function ${L}_{D}$  defined by \eqref{square_loss} with $\mathbf{y}:=[y_j:j\in\mathbb{N}_n]\in\mathbb{R}^n$ is often used in the $\ell_1$ SVM regression model. In this case, the $\ell_1$ SVM regression model is equivalent to the  regularization problem  \eqref{optimization_problem_nm}
composed of the fidelity term $\bm{\psi}$ defined by $\bm{\psi}(\mathbf{u}):=\frac{1}{2}\|\mathbf{K}'\mathbf{u}-\mathbf{y}\|_2^2$, $ \mathbf{u}\in\mathbb{R}^{n+1}$, with $\mathbf{y}:=[y_j:j\in\mathbb{N}_n]\in\mathbb{R}^n$ and the transform matrix $\mathbf{B}:=[\mathbf{I}_n \ \mathbf{0}]\in\mathbb{R}^{n\times(n+1)}$.

The regularization problem
\eqref{optimization_problem_nm} also appears in regularized learning in RKBSs. In such spaces, the regularized learning problem is usually an infinite dimensional optimization problem. The remarkable representer theorem  \cite{cox1990asymptotic,kimeldorf1970correspondence,scholkopf2001generalized,unser2021unifying,wang2021representer} reduces the solutions to finding coefficients of a finite number of elements in the space. In particular, the regularized learning model in the RKBS with the $\ell_1$ norm \cite{song2011reproducing,song2013reproducing} can be formulated in the form \eqref{optimization_problem_nm}. Specifically, suppose that $\{(\mathbf{x}_j,y_j):j\in \mathbb{N}_n\}\subset\mathbb{R}^d\times \mathbb{R}$ are given with $\mathbf{y}:=[y_j:j\in\mathbb{N}_n]$, $K:\mathbb{R}^d\times\mathbb{R}^d\to\mathbb{R}$ is a given reproducing kernel and $\mathbf{K}:=[K(\mathbf{x}_j,\mathbf{x}_k):j,k\in\mathbb{N}_n]$ is the resulting kernel matrix. The regularized learning model in the RKBS with the $\ell_1$ norm has the form of \eqref{optimization_problem_nm} with $\bm{\psi}(\mathbf{u}):=\|\mathbf{Ku}-\mathbf{y}\|_2^2$ and $\mathbf{B}:=\mathbf{I}_n$.

\section{Parameter Choices for Sparsity of the Regularized Solutions}

In this section and the next one, we discuss choices of the regularization parameter so that a solution of the resulting regularization  problem \eqref{optimization_problem_nm} has sparsity of a prescribed level. In this section we first consider the special case when $m=n$ and  $\mathbf{B}:=\mathbf{I}_n$. In this case, the regularization problem \eqref{optimization_problem_nm} has the special form
\begin{equation}\label{optimization_problem}
\min
\left\{\bm{\psi}(\mathbf{u})+\lambda\|\mathbf{u}\|_{1}:\mathbf{u}\in\mathbb{R}^n\right\}.
\end{equation}
We postpone the general case to the next section.

As a preparation, we recall the definition of the level of sparsity for a vector in $\mathbb{R}^n$. For each $n\in\mathbb{N}$, we set $\mathbb{Z}_n:=\{0,1,\ldots,n-1\}$. A vector $\mathbf{x}\in\mathbb{R}^n$ is said to have sparsity of level $l\in\mathbb{Z}_{n+1}$ if it has exactly $l$ nonzero components. To further characterize sparsity of vectors in $\mathbb{R}^n$, we make use of the sparsity partition of $\mathbb{R}^n$, introduced initially in \cite{xu2021}. For each $j\in\mathbb{N}_n$, we denote by $\mathbf{e}_j$ the unit vector with $1$ for the $j$-th component and 0 otherwise. The vectors $\mathbf{e}_j,j\in\mathbb{N}_n,$ form the canonical basis for $\mathbb{R}^n$. Using these vectors, we define $n+1$ numbers of subsets of $\mathbb{R}^n$ by
\begin{align}\label{Partition_A}
&\Omega_0:=\{\mathbf{0}\in\mathbb{R}^n\}, \nonumber\\
&\Omega_l:=\left\{\sum_{j\in\mathbb{N}_l}u_{k_j}\mathbf{e}_{k_j}
:u_{k_j}\in\mathbb{R}\setminus{\{0\}},\ \mathrm{for} \ 
1\leq k_1<k_2<\cdots< k_l\leq n\right\},
\ \mathrm{for} \ l\in\mathbb{N}_n.
\end{align}
It was shown in  \cite{xu2021} that the sets $\Omega_l,l\in \mathbb{Z}_{n+1}$, defined by \eqref{Partition_A}, are mutually disjoint and form a partition for $\mathbb{R}^n$, that is,
$\mathbb{R}^{n}=\bigcup_{l\in\mathbb{Z}_{n+1}}\Omega_l.$
Note that for each $l\in\mathbb{Z}_{n+1}$, $\Omega_l$ is the set of all vectors in $\mathbb{R}^n$ having sparsity of level $l$. The goal of this study is to relate the choice of the regularization parameter $\lambda$ with the set $\Omega_l$ to which a solution $\mathbf{u}$ of the regularization problem \eqref{optimization_problem} belongs.

We will employ the notion of the subdifferential of a convex function on $\mathbb{R}^n$ for this study. The subdifferential of a real-valued convex function $f: \mathbb{R}^n\to \mathbb{R}$ at $\mathbf{x}\in\mathbb{R}^n$ is defined by
\begin{equation*}\label{subgradients}
\partial f(\mathbf{x}):=\{\mathbf{y}\in\mathbb{R}^n: \ f(\mathbf{z})\geq f(\mathbf{x})+\langle \mathbf{y},\mathbf{z}-\mathbf{x}\rangle, \ \mathrm{for} \ \mathrm{all} \ \mathbf{z}\in\mathbb{R}^n\}.
\end{equation*}
Elements in $\partial f(\mathbf{x})$ are called subgradients. Suppose that $f$ and $g$ are two real-valued convex functions on $\mathbb{R}^n$. It is known \cite{zalinescu2002convex} that if $g$ is continuous on $\mathbb{R}^n$ then there holds 
$\partial (f+g)(\mathbf{x})
=\partial f(\mathbf{x}) +\partial g(\mathbf{x})$, for all $\mathbf{x}\in\mathbb{R}^n$.

We are now ready to characterize the sparsity of a solution of the regularization problem \eqref{optimization_problem}. We start with the case that the fidelity term $\bm{\psi}$ involved in \eqref{optimization_problem} is additively separable. That is, there exist $n$ univariate functions $\psi_{j}$, $j\in\mathbb{N}_{n}$, on $\mathbb{R}$ such that 
\begin{equation}\label{separable_psi}
\bm{\psi}(\mathbf{u}):=\sum_{j\in\mathbb{N}_n}\psi_{j}(u_j),\ 
\mbox{for all}\ \mathbf{u}:=[u_j:j\in\mathbb{N}_{n}]\in\mathbb{R}^n.
\end{equation}
It is easy to see that the $\ell_1$ norm $\|\cdot\|_1$ is an additively separable function. Due to the additive separability of the objective function, we observe that $\mathbf{u}^{*}:=[u^*_j:j\in\mathbb{N}_n]$ is a solution of the regularization problem \eqref{optimization_problem} if and only if  for each $j\in\mathbb{N}_{n}$, $u_j^*$ is a solution of the regularization problem
$\min\{\psi_{j}(u_j)+\lambda|u_j|:u_j\in\mathbb{R}\}.$
Based upon this observation, it suffices to consider the  regularization problem 
\begin{equation}\label{optimization_1m}
\min
\left\{\psi(u)+\lambda|u|:u\in\mathbb{R}\right\},
\end{equation}
with $\psi$ being a convex function on $\mathbb{R}$, and discuss how we choose an appropriate regularization parameter so that the resulting regularization problem has a zero solution. 

We give a lemma which concerns a parameter choice for the regularization problem \eqref{optimization_1m}. It is known \cite{zalinescu2002convex} that for a convex function $\psi:\mathbb{R}\rightarrow\mathbb{R}$, both of its left derivative $\psi'_{-}$ and its right derivative $\psi'_{+}$ exist at any $u\in\mathbb{R}$. Moreover,
\begin{equation}\label{psi_subbdiff}
\partial\psi(u)=[\psi'_{-}(u),\psi'_{+}(u)], \ \ \mbox{for all}\ \ u\in\mathbb{R}.
\end{equation}

\begin{lemma}\label{choice_sparsity_1dimension_nonsmooth}
Suppose that $\psi$ is a convex function on $\mathbb{R}$. Then the regularization problem \eqref{optimization_1m} with $\lambda>0$ has a solution $u^{*}=0$ if and only if
$\lambda \geq \mathcal{\max}\{\psi'_{-}(0),-\psi'_{+}(0)\}.$
\end{lemma}
\begin{proof}
By employing the Fermat rule \cite{zalinescu2002convex} and the continuity of the function $|\cdot|$, we have that $u^{*}=0$ is a solution of the regularization problem \eqref{optimization_1m} if and only if 
\begin{equation}\label{0_subgradient_1m_nonsmooth}
{0}\in \partial\psi(0) + \lambda \partial|\cdot|(0).
\end{equation}
It follows from the definition of the subdifferential that 
\begin{equation}\label{abs_subbiff}
\partial|\cdot|(0)=
\{y\in\mathbb{R}:|y|\leq 1\}.
\end{equation}
By equation  \eqref{psi_subbdiff} with $u$ being replaced by $0$ and equation \eqref{abs_subbiff}, the inclusion relation \eqref{0_subgradient_1m_nonsmooth}
has the equivalent form 
$0\in[\psi'_{-}(0),\psi'_{+}(0)]+[-\lambda, \lambda],$
which is further equivalent to 
$0\in[\psi'_{-}(0)-\lambda,\psi'_{+}(0)+\lambda].$
The latter is equivalent to
$\psi'_{-}(0)\leq\lambda$  and $-\psi'_{+}(0)\leq\lambda$.
Consequently, we obtain the desired result of this lemma.
\end{proof}

Lemma \ref{choice_sparsity_1dimension_nonsmooth} allows us to obtain a choice of the parameter $\lambda$ with which the regularization problem \eqref{optimization_problem}, with the fidelity term $\bm{\psi}$ being in the form \eqref{separable_psi}, has a solution with sparsity of a general level. For each $j\in\mathbb{N}_n$, let $\psi'_{j,-}$ and $\psi'_{j,+}$ denote the left and right derivatives of $\psi_j$, respectively.

\begin{theorem}\label{choice_sparsity_separable_nonsmooth}
Suppose that $\psi_j$, $j\in\mathbb{N}_n$, are  convex functions on $\mathbb{R}$ and $\bm{\psi}$ has the form \eqref{separable_psi}. Then the regularization problem \eqref{optimization_problem} with $\lambda>0$ has a solution having sparsity of level $l'$ with $l'\leq l$ for some $l\in\mathbb{Z}_{n+1}$ if and only if there exist distinct 
$k_i \in \mathbb{N}_n$, $i\in\mathbb{N}_l$, such that 
\begin{equation}\label{lambda_separable_nonsmooth}
\lambda\geq\mathcal{\max}\left\{\psi'_{j,-}(0),-\psi'_{j,+}(0)\right\},\ \ \mbox{for all}\ \ j\in\mathbb{N}_n\setminus{\{k_i:i\in\mathbb{N}_l\}}.
\end{equation}
In particular, if $\psi_j$, $j\in\mathbb{N}_n$, are differentiable, then condition  \eqref{lambda_separable_nonsmooth} is equivalent to
\begin{equation}\label{lambda_separable_smooth}
\lambda\geq|\psi_{j}'(0)|,\ \ \mbox{for all}\ \ j\in\mathbb{N}_n\setminus{\{k_i:i\in\mathbb{N}_l\}}.
\end{equation}
\end{theorem}
\begin{proof}
We first consider the case when the regularization problem \eqref{optimization_problem} has the most sparse solution, that is,  $l=0$. Since $\bm{\psi}$ has the form \eqref{separable_psi}, we have that  $\mathbf{u}^*=\mathbf{0}$ is a solution of  \eqref{optimization_problem} if and only if 
for each $j\in\mathbb{N}_{n}$, $u_j^*=0$ is a solution of \eqref{optimization_1m} with $\psi:=\psi_j$. The latter guaranteed by Lemma \ref{choice_sparsity_1dimension_nonsmooth} is equivalent to $\lambda\geq \mathcal{\max}\{\psi'_{j,-}(0),-\psi'_{j,+}(0)\}$ for all $j\in\mathbb{N}_n$, that is, the inequalities \eqref{lambda_separable_nonsmooth} with $l=0$ hold.

We now prove the case when $l\neq0$. On one hand, suppose that $\mathbf{u}^{*}$ is a solution of the regularization problem \eqref{optimization_problem} having sparsity of level $l'$ with $l'\leq l$ for some $l\in\mathbb{N}_n$. It follows from $\mathbf{u}^{*}\in \Omega_{l'}$ that there exist $1\leq k_1<k_2<\cdots< k_{l'}\leq n$ and $u_{k_i}^*\in\mathbb{R}\setminus{\{0\}}$, $i\in\mathbb{N}_{l'}$, such that $\mathbf{u}^{*}=\sum_{i\in\mathbb{N}_{l'}}u_{k_i}^*\mathbf{e}_{k_i}.$ Since $\bm{\psi}$ has the form \eqref{separable_psi}, we obtain that for each $i\in\mathbb{N}_{l'}$, $u_{k_i}^*$ is a nonzero solution of the regularization problem \eqref{optimization_1m} with $\psi:=\psi_{k_i}$ and for each $j\in\mathbb{N}_n\setminus{\{k_i:i\in\mathbb{N}_{l'}\}}$, the regularization problem \eqref{optimization_1m} with $\psi:=\psi_{j}$ has $u_j^*=0$ as its solution. Lemma \ref{choice_sparsity_1dimension_nonsmooth} ensures that $\lambda\geq\mathcal{\max}\left\{\psi'_{j,-}(0),-\psi'_{j,+}(0)\right\}$, for all $j\in\mathbb{N}_n\setminus{\{k_i:i\in\mathbb{N}_{l'}\}}.$ The above inequalities imply that inequalities \eqref{lambda_separable_nonsmooth} hold by choosing arbitrary distinct integers $k_i$, $i=l'+1,\ldots,l$ in $\mathbb{N}_n\setminus{\{k_i:i\in\mathbb{N}_{l'}\}}$. 
On the other hand, suppose that there exist distinct integers 
$k_i\in\mathbb{N}_n$, $i\in\mathbb{N}_l$, such that inequalities 
\eqref{lambda_separable_nonsmooth} hold. It suffices to find a solution for the  regularization problem \eqref{optimization_problem}, whose level of sparsity is not more than $l$. By  Lemma \ref{choice_sparsity_1dimension_nonsmooth}, inequalities 
\eqref{lambda_separable_nonsmooth} tell us that for each $j\in\mathbb{N}_n\setminus{\{k_i:i\in\mathbb{N}_{l}\}}$, the regularization problem \eqref{optimization_1m} with $\psi:=\psi_{j}$ has a zero solution. In addition, we choose for each $i\in\mathbb{N}_{l}$ a solution $u_{k_i}^*$ of the regularization problem \eqref{optimization_1m} with $\psi:=\psi_{k_i}$. Therefore, by setting $\mathbf{u}^{*}:=\sum_{i\in\mathbb{N}_{l}}u_{k_i}^*\mathbf{e}_{k_i}$, we obtain a desired solution for the regularization problem \eqref{optimization_problem}. 

If $\psi_j$, $j\in\mathbb{N}_n$, are differentiable, then there hold for all $j\in\mathbb{N}_n$ that  $\psi_{j,-}'(0)=\psi_{j,+}'(0)=\psi'(0)$. Substituting these equations into inequalities \eqref{lambda_separable_nonsmooth} leads to the desired inequalities \eqref{lambda_separable_smooth}.
\end{proof}

Theorem \ref{choice_sparsity_separable_nonsmooth} reveals the relation between sparsity of a solution of the regularization problem \eqref{optimization_problem} and a choice of the regularization parameter, when the fidelity term $\bm{\psi}$ being additively separable. The choice of the regularization parameter depends on the left and right derivatives or the derivatives of the sequence $\psi_j$, $j\in\mathbb{N}_n$, of univariate functions.

We specialize Theorem \ref{choice_sparsity_separable_nonsmooth} to the Lasso regularized model \eqref{lasso} with $p=n$ and the matrix $\mathbf{A}$ being the identity matrix of order $n$, that is,  
\begin{equation}\label{Lasso_A_I}
\min\left\{\frac{1}{2}\|\mathbf{u-x}\|_2^2+\lambda\|\mathbf{u}\|_1:\mathbf{u}\in\mathbb{R}^n\right\}.
\end{equation}
In this case, the fidelity term $\bm{\psi}$ defined by 
$\bm{\psi}(\mathbf{u})
:=\frac{1}{2}\|\mathbf{u-x}\|_2^2$, $\mathbf{u}\in\mathbb{R}^n,$ has the form \eqref{separable_psi} with 
$\psi_j(u):=\frac{1}{2}(u-x_j)^2$, $u\in\mathbb{R}$, $j\in\mathbb{N}_n$, and is differentiable.

\begin{corollary}\label{separable_smooth_example}
Suppose that $\mathbf{x}:=[x_j:j\in\mathbb{N}_n]$ is a given vector in $\mathbb{R}^n$. Then the regularization problem \eqref{Lasso_A_I} with $\lambda>0$ has a unique solution having sparsity of level $l'$ with $l'\leq l$ for some $l\in\mathbb{Z}_{n+1}$ if and only if there exist distinct 
$k_i\in\mathbb{N}_n$, $i\in\mathbb{N}_l$, such that 
$\lambda\geq|x_j|$ for all $j\in\mathbb{N}_n\setminus{\{k_i:i\in\mathbb{N}_l\}}.$ 
\end{corollary}
\begin{proof}
Since the objective function is strictly convex, the regularization problem \eqref{Lasso_A_I} has a unique solution. Note that $\psi_j'(0)=-x_j$, for all $j\in\mathbb{N}_n$. Substituting these equations into  \eqref{lambda_separable_smooth} of  Theorem \ref{choice_sparsity_separable_nonsmooth} leads to the desired result of this corollary. 
\end{proof}

We next characterize the sparsity of a solution of the regularization problem \eqref{optimization_problem} for the case when the fidelity term $\bm{\psi}$ is block separable. For this purpose, we introduce the definition of the block separable function. Let  $\mathcal{S}:=\left\{S_{1}, S_2, \ldots, S_{d}\right\}$ be a partition of  $\mathbb{N}_n$, where for each $j\in \mathbb{N}_d$, $S_j$ is an ordered set following the natural order of elements in $\mathbb{N}_n$ and has cardinality $n_j$. Specifically, we assume $S_j:=\{i(j)_1, i(j)_2, \dots, i(j)_{n_j}\}$, with $i(j)_l\in \mathbb{N}_n$, $l\in \mathbb{N}_{n_j}$ and $i(j)_1<i(j)_2<\dots<i(j)_{n_j}$.
For each $\mathbf{u}\in\mathbb{R}^n$, set $\mathbf{u}_j:=[{u}_{{i(j)}_1},{u}_{{i(j)}_2}, \dots, {u}_{{i(j)}_{n_j}}]$ for all $j\in \mathbb{N}_{d}$.
A function  $\bm{\psi}:\mathbb{R}^n\rightarrow\mathbb{R}$ is called $\mathcal{S}$-block separable if there exist functions $\bm{\psi}_j:\mathbb{R}^{n_j}\rightarrow\mathbb{R}$,  $j\in\mathbb{N}_d$, such that 
\begin{equation}\label{block_separable_psi}
    \bm{\psi}(\mathbf{u})=\sum\limits_{j\in\mathbb{N}_d}\bm{\psi}_{j}(\mathbf{u}_{j}),\ \mbox{for all}\  \mathbf{u}\in\mathbb{R}^n.
\end{equation}
It is clear that an additively separable function $\bm{\psi}$ with the form \eqref{separable_psi} is $\mathcal{S}$-block separable with $\mathcal{S}$ being the nature partition of $\mathbb{N}_n$. That is, $S_j:=\{j\}$. A high-dimensional optimization problem having a block separable objective function can be reduced to several disjoint optimization problems with lower dimensionalities. 

By virtue of the block separability of $\bm{\psi}$ and the additive separability of the norm function $\|\cdot\|_1$, the regularization problem \eqref{optimization_problem} can be reduced to the following lower dimensional regularization problems 
\begin{equation}\label{subproblem_block_separable}
    \min\left\{\bm{\psi}_j(\mathbf{u}_j)+\lambda\|\mathbf{u}_j\|_1:\mathbf{u}_j\in\mathbb{R}^{n_j}\right\},\ j\in\mathbb{N}_d.
\end{equation}
To discuss the sparsity of the solutions of the regularization problem \eqref{optimization_problem} with $\bm{\psi}$ being block separable, we also need to define the level of block sparsity for a vector in $\mathbb{R}^n$. Let $\mathcal{S}:=\left\{S_{1}, S_2, \ldots, S_{d}\right\}$ be a partition of the index set $\mathbb{N}_n.$ We say that a vector $\mathbf{x}\in\mathbb{R}^n$ has $\mathcal{S}$-block sparsity of level $l\in\mathbb{Z}_{d+1}$ if it has exactly $l$ number of nonzero sub-vectors. 

We give in the following lemma a choice of the parameter for the case when the regularization problem \eqref{optimization_problem} has a most sparse solution without assuming $\bm{\psi}$ being block separable.

\begin{lemma}\label{choice_sparsity_ndimension_nonsmooth}
Suppose that $\bm{\psi}$ is a convex function on $\mathbb{R}^{n}$. Then 
the regularization problem \eqref{optimization_problem} with $\lambda>0$ has a solution $\mathbf{u}^*=\mathbf{0}$  if and only if 
\begin{equation}\label{lambda_nonsmooth}
\lambda\geq\mathrm{min}\left\{\|\mathbf{y}\|_{\infty}:\mathbf{y}\in\partial\bm{\psi}(\mathbf{0}) \right\}.
\end{equation}
\end{lemma}
\begin{proof}
According to the Fermat rule, vector $\mathbf{u}^*=\mathbf{0}$ is a solution of the regularization problem \eqref{optimization_problem} if and only if
$\mathbf{0}\in  \partial(\bm{\psi}+\lambda\|\cdot\|_1)(\mathbf{0}),$
which guaranteed by the continuity of the $\ell_1$ norm function is equivalent to $\mathbf{0}\in  \partial\bm{\psi}(\mathbf{0})+\lambda \partial\|\cdot\|_1(\mathbf{0}).$ That is to say there exists $\mathbf{y}\in\partial\bm{\psi}(\mathbf{0})$ such that $-\mathbf{y}\in\lambda \partial\|\cdot\|_1(\mathbf{0}).$ The subdifferential of the $\ell_1$ norm function at zero can be represented by 
\begin{equation}\label{abs_nm_subdiff}
\partial\|\cdot\|_1(\mathbf{0})
=\left\{\mathbf{y}\in\mathbb{R}^n: |y_j|\leq 1, j\in\mathbb{N}_n\right\}. 
\end{equation}
By employing equation \eqref{abs_nm_subdiff}, we rewrite the inclusion relation $-\mathbf{y}\in\lambda \partial\|\cdot\|_1(\mathbf{0})$ as $\lambda\geq\|\mathbf{y}\|_{\infty}$. We then conclude that $\mathbf{u}^*=\mathbf{0}$ is a solution of  \eqref{optimization_problem} if and only if there exists $\mathbf{y}\in\partial\bm{\psi}(\mathbf{0})$ such that $\lambda\geq\|\mathbf{y}\|_{\infty}$. It is clear that the latter is equivalent to   \eqref{lambda_nonsmooth}.
\end{proof}

With the help of Lemma \ref{choice_sparsity_ndimension_nonsmooth}, we present a parameter choice so that a solution of the regularization problem \eqref{optimization_problem}, with the fidelity term $\bm{\psi}$ being in the form \eqref{block_separable_psi}, has block sparsity of a prescribed level. In the following discussions, we alwayse suppose that $\mathcal{S}:=\left\{S_{1},S_{2},\ldots, S_{d}\right\}$ is a partition of the index set $\mathbb{N}_n$ and $n_j$ is the cardinality of $S_j$ for all $j\in\mathbb{N}_d$.

\begin{theorem}\label{choice_sparsity_block_nonsmooth}
Suppose that $\bm{\psi}_j$, $j\in\mathbb{N}_d$, are convex functions on $\mathbb{R}^{n_j}$ and $\bm{\psi}$ is an $\mathcal{S}$-block separable function having the form \eqref{block_separable_psi}. Then the regularization problem \eqref{optimization_problem} with $\lambda>0$ has a solution having the $\mathcal{S}$-block sparsity of level $l'$ with $l'\leq l$ for some $l\in\mathbb{Z}_{d+1}$ if and only if there exist distinct 
$k_i\in\mathbb{N}_d$, $i\in\mathbb{N}_l$, such that
\begin{equation}\label{lambda_block_nonsmooth}
\lambda\geq\mathrm{min}\left\{\|\mathbf{y}\|_{\infty}:\mathbf{y}\in\partial\bm{\psi}_{j}(\mathbf{0}) \right\}, \ \mbox{for all}\ j\in\mathbb{N}_d\setminus{\{k_i:i\in\mathbb{N}_l\}}.  
\end{equation}
In particular, if $\bm{\psi}_j$, $j\in\mathbb{N}_d$, are differentiable, then 
condition \eqref{lambda_block_nonsmooth} is equivalent to 
\begin{equation}\label{lambda_block_smooth}
    \lambda\geq\|\nabla\bm{\psi}_{j}(\mathbf{0})\|_\infty,\ \ \mbox{for all}\ \ j\in\mathbb{N}_d\setminus{\{k_i:i\in\mathbb{N}_l\}}.
\end{equation} 
\end{theorem}
\begin{proof}
If $l=0$, we shall give a choice of parameter $\lambda$ so that the regularization problem \eqref{optimization_problem} has the trival solution $\mathbf{u}^*=\mathbf{0}$.
Note that $\mathbf{u}^*=\mathbf{0}$ is a solution of \eqref{optimization_problem} if and only if 
for each $j\in\mathbb{N}_{d}$, $\mathbf{u}_j^*=\mathbf{0}$ is a solution of the regularization problem
\eqref{subproblem_block_separable}. Lemma \ref{choice_sparsity_ndimension_nonsmooth} ensures that the latter holds if and only if $\lambda\geq\mathrm{min}\left\{\|\mathbf{y}\|_{\infty}:\mathbf{y}\in\partial\bm{\psi}_{j}(\mathbf{0}) \right\}$ for all $j\in\mathbb{N}_d$, which coincides with the inequalities \eqref{lambda_block_nonsmooth} with $l=0$.

We next prove this theorem for the case that $l\neq0$. We suppose that $\mathbf{u}^*$ as a solution of the regularization problem \eqref{optimization_problem} has the $\mathcal{S}$-block sparsity of level $l'$ with $l'\leq l$ and $\mathbf{u}^*_{j},$ $j\in\mathbb{N}_d$, are its sub-vectors. That is, there exist distinct integers
$k_i$, $i\in\mathbb{N}_{l'}$, in $\mathbb{N}_d$ such that $\mathbf{u}^{*}_{j}=\mathbf{0}$ for all $j\in\mathbb{N}_d\setminus{\{k_i:i\in\mathbb{N}_{l'}\}}$. Hence, the regularization problem \eqref{subproblem_block_separable} with $j\in\mathbb{N}_d\setminus{\{k_i:i\in\mathbb{N}_{l'}\}}$ has the trival solution $\mathbf{u}^{*}_{j}=\mathbf{0}$. Again by employing Lemma \ref{choice_sparsity_ndimension_nonsmooth}, we obtain that
$\lambda\geq\mathrm{min}\left\{\|\mathbf{y}\|_{\infty}:\mathbf{y}\in\partial\bm{\psi}_{j}(\mathbf{0}) \right\},$ for all $j\in\mathbb{N}_d\setminus{\{k_i:i\in\mathbb{N}_{l'}\}}.$ To prove the desired conclusion, we choose distinct integers $k_i$, $i=l'+1,\ldots,l$ in $\mathbb{N}_d\setminus{\{k_i:i\in\mathbb{N}_{l'}\}}$. Then the inequalities  \eqref{lambda_block_nonsmooth} follow directly from the above inequalities.

Conversely, we suppose that there exist distinct integers 
$k_i$, $i\in\mathbb{N}_l$, in $\mathbb{N}_d$ such that the inequalities 
\eqref{lambda_block_nonsmooth} hold. It follows from Lemma \ref{choice_sparsity_ndimension_nonsmooth} that for each $j\in\mathbb{N}_d\setminus{\{k_i:i\in\mathbb{N}_{l}\}}$, $\mathbf{u}_j^*=\mathbf{0}$ is a solution of the  regularization problem \eqref{subproblem_block_separable}. To construct a solution of the regularization problem \eqref{optimization_problem}, we choose for each $i\in\mathbb{N}_{l}$ a solution $\mathbf{u}_{k_i}^*$ of the regularization problem \eqref{subproblem_block_separable} with $j:=k_i$. Let $\mathbf{u}^{*}$ be the vector in $\mathbb{R}^n$ with the sub-vectors $\mathbf{u}_j$, $j\in\mathbb{N}_d,$ being defined above. It is obvious that $\mathbf{u}^{*}$ is a solution for the regularization problem \eqref{optimization_problem} and its level of $\mathcal{S}$-block sparsity is not more than $l$. 

If  $\bm{\psi}_j$, $j\in\mathbb{N}_d$, are differentiable, then their subdifferential at zero are the singleton $\nabla\bm{\psi}_{j}(\mathbf{0})$. This together with inequalities \eqref{lambda_block_nonsmooth} leads to the desired inequalities \eqref{lambda_block_smooth}.
\end{proof}

In the following, we consider the Lasso regularized model \eqref{lasso} and discuss when the fidelity term $\bm{\psi}$ defined by \eqref{square_psi_u} is block separable. Throughout this paper, we denote the $j$-th column of a matrix $\mathbf{M}\in\mathbb{R}^{m\times n}$ by $\mathbf{M}_j$. Suppose that $\mathcal{S}:=\left\{S_{1},S_{2},\ldots, S_{d}\right\}$ is a partition of the set $\mathbb{N}_n$ and $n_j$ is the cardinality of $S_j$ for all $j\in\mathbb{N}_d$. Associated with the partition $\mathcal{S}$, we decompose a matrix $\mathbf{M}\in\mathbb{R}^{m\times n}$ into $d$ sub-matrices by setting $\mathbf{M}_{(j)}:=[\mathbf{M}_k:k\in S_{j}]\in\mathbb{R}^{m\times n_j}$ for all $j\in\mathbb{N}_d$. The next lemma provides a sufficient and necessary condition ensuring the block separability of the fidelity term $\bm{\psi}$ defined by \eqref{square_psi_u}.
\begin{lemma}\label{Lasso_block_separable}
Suppose that   $\mathbf{x}\in\mathbb{R}^p$ and $\mathbf{A}\in\mathbb{R}^{p\times n}$. Then the function $\bm{\psi}$ defined by \eqref{square_psi_u} is $\mathcal{S}$-block separable if and only if there holds \begin{equation}\label{suff_nece_Lasso_block_separable}
(\mathbf{A}_{(j)})^{\top}\mathbf{A}_{(k)}=\mathbf{0},\ \mbox{for all}\ j,k\in\mathbb{N}_d \ \mbox{and}\ j\neq{k}.
\end{equation}
\end{lemma}
\begin{proof}
According to the definition  \eqref{square_psi_u} of $\bm{\psi}$, we have that   
\begin{equation}\label{equivalent_representation_psi}
\bm{\psi}(\mathbf{u})=\frac{1}{2}\mathbf{u}^{\top}\mathbf{A}^{\top}\mathbf{A}\mathbf{u}-\mathbf{x}^{\top}\mathbf{A}\mathbf{u}+\frac{1}{2}\mathbf{x}^{\top}\mathbf{x}, \ 
\mbox{for\ all}\ \mathbf{u}\in\mathbb{R}^n.
\end{equation}
It follows from the decomposition of $\mathbf{A}$ and that of each vector $\mathbf{u}$ in $\mathbb{R}^n$ with respect to $\mathcal{S}$ that 
$\mathbf{A}\mathbf{u}=\sum_{j\in\mathbb{N}_d}\mathbf{A}_{(j)}\mathbf{u}_j$, for all $\mathbf{u}\in\mathbb{R}^n.
$ Substituting the above equation into equation \eqref{equivalent_representation_psi}, we obtain that 
\begin{equation}\label{equivalent_representation_psi1}
\bm{\psi}(\mathbf{u})=\frac{1}{2}\sum_{j\in\mathbb{N}_d}\sum_{k\in\mathbb{N}_d}\mathbf{u}^{\top}_j(\mathbf{A}_{(j)})^{\top}\mathbf{A}_{(k)}\mathbf{u}_k-\sum_{j\in\mathbb{N}_d}\mathbf{x}^{\top}\mathbf{A}_{(j)}\mathbf{u}_j+\frac{1}{2}\mathbf{x}^{\top}\mathbf{x}, \ 
\mbox{for\ all}\ \mathbf{u}\in\mathbb{R}^n.
\end{equation}
It is obvious that the last two items in the right hand side of equation \eqref{equivalent_representation_psi1} are both $\mathcal{S}$-block separable. Hence, $\bm{\psi}$ is $\mathcal{S}$-block separable if and only if the first item in the right hand side of equation \eqref{equivalent_representation_psi1}
is $\mathcal{S}$-block separable.
The latter one is equivalent to that condition \eqref{suff_nece_Lasso_block_separable} holds.
\end{proof}

We now apply Theorem \ref{choice_sparsity_block_nonsmooth}
to the Lasso regularized model \eqref{lasso} when the matrix $\mathbf{A}$ satisfies condition \eqref{suff_nece_Lasso_block_separable}.

\begin{corollary}\label{block_smooth_example}
Suppose that $\mathbf{x}\in\mathbb{R}^p$, $\mathbf{A}\in\mathbb{R}^{p\times n}$ and condition \eqref{suff_nece_Lasso_block_separable} holds. Then the regularization problem \eqref{lasso} with $\lambda>0$ has a solution having the $\mathcal{S}$-block sparsity of level $l'$ with $l'\leq l$ for some $l\in\mathbb{Z}_{d+1}$ if and only if there exist distinct 
$k_i\in\mathbb{N}_d$, $i\in\mathbb{N}_l$, such that
\begin{equation}\label{lambda_Lasso_block}
\lambda\geq\|(\mathbf{A}_{(j)})^{\top}\mathbf{x}\|_\infty,\ \mbox{for all}\ j\in\mathbb{N}_d\setminus{\{k_i:i\in\mathbb{N}_l\}}.
\end{equation}     
\end{corollary}
\begin{proof}
Since condition \eqref{suff_nece_Lasso_block_separable} holds, Lemma \ref{Lasso_block_separable} ensures that the fidelity term $\bm{\psi}$ involved in the regularization problem \eqref{lasso} is  $\mathcal{S}$-block separable. 
Substituting condition \eqref{suff_nece_Lasso_block_separable} into equation \eqref{equivalent_representation_psi1}, $\bm{\psi}$ can be represented in the form \eqref{block_separable_psi} with $\bm{\psi}_j$, $j\in\mathbb{N}_d$, being defined by 
\begin{equation}\label{subfunction_psi_u_j}
\bm{\psi}_j(\mathbf{u}_j):=\frac{1}{2}\|\mathbf{A}_{(j)}\mathbf{u}_j\|^2_2-\mathbf{x}^{\top}\mathbf{A}_{(j)}\mathbf{u}_j+\frac{1}{2d}\mathbf{x}^{\top}\mathbf{x}, \
\mbox{for all}\ \mathbf{u}_j\in\mathbb{R}^{n_j}\ \mbox{and all}\ j\in\mathbb{N}_d.
\end{equation}
Thus, we conclude by  Theorem \ref{choice_sparsity_block_nonsmooth} that the regularization problem \eqref{lasso} has a solution having the $\mathcal{S}$-block sparsity of level $l'$ with $l'\leq l$ for some $l\in\mathbb{Z}_{d+1}$ if and only if there exist distinct integers 
$k_i$, $i\in\mathbb{N}_l$, in $\mathbb{N}_d$ such that inequality \eqref{lambda_block_smooth} holds. Substituting 
$\nabla\bm{\psi}_{j}(\mathbf{0})
=-(\mathbf{A}_{(j)})^{\top}\mathbf{x}$ for all $j\in\mathbb{N}_d$ into inequality \eqref{lambda_block_smooth} leads directly to the desired inequality \eqref{lambda_Lasso_block}. 
\end{proof}

In signal or imaging processing, the matrix $\mathbf{A}$ involved in the regularization problem \eqref{lasso} is often chosen as an orthogonal matrix, such as an orthogonal discrete wavelet transform. We note that an orthogonal matrix is a special matrix satisfying condition \eqref{suff_nece_Lasso_block_separable} for any partition $\mathcal{S}$ of the index set $\mathbb{N}_n$. Especially, condition \eqref{suff_nece_Lasso_block_separable} holds for the nature partition  $\mathcal{S}$ of $\mathbb{N}_n$. In this case, Corollary \ref{block_smooth_example} ensures that the regularization problem \eqref{lasso} with $\lambda>0$ has a solution having the sparsity of level $l'$ with $l'\leq l$ for some $l\in\mathbb{Z}_{n+1}$ if and only if there exist distinct integers $k_i\in\mathbb{N}_n$, $i\in\mathbb{N}_l$, such that
$\lambda\geq|(\mathbf{A}_j)^{\top}\mathbf{x}|,\ \mbox{for all}\ j\in\mathbb{N}_n\setminus{\{k_i:i\in\mathbb{N}_l\}}$.

We also consider the group Lasso regularized model \eqref{group_lasso} which is designed to obtain the block sparsity of the solutions. To describe a choice of the parameter for this regularization problem, we also assume that condition  \eqref{suff_nece_Lasso_block_separable} holds. Then by Lemma 
\ref{Lasso_block_separable}, the  fidelity term in this problem is $\mathcal{S}$-block separable and has the form  \eqref{block_separable_psi} with $\bm{\psi}_j$, $j\in\mathbb{N}_d$, being defined by \eqref{subfunction_psi_u_j}. It is obvious that the regularizer of the group Lasso regularized model \eqref{group_lasso} is also $\mathcal{S}$-block separable. Therefore, it can be reduced to $d$ lower
dimensional regularization problems
\begin{equation}\label{group_lasso_sub-problem}
\min\left\{\bm{\psi}_j(\mathbf{u}_j)+\lambda\sqrt{n_j}\|\mathbf{u}_j\|_2:\mathbf{u}_j\in\mathbb{R}^{n_j}\right\}, \ j\in\mathbb{N}_d.
\end{equation}
Though characterizing the sparsity of the solutions of the regularization problem \eqref{group_lasso_sub-problem}, we obtain the following parameter choice with which the group Lasso regularized model \eqref{group_lasso} has a solution having block sparsity of a prescribed level. 
\begin{theorem}\label{block_smooth_example1}
Suppose that $\mathbf{x}\in\mathbb{R}^p$, $\mathbf{A}\in\mathbb{R}^{p\times n}$ and condition \eqref{suff_nece_Lasso_block_separable} holds. Then the regularization problem \eqref{group_lasso} with $\lambda>0$ has a solution having the $\mathcal{S}$-block sparsity of level $l'$ with $l'\leq l$ for some $l\in\mathbb{Z}_{d+1}$ if and only if there exist distinct 
$k_i\in\mathbb{N}_d$, $i\in\mathbb{N}_l$, such that 
$\lambda\geq\left\|(\mathbf{A}_{(j)})^{\top}\mathbf{x}\right\|_2/{\sqrt{n_j}}$, for all $j\in\mathbb{N}_d\setminus{\{k_i:i\in\mathbb{N}_l\}}.$
\end{theorem}
\begin{proof}
Condition \eqref{suff_nece_Lasso_block_separable} ensures that the fidelity term in the regularization problem \eqref{group_lasso} is $\mathcal{S}$-block separable and then $\mathbf{u}^*$ is a solution of \eqref{group_lasso} if and only if 
for each $j\in\mathbb{N}_{d}$, $\mathbf{u}_j^*$ is a solution of 
\eqref{group_lasso_sub-problem}. To characterize the sparsity of $\mathbf{u}^*$, we shall show that for each $j\in\mathbb{N}_d$, the regularization problem \eqref{group_lasso_sub-problem} has a solution  $\mathbf{u}_j^*=\mathbf{0}$ if and only if there holds
\begin{equation}\label{lambda_group_Lasso_block_sub-problem}
\lambda\geq\left\|(\mathbf{A}_{(j)})^{\top}\mathbf{x}\right\|_2/\sqrt{n_j}.
\end{equation}    
It follows from the Fermat rule and the differentiability of $\bm{\psi}_j$ that $\mathbf{u}_j^*=\mathbf{0}$ is a solution of \eqref{group_lasso_sub-problem} if and only if 
\begin{equation}\label{lambda_group_Lasso_block_sub-problem1}
\mathbf{0}\in\nabla\bm{\psi}_{j}(\mathbf{0})+\lambda\sqrt{n_j}\partial\|\cdot\|_2(\mathbf{0}).
\end{equation}
Note that there hold $\nabla\bm{\psi}_{j}(\mathbf{0})=-(\mathbf{A}_{(j)})^{\top}\mathbf{x}$ and $
\partial\|\cdot\|_2(\mathbf{0})
=\left\{\mathbf{y}\in\mathbb{R}^{n_j}:\|\mathbf{y}\|_2\leq 1\right\}.$
By employing these equations, we get that inclusion relation \eqref{lambda_group_Lasso_block_sub-problem1} is equivalent to inequality  \eqref{lambda_group_Lasso_block_sub-problem}. Consequently, we conclude that $\mathbf{u}_j^*=\mathbf{0}$ is a solution of \eqref{group_lasso_sub-problem} if and only if inequality  \eqref{lambda_group_Lasso_block_sub-problem} holds.

By arguments similar to those
used in the proof of Theorem  \ref{choice_sparsity_block_nonsmooth} and by employing inequality \eqref{lambda_group_Lasso_block_sub-problem}, we get the
desired conclusion of this corollary.
\end{proof}

Many practical applications can be modeled as in the form \eqref{optimization_problem} with
the fidelity term $\bm{\psi}$ being neither additively separable nor block separable. Here comes the theorem concerning
a sparsity characterization of the solutions of the regularization problem \eqref{optimization_problem} when $\bm{\psi}$ is a general convex on $\mathbb{R}^n$. For each $j\in\mathbb{N}_n$, we denote by $\bm{\psi}'_j$ the partial derivative of $\bm{\psi}$ with respect to the $j$-th variable.

\begin{theorem}\label{choice_sparsity_general_nonsmooth}
Suppose that $\bm{\psi}$ is a convex function on $\mathbb{R}^{n}$. Then the regularization problem \eqref{optimization_problem} with $\lambda>0$ has a solution  $\mathbf{u}^{*}=\sum_{i\in\mathbb{N}_{l}}u_{k_i}^*\mathbf{e}_{k_i}\in \Omega_l$ for some $l\in\mathbb{Z}_{n+1}$ if and only if there exists $\mathbf{y}:=[y_j:j\in\mathbb{N}_n]\in\partial\bm{\psi}
(\mathbf{u}^{*})$ such that 
\begin{equation}\label{lambda_general_nonsmooth1}
\lambda=-y_{k_i}\mathrm{sign}({u}_{k_i}^{*}), \ 
i\in\mathbb{N}_l,
\end{equation}
and
\begin{equation}\label{lambda_general_nonsmooth2}
\lambda\geq|y_j|,\ j\in \mathbb{N}_n\setminus\{k_i:i\in\mathbb{N}_l\}.
\end{equation}
In particular, if $\bm{\psi}$ is a differentiable, then conditions \eqref{lambda_general_nonsmooth1} and \eqref{lambda_general_nonsmooth2} are equivalent to 
\begin{equation}\label{lambda_general_smooth1}
\lambda=-\bm{\psi}'_{k_i}(\mathbf{u}^{*})\mathrm{sign}({u}_{k_i}^{*}), \ 
i\in \mathbb{N}_l,
\end{equation}
and
\begin{equation}\label{lambda_general_smooth2}
\lambda\geq\left|\bm{\psi}'_j(\mathbf{u}^{*})\right|, \ j\in \mathbb{N}_n\setminus\{k_i:i\in\mathbb{N}_l\}.
\end{equation}
\end{theorem}
\begin{proof}
By using the Fermat rule and the continuity of the $\ell_1$ norm function, we conclude that $\mathbf{u}^{*}$ is a solution of the regularization problem \eqref{optimization_problem} if and only if 
$\mathbf{0}\in  \partial\bm{\psi}(\mathbf{u}^{*})+
\lambda \partial\|\cdot\|_1(\mathbf{u}^{*}).$ 
This is equivalent to that there exists $\mathbf{y}:=[y_j:j\in\mathbb{N}_n]\in\partial\bm{\psi}(\mathbf{u}^*)$ such that 
$-\mathbf{y}\in\lambda \partial\|\cdot\|_1(\mathbf{u}^*).$ Noting that $\mathbf{u}^{*}=\sum_{i\in\mathbb{N}_{l}}u_{k_i}^*\mathbf{e}_{k_i}$ with $u_{k_i}^*\in\mathbb{R}\setminus\{0\}$, $i\in\mathbb{N}_l$, we obtain that \begin{equation*}\label{subdiff-1norm-u*}
\partial\|\cdot\|_1(\mathbf{u}^*)
=\left\{\mathbf{z}\in\mathbb{R}^n: z_{k_i}=\mathrm{sign}(u_{k_i}^*),i\in\mathbb{N}_l \ \mbox{and}\ |z_j|\leq 1, j\in\mathbb{N}_n\setminus\{k_i:i\in\mathbb{N}_l\}\right\}. 
\end{equation*}
By using the above equation, we rewrite inclusion relation $-\mathbf{y}\in\lambda \partial\|\cdot\|_1(\mathbf{u}^*)$ as  \eqref{lambda_general_nonsmooth1} and \eqref{lambda_general_nonsmooth2}. 

If $\bm{\psi}$ is differentiable, 
then the subdifferential of $\bm{\psi}$ at $\mathbf{u}^{*}$ is the singleton  $\nabla\bm{\psi}(\mathbf{u}^{*})$. Substituting $y_j=\bm{\psi}'_j(\mathbf{u}^{*})$, $j\in\mathbb{N}_n$, into \eqref{lambda_general_nonsmooth1} and \eqref{lambda_general_nonsmooth2} leads directly \eqref{lambda_general_smooth1} and \eqref{lambda_general_smooth2}, respectively. 
\end{proof}

We now apply Theorem \ref{choice_sparsity_general_nonsmooth} to the Lasso regularized model \eqref{lasso}. Here, the fidelity term $\bm{\psi}$ defined by \eqref{square_psi_u} is differentiable but neither additively separable nor block separable. 

\begin{corollary}\label{general-smooth-example}
Suppose that $\mathbf{x}\in\mathbb{R}^p$ is a given vector and $\mathbf{A}\in\mathbb{R}^{p\times n}$ a prescribed matrix. Then the regularization problem \eqref{lasso} with $\lambda>0$ has a solution $\mathbf{u}^{*}=\sum_{i\in\mathbb{N}_{l}}u_{k_i}^*\mathbf{e}_{k_i}\in \Omega_l$ for some $l\in\mathbb{Z}_{n+1}$ if and only if there hold
\begin{equation}\label{lambda_Au1}
\lambda=(\mathbf{A}_{k_i})^{\top}(\mathbf{x}-\mathbf{Au^*})\mathrm{sign}({u}_{k_i}^{*}), \ i\in \mathbb{N}_l,
\end{equation}
and 
\begin{equation}\label{lambda_Au2}
\lambda\geq\left|(\mathbf{A}_j)^{\top}(\mathbf{Au^*}-\mathbf{x})\right|,  \ j\in \mathbb{N}_n\setminus\{k_i:i\in\mathbb{N}_l\}.
\end{equation}
\end{corollary}

\begin{proof}
Note that the gradient of $\bm{\psi}$ at $\mathbf{u}^*$ has the form $\nabla\bm{\psi}(\mathbf{u}^*)=\mathbf{A}^\top(\mathbf{A}\mathbf{u}^*-\mathbf{x})$. That is, 
\begin{equation}\label{partial-diff}
\bm{\psi}'_j(\mathbf{u}^*)=(\mathbf{A}_j)^{\top}(\mathbf{Au^*}-\mathbf{x}),\ \mbox{for all}\ j\in\mathbb{N}_n.
\end{equation} 
Theorem \ref{choice_sparsity_general_nonsmooth} ensures that $\mathbf{u}^{*}=\sum_{i\in\mathbb{N}_{l}}u_{k_i}^*\mathbf{e}_{k_i}\in \Omega_l$ is a solution of the regularization problem \eqref{lasso} if and only if \eqref{lambda_general_smooth1} and \eqref{lambda_general_smooth2} hold. According to equation \eqref{partial-diff}, we conclude that the latter is equivalent to that \eqref{lambda_Au1} and \eqref{lambda_Au2} hold. 
\end{proof}

As a special case of Corollary \ref{general-smooth-example}, the Lasso regularized model has $\mathbf{u^*}=\mathbf{0}$ as a solution if and only if there holds
$\lambda\geq\left\|\mathbf{A}^{\top}\mathbf{x}\right\|_{\infty}.$
We remark that the results about the Lasso regularized model stated in Corollary \ref{general-smooth-example} has been established in \cite{bach2011optimization}. 

When the fidelity term $\bm{\psi}$ has no special form such as \eqref{separable_psi} and \eqref{block_separable_psi}, Theorem \ref{choice_sparsity_general_nonsmooth}  provides a characterization of the regularization parameter with which the regularization problem \eqref{optimization_problem} has a solution with sparsity of a certain level. In fact, since conditions \eqref{lambda_general_nonsmooth1} and \eqref{lambda_general_nonsmooth2} (or \eqref{lambda_general_smooth1} and \eqref{lambda_general_smooth2}) depend on the corresponding solution, the characterization stated in Theorem \ref{choice_sparsity_general_nonsmooth} can not be used as a parameter choice strategy. Nevertheless, we can still observe from the characterization that the choice of the  regularization parameter can influence the sparsity of the solution. To ensure sparsity of the solution, the regularization parameter need to satisfy equalities and inequalities, which respectively characterize the non-zero components and the zero components of the solution. As the number of the inequalities increases, the solution becomes more sparse. When the conditions only include the inequalities, the solution is the most sparse. When partial information of the solution is known,  Theorem \ref{choice_sparsity_general_nonsmooth} can provide a choice of the regularization parameter. The following result is an attempt along this direction.

\begin{proposition}\label{Sufficient-Condition}
Suppose that $\bm{\psi}$ is a differentiable and convex function on $\mathbb{R}^n$ and for each $i\in\mathbb{N}_n$, $\bm{\psi}_i'$ is $L_i$-Lipschitz continuous on $\mathbb{R}^n$. Let $\mathbf{u}^*\in\mathbb{R}^n$ be a solution of the regularization problem \eqref{optimization_problem} with $\lambda>0$ and $\mathbf{v}\in\mathbb{R}^n$ satisfy $\|\mathbf{u}^*-\mathbf{v}\|_2\leq\epsilon$ for some $\epsilon>0$. If 
there exist distinct $k_i\in\mathbb{N}_n$, $i\in\mathbb{N}_l$, for some $l\in\mathbb{Z}_{n+1}$ such that 
\begin{equation}\label{lambda-sufficient-condition}
\lambda>\left|\bm{\psi}_j'(\mathbf{v})\right|+\epsilon L_j, \ \mbox{for all}\ \ j\in \mathbb{N}_n\setminus\{k_i:i\in\mathbb{N}_l\},
\end{equation}
then the sparsity level of $\mathbf{u}^*$ is less than or equal to $l$.
\end{proposition}

\begin{proof}
It follows from the Lipschitz continuity of $\bm{\psi}_j'$, $j\in\mathbb{N}_n$, that 
\begin{equation*}\label{Lipschitz-continuity}
\left|\bm{\psi}_j'(\mathbf{u}^*)-\bm{\psi}_j'(\mathbf{v})\right|\leq L_j\|\mathbf{u}^*-\mathbf{v}\|_2,\ \ j\in\mathbb{N}_n.
\end{equation*}
This together with the assumption $\|\mathbf{u}^*-\mathbf{v}\|_2\leq\epsilon$ leads to 
$\left|\bm{\psi}_j'(\mathbf{u}^*)-\bm{\psi}_j'(\mathbf{v})\right|\leq\epsilon L_j$, $j\in\mathbb{N}_n.$ 
Substituting these inequalities into inequality \eqref{lambda-sufficient-condition}, we get that   
\begin{equation*}
\lambda>\left|\bm{\psi}_j'(\mathbf{v})\right|+\left|\bm{\psi}_j'(\mathbf{u}^*)-\bm{\psi}_j'(\mathbf{v})\right|,\ j\in \mathbb{N}_n\setminus\{k_i:i\in\mathbb{N}_l\},
\end{equation*}
which further yields that
\begin{equation}\label{Lipschitz-continuity2} 
\lambda>|\bm{\psi}_j'(\mathbf{u}^*)|,\ j\in \mathbb{N}_n\setminus\{k_i:i\in\mathbb{N}_l\}.
\end{equation}
Suppose that $\mathbf{u}^{*}:=[u_j^*:j\in\mathbb{N}_n]$.
By Theorem \ref{choice_sparsity_general_nonsmooth},
from inequality \eqref{Lipschitz-continuity2} we conclude that $u_j^*=0$ for all $j\in \mathbb{N}_n\setminus\{k_i:i\in\mathbb{N}_l\}$. In fact, if there exists $j_0\in \mathbb{N}_n\setminus\{k_i:i\in\mathbb{N}_l\}$ such that $u_{j_0}^*\neq0$, then Theorem \ref{choice_sparsity_general_nonsmooth} ensures that 
$\lambda=-\bm{\psi}'_{j_0}(\mathbf{u}^{*})\mathrm{sign}({u}_{j_0}^{*})$. 
That is, $\lambda=\left|\bm{\psi}_{j_0}'(\mathbf{u}^*)\right|$, which contradicts to inequality \eqref{Lipschitz-continuity2}. Thus, we get that conclusion that the sparsity level of $\mathbf{u}^*$ is less than or equal to $l$.
\end{proof}

\section{Parameter Choices for Sparsity of the Transformed Regularized Solutions}

In this section, we continue our investigation about what choices of the regularization parameter lead to sparsity under a general transform matrix  $\mathbf{B}$ for the solutions of the regularization problem \eqref{optimization_problem_nm}. 
We first employ the singular value decomposition of matrix  $\mathbf{B}$ to convert  the regularization problem \eqref{optimization_problem_nm} to one with  $\mathbf{B}$ being a {\it degenerated} identity (an identity matrix augmented by zero matrices). Based on this result, we characterize the regularization parameter for a sparse regularized solution of the resulting regularization problem by using the approach used in section 3 for the case when  $\mathbf{B}:= \mathbf{I}$. We then present special results for several specific learning models.

We now characterize the sparsity of a solution under the transform matrix $\mathbf{B}$ of the regularization problem \eqref{optimization_problem_nm}.
Note that if $\mathbf{B}$ is an invertible square matrix, then by a simple change of variables the regularization problem \eqref{optimization_problem_nm} can be converted to problem \eqref{optimization_problem}. For general matrix  $\mathbf{B}$, we appeal to its pseudoinverse (Moore–Penrose inverse) \cite{horn2012matrix}.
To this end, we review the notion of the singular value decomposition  of a matrix  \cite{horn2012matrix}. Suppose that $\mathbf{B}$ is a real $m\times n$ matrix with the rank $r$ satisfying $0<r\leq \mathrm{min}\{m,n\}$. It is well-known that $\mathbf{B}$ has the singular value decomposition as
\begin{equation}\label{psi_SVD}
\mathbf{B}=\mathbf{U}\mathbf{\Lambda} \mathbf{V}^{\top},
\end{equation}
where $\mathbf{U}$ is an $m\times m$ orthogonal matrix, $\mathbf{\Lambda}$ is an $m\times n$ diagonal matrix with the singular values  $\sigma_1\geq\cdots\geq\sigma_r>0$ on the diagonal, and $\mathbf{V}$ is an $n\times n$ orthogonal matrix. 
A matrix, denoted by $\mathbf{M}^{\dag}$, is called the pseudoinverse of $\mathbf{M}$ if it satisfies the four conditions: (1) $\mathbf{M}\mathbf{M}^{\dag}\mathbf{M}=\mathbf{M}$, (2) $\mathbf{M}^{\dag}\mathbf{M}\mathbf{M}^{\dag}=\mathbf{M}^{\dag}$, (3) $(\mathbf{M}\mathbf{M}^{\dag})^{\top}=\mathbf{M}\mathbf{M}^{\dag}$, (4) $(\mathbf{M}^{\dag}\mathbf{M})^{\top}=\mathbf{M}^{\dag}\mathbf{M}$. The pseudoinverse is well-defined and unique for all matrices. It can be readily verified that the  pseudoinverse  $\mathbf{\Lambda}^\dag$ of the $m\times n$ diagonal matrix $\mathbf{\Lambda}$ is the $n\times m$ diagonal matrix with the nonzero diagonal entries given by $\sigma_1^{-1}, \dots, \sigma_r^{-1}$.
Thus, the pseudoinverse $\mathbf{B}^\dag$ of $\mathbf{B}$ can be represented by the singular value decomposition of $\mathbf{B}$ as $\mathbf{B}^\dag=\mathbf{V}\mathbf{\Lambda}^\dag\mathbf{U}^\top$.

In the next lemma, we consider inverting the linear system 
\begin{equation}\label{LinearSystem}
    \mathbf{B}\mathbf{u}=\mathbf{z}, \ \ \mbox{for}\ \ \mathbf{z}\in\mathcal{R}(\mathbf{B}).
\end{equation}
Here, $\mathcal{R}(\mathbf{B})$ denotes the range of $\mathbf{B}$. 
Note that solutions of the system \eqref{LinearSystem} may not be unique since if $\mathbf{u}'$ is a solution of \eqref{LinearSystem}, then $\mathbf{u}:=\mathbf{u}'+\mathbf{u}_0$ is a solution of \eqref{LinearSystem}, for any $\mathbf{u}_0$ satisfying $\mathbf{B}\mathbf{u}_0=\mathbf{0}$. 
It is known from \cite{Bjork1996} that by choosing $\mathbf{u}':=\mathbf{B}^\dag\mathbf{z}$ as a particular solution of \eqref{LinearSystem}, the general solution of \eqref{LinearSystem} has the form $\mathbf{u}=\mathbf{B}^\dag\mathbf{z}+\mathbf{V}
\llbracket\mathbf{0},\mathbf{v}\rrbracket$ for any $\mathbf{v}\in\mathbb{R}^{n-r}$. To convert the regularization problem \eqref{optimization_problem_nm} to an equivalent one, we gives in the next lemma an alternative form of the general solution of \eqref{LinearSystem}.
For convenience, let $\widetilde{\mathbf{U}}_{r}\in\mathbb{R}^{m\times r}$ denote the matrix composed of the first $r$ columns of $\mathbf{U}$. We introduce a diagonal matrix of order $n$ by $\mathbf{\Lambda}':=\mathrm{diag}\left(\sigma_1^{-1},\sigma_2^{-1},\ldots,\sigma_r^{-1},1,\ldots,1\right)$ and an $n\times(m+n-r)$ block diagonal matrix   
$\mathbf{U}':=\mathrm{diag}\left(\widetilde{\mathbf{U}}_{r}^{\top}, \mathbf{I}_{n-r}\right).$
Using these matrices, we define an $n\times(m+n-r)$ matrix  
\begin{equation}\label{psi_widetilde_B}
\mathbf{B}'
:=\mathbf{V}\mathbf{\Lambda}'\mathbf{U}'.
\end{equation}

\begin{lemma}\label{lemma: change variable}
Suppose that $\mathbf{B}$ is a real  $m\times n$ matrix with the singular value decomposition \eqref{psi_SVD} and $\mathbf{B}'$ is defined by \eqref{psi_widetilde_B}. If $\mathbf{z}\in\mathcal{R}(\mathbf{B})$, then the general solution of the linear system  \eqref{LinearSystem} has the form 
\begin{equation}\label{solution-Bu=z}
\mathbf{u}=\mathbf{B}'
\llbracket\mathbf{z},\mathbf{v}\rrbracket, \ \ \mbox{for any}\ \ \mathbf{v}\in\mathbb{R}^{n-r}. 
\end{equation}
Moreover, for each solution $\mathbf{u}$, the vector $\mathbf{v}$ satisfying \eqref{solution-Bu=z} is unique. 
\end{lemma} 
\begin{proof}

Substituting $\mathbf{B}^\dag=\mathbf{V}\mathbf{\Lambda}^\dag\mathbf{U}^\top$ into the representation of the general solution $\mathbf{u}$ of \eqref{LinearSystem}, we have that   
$\mathbf{u}=\mathbf{V}\left(\mathbf{\Lambda}^\dag\mathbf{U}^\top\mathbf{z}+\llbracket\mathbf{0},\mathbf{v}\rrbracket\right)$ for any $\mathbf{v}\in\mathbb{R}^{n-r}$. In terms of $\mathbf{\Lambda'}$ and $\mathbf{U}'$, the above equation can be rewritten as
\begin{equation}\label{matrix_presnet_u}
\mathbf{u}=\mathbf{V}\mathbf{\Lambda'}\mathbf{U}'
\llbracket\mathbf{z},\mathbf{v}\rrbracket.
\end{equation}
By setting $\mathbf{B}'$ as in \eqref{psi_widetilde_B}, equation
\eqref{matrix_presnet_u} can be represented as the desired form \eqref{solution-Bu=z}. 

It remains to verify that for each solution $\mathbf{u}$ of  \eqref{LinearSystem}, the vector $\mathbf{v}$ appearing in \eqref{solution-Bu=z} is unique. Suppose that $\mathbf{v}_1, \mathbf{v}_2 \in\mathbb{R}^{n-r}$ both satisfy  \eqref{solution-Bu=z}. We then obtain that $\mathbf{B}'\llbracket\mathbf{0},\mathbf{v}_1-\mathbf{v}_2\rrbracket=\mathbf{0},$ which together with the definition \eqref{psi_widetilde_B} of $\mathbf{B}'$ implies that $\mathbf{V}\llbracket\mathbf{0}, \mathbf{v}_1-\mathbf{v}_2\rrbracket=\mathbf{0}.$ The the invertibility of $\mathbf{V}$ ensures that $\mathbf{v}_1=\mathbf{v}_2$,  proving the desired result.
\end{proof}

Lemma \ref{lemma: change variable} allows us to introduce a mapping $\mathcal{B}$ from $\mathbb{R}^n$ to $\mathcal{R}(\mathbf{B})\times\mathbb{R}^{n-r}$. Specifically, for each $\mathbf{u}\in\mathbb{R}^n$,  we define  
\begin{equation}\label{mapping_B}
\mathcal{B}\mathbf{u}:=\llbracket\mathbf{z},\mathbf{v}\rrbracket,
\end{equation}
where $\mathbf{z}:=\mathbf{B}\mathbf{u}$ and $\mathbf{v}\in\mathbb{R}^{n-r}$ is the vector satisfying \eqref{solution-Bu=z}.
It follows from Lemma \ref{lemma: change variable} that the mapping $\mathcal{B}$ is well-defined and satisfies that \begin{equation}\label{operator_B_property}
\mathbf{B}'\mathcal{B}\mathbf{u}=\mathbf{u}, \ \mbox{for all}\ \mathbf{u}\in\mathbb{R}^n.
\end{equation} 
Next, we show that $\mathcal{B}$ is bijective. 

\begin{lemma}\label{B_bijective}
If $\mathbf{B}$ is a real $m\times n$ matrix and the mapping $\mathcal{B}$ is defined by \eqref{mapping_B}, then $\mathcal{B}$ is bijective from $\mathbb{R}^n$ onto $\mathcal{R}(\mathbf{B})\times\mathbb{R}^{n-r}$.
\end{lemma}
\begin{proof} It suffices to show that  $\mathcal{B}$ is surjective and injective.
We first verify the surjectivity of $\mathcal{B}$. For any $\mathbf{z}\in\mathcal{R}(\mathbf{B})$ and any $\mathbf{v}\in\mathbb{R}^{n-r}$, we define a vector $\mathbf{u}\in\mathbb{R}^n$ through \eqref{solution-Bu=z}. Lemma \ref{lemma: change variable} guarantees that $\mathbf{B}\mathbf{u}=\mathbf{z}$, which together with the definition of $\mathcal{B}$ implies that  $\mathcal{B}\mathbf{u}=\llbracket\mathbf{z},\mathbf{v}\rrbracket$. Hence, $\mathcal{B}$ is surjective. To prove the injectivity of $\mathcal{B}$, we suppose that $\mathbf{u}_1,\mathbf{u}_2\in\mathbb{R}^n$ satisfy  $\mathcal{B}\mathbf{u}_1=\mathcal{B}\mathbf{u}_2$. It follows from equations \eqref{operator_B_property} with $\mathbf{u}$ being replaced respectively by $\mathbf{u}_1$ and $\mathbf{u}_2$ that $\mathbf{u}_1=\mathbf{B}'\mathcal{B}\mathbf{u}_1=\mathbf{B}'\mathcal{B}\mathbf{u}_2=\mathbf{u}_2$. That is, $\mathcal{B}$ is injective. 
\end{proof}

We now reformulate the regularization problem \eqref{optimization_problem_nm} as an equivalent constrained regularization problem with $\mathbf{B}$ being a degenerated identity $\mathbf{I}':=[\mathbf{I}_m \ \mathbf{0}]\in\mathbb{R}^{m\times(m+n-r)}$, that is,
\begin{equation}\label{psi_solve_Bu}
\min\left\{\bm{\psi}\circ
\mathbf{B}'(\mathbf{w})
+\lambda\|\mathbf{I}'\mathbf{w}\|_1:\mathbf{w}\in\mathcal{R}(\mathbf{B})\times\mathbb{R}^{n-r}\right\}.
\end{equation}
Noting that the vector composed by the first $m$ components of $\mathbf{w}$ is constrained to $\mathcal{R}(\mathbf{B})$, \eqref{psi_solve_Bu} is a constrained optimization problem.

\begin{proposition}\label{equi_mini}
Suppose that $\mathbf{B}$ is a real  $m\times n$ matrix with the singular value decomposition \eqref{psi_SVD},  $\mathbf{B}'$ is the matrix defined by  \eqref{psi_widetilde_B} and $\mathcal{B}$ is the mapping defined by \eqref{mapping_B}. Then $\mathbf{u}^*$ is a solution of the regularization problem \eqref{optimization_problem_nm} if and only if $\mathcal{B}\mathbf{u}^*$ is a solution of the regularization problem  \eqref{psi_solve_Bu}. 
\end{proposition}
\begin{proof}
Note that the mapping $\mathcal{B}$ provides a bijective correspondence between $\mathbb{R}^n$ and $\mathcal{R}(\mathbf{B})\times\mathbb{R}^{n-r}$. It suffices to verify that for all $\mathbf{u}\in\mathbb{R}^n$  there holds
\begin{equation*}\label{two-objective-function}
\bm{\psi}(\mathbf{u})+\lambda\|\mathbf{B}\mathbf{u}\|_{1}=\bm{\psi}\circ\mathbf{B}'(\mathcal{B}\mathbf{u})+\lambda\|\mathbf{I}'\mathcal{B}\mathbf{u}\|_{1}.
\end{equation*}
By the definition of $\mathcal{B}$, we have that $\mathbf{I}'\mathcal{B}\mathbf{u}=\mathbf{B}\mathbf{u}$. This together with equation \eqref{operator_B_property} confirms the validity of the equation above. 
\end{proof}

We reformulate the regularization problem \eqref{psi_solve_Bu} as an equivalent unconstrained regularization problem for the purpose of characterizing its sparse solutions. 
Set 
$\mathbb{M}:=\mathcal{R}(\mathbf{B})\times\mathbb{R}^{n-r}$
and denote by $\mathbb{M}^{\bot}$ its orthogonal complement. Let $\iota_{\mathbb{M}}:\mathbb{R}^{m+n-r}\rightarrow\mathbb{R}\cup\{+\infty\}$ be the indicator function of the subset $\mathbb{M}$, that is, $\iota_{\mathbb{M}}(\mathbf{x})=0$ if $\mathbf{x}\in \mathbb{M}$, and $+\infty$ otherwise. Using the indicator
function, the constrained  regularization problem \eqref{psi_solve_Bu} is rewritten as the equivalent unconstrained regularization problem 
\begin{equation}\label{psi_solve_Bu_unconstrained}
\min\left\{\bm{\psi}\circ
\mathbf{B}'(\mathbf{w})+\iota_{\mathbb{M}}(\mathbf{w})
+\lambda\|\mathbf{I}'\mathbf{w}\|_1:\mathbf{w}\in\mathbb{R}^{m+n-r}\right\}.
\end{equation}
We present below a characterization of a solution of the regularization problem \eqref{psi_solve_Bu_unconstrained} having sparsity of a certain level. Let $\mathcal{N}(\mathbf{A})$ denote the null space of matrix $\mathbf{A}$. 

\begin{proposition}\label{sparse_equ_mini}
Suppose that $\bm{\psi}$ is a convex function on $\mathbb{R}^n$, $\mathbf{B}$ is a real $m\times n$ matrix with the singular value decomposition \eqref{psi_SVD} and $\mathbf{B}'$ is defined by \eqref{psi_widetilde_B}. Then the regularization problem \eqref{psi_solve_Bu_unconstrained} with $\lambda>0$ has a solution  $\mathbf{w}^*:=\llbracket\mathbf{z}^*,\mathbf{v}^*\rrbracket$ with $\mathbf{z}^*:=\sum_{i\in\mathbb{N}_l}z^*_{k_i}\mathbf{e}_{k_i}\in \Omega_l$ for some $l\in\mathbb{Z}_{m+1}$ and distinct $k_i\in \mathbb{N}_n$, $i\in \mathbb{N}_l$ if and only if there exist $\mathbf{a}\in\partial\bm{\psi}(\mathbf{B}'\mathbf{w}^*)$ and $\mathbf{b}:=[b_j:j\in\mathbb{N}_m]\in\mathcal{N}(\mathbf{B}^{\top})$ such that 
\begin{equation}\label{cond1_B}
\lambda=-\left( (\mathbf{B}_{k_i}')^\top\mathbf{a}+b_{k_i}\right)\mathrm{sign}(z_{k_i}^*),\ i\in\mathbb{N}_l,
\end{equation}
\begin{equation}\label{cond2_B}
\lambda\geq\left|(\mathbf{B}_j')^\top\mathbf{a}+b_j\right|, \ j\in \mathbb{N}_m\setminus  \{k_i:i\in\mathbb{N}_l\},
\end{equation}
\begin{equation}\label{cond3_B}
(\mathbf{B}'_{j})^{\top}\mathbf{a}={0}, \ j\in\mathbb{N}_{m+n-r}\setminus\mathbb{N}_m.
\end{equation}
\end{proposition}
\begin{proof}
According to the Fermat rule and the chain rule of the subdifferential, we get that  $\mathbf{w}^*:=
\llbracket\mathbf{z}^*,\mathbf{v}^*\rrbracket$ with $\mathbf{z}^*:=\sum_{i\in\mathbb{N}_l}z^*_{k_i}\mathbf{e}_{k_i}\in \Omega_l$ is a solution of \eqref{psi_solve_Bu_unconstrained} if and only if 
\begin{equation}\label{eq: 0inV'topnabla}
\mathbf{0}\in(\mathbf{B}')^\top\partial\bm{\psi}(\mathbf{B}'\mathbf{w}^*)+\partial\iota_{\mathbb{M}}(\mathbf{w}^*)+\lambda(\mathbf{I}')^\top\partial \|\cdot\|_1(\mathbf{z}^*).
\end{equation}
It is known that $\partial\iota_{\mathbb{M}}(\mathbf{w})=\mathbb{M}^{\bot}$ for all $\mathbf{w}\in \mathbb{M}$, which together with 
$$
\mathbb{M}^{\bot}=(\mathcal{R}(\mathbf{B}))^{\bot}\times(\mathbb{R}^{n-r})^{\bot}=\mathcal{N}(\mathbf{B}^{\top})\times\{\mathbf{0}\}
$$
further leads to 
$\partial\iota_{\mathbb{M}}(\mathbf{w})=\mathcal{N}(\mathbf{B}^{\top})\times\{\mathbf{0}\}$, for all $\mathbf{w}\in \mathbb{M}$.
By employing the above equation and noting that $\mathbf{w}^*\in \mathbb{M}$, the inclusion relation \eqref{eq: 0inV'topnabla} is equivalent to the existence of $\mathbf{a}\in\partial\bm{\psi}(\mathbf{B}'\mathbf{w}^*)$ and $\mathbf{b}\in\mathcal{N}(\mathbf{B}^{\top})$ satisfying
\begin{equation}\label{0inB'topnablapsiB'w*}
-(\mathbf{B}')^\top\mathbf{a}-\llbracket\mathbf{b},\mathbf{0}\rrbracket \in\lambda(\mathbf{I}')^\top\partial \|\cdot\|_1(\mathbf{z}^*).
\end{equation}
Moreover, noting that $\mathbf{z}^{*}=\sum_{i\in\mathbb{N}_{l}}z_{k_i}^*\mathbf{e}_{k_i}$ with $z_{k_i}^*\in\mathbb{R}\setminus\{0\}$, $i\in\mathbb{N}_l$, we obtain that 
\begin{equation*}\label{subdiff-1norm-z*}
\partial\|\cdot\|_1(\mathbf{z}^*)
=\left\{\mathbf{x}\in\mathbb{R}^m: x_{k_i}=\mathrm{sign}(z_{k_i}^*),i\in\mathbb{N}_l \ \mbox{and}\ |x_j|\leq 1, j\in\mathbb{N}_m\setminus\{k_i:i\in\mathbb{N}_l\}\right\}. 
\end{equation*}
According to the above equation,  the inclusion relation \eqref{0inB'topnablapsiB'w*} may be rewritten equivalently as \eqref{cond1_B}, \eqref{cond2_B} and \eqref{cond3_B}. This proves the desired result.
\end{proof}

Combining Propositions \ref{equi_mini} and \ref{sparse_equ_mini}, we  characterize a solution of the regularization problem \eqref{optimization_problem_nm} having sparsity of a certain level under the transform $\mathbf{B}$.

\begin{theorem}\label{sparse_B}
Suppose that $\bm{\psi}$ is a convex function on $\mathbb{R}^n$ and $\mathbf{B}$ is a real $m\times n$ matrix with the singular value decomposition \eqref{psi_SVD}. Let $\mathbf{B}'$ be defined by \eqref{psi_widetilde_B}. Then the regularization problem \eqref{optimization_problem_nm} with $\lambda>0$ has a solution $\mathbf{u}^{*}\in\mathbb{R}^n$ with $\mathbf{B}\mathbf{u}^*:=\sum_{i\in\mathbb{N}_l}z^*_{k_i}\mathbf{e}_{k_i}\in \Omega_l$ for some $l\in\mathbb{Z}_{m+1}$ if and only if there exist $\mathbf{a}\in\partial\bm{\psi}(\mathbf{u}^*)$ and $\mathbf{b}\in\mathcal{N}(\mathbf{B}^{\top})$ such that \eqref{cond1_B}, \eqref{cond2_B} and \eqref{cond3_B} hold. 
\end{theorem}
\begin{proof}
By Proposition \ref{equi_mini} we conclude that $\mathbf{u}^{*}\in\mathbb{R}^n$ is a solution of \eqref{optimization_problem_nm} and $\mathbf{B}\mathbf{u}^*:=\sum_{i\in\mathbb{N}_l}z^*_{k_i}\mathbf{e}_{k_i}\in \Omega_l$ if and only if $\mathcal{B}\mathbf{u}^*:=
\llbracket\mathbf{z}^*,\mathbf{v}^*\rrbracket$ with $\mathbf{z}^*:=\sum_{i\in\mathbb{N}_l}z^*_{k_i}\mathbf{e}_{k_i}\in \Omega_l$ is a solution of the unconstrained regularization problem \eqref{psi_solve_Bu_unconstrained}. Proposition \ref{sparse_equ_mini} ensures that the latter is equivalent to that there exist $\mathbf{a}\in\partial\bm{\psi}(\mathbf{B}'\mathcal{B}\mathbf{u}^*)$ and $\mathbf{b}\in\mathcal{N}(\mathbf{B}^{\top})$ such that  \eqref{cond1_B}, \eqref{cond2_B} and \eqref{cond3_B} hold. 
By equation \eqref{operator_B_property}, we observe that $\mathbf{B}'\mathcal{B}\mathbf{u}^*=\mathbf{u}^*$, from which the desired result is obtained.
\end{proof}

For the most sparse solution $\mathbf{u}^*$ under the transform (that is, $\mathbf{B}\mathbf{u}^*=\mathbf{0}$),  conditions \eqref{cond1_B}, \eqref{cond2_B} and \eqref{cond3_B} reduce to $\lambda\geq \left\|(\widetilde{\mathbf{B}}'_1)^{\top}\mathbf{a}+\mathbf{b}\right\|_{\infty}$ and $(\widetilde{\mathbf{B}}'_2)^{\top}\mathbf{a}=\mathbf{0}, $ where $\widetilde{\mathbf{B}}'_1$ and $\widetilde{\mathbf{B}}'_2$ denote the matrices composed of the first $m$ columns and the last $n-r$ columns of $\mathbf{B}'$, respectively.

We next describe a special characterization when the transform matrix $\mathbf{B}$ has full row rank, that is, $\mathrm{rank}(\mathbf{B})=m$.

\begin{corollary}\label{sparse_B_full}
Suppose that $\bm{\psi}$ is a convex function on $\mathbb{R}^n$ and $\mathbf{B}$ is a real $m\times n$ matrix having full row rank. Let $\mathbf{B}'$ be defined by \eqref{psi_widetilde_B}. Then the regularization problem \eqref{optimization_problem_nm} with $\lambda>0$ has a solution $\mathbf{u}^{*}$ with $\mathbf{B}\mathbf{u}^*:=\sum_{i\in\mathbb{N}_l}z^*_{k_i}\mathbf{e}_{k_i}\in\Omega_l$ for some $l\in\mathbb{Z}_{m+1}$ if and only if there exists $\mathbf{a}\in\partial\bm{\psi}(\mathbf{u}^*)$ such that 
\begin{equation}\label{cond1_B_full}
\lambda=-(\mathbf{B}'_{k_i})^{\top}\mathbf{a}\, \mathrm{sign}(z^*_{k_i}),\ i\in\mathbb{N}_l,
\end{equation}
\begin{equation}\label{cond2_B_full}
\lambda\geq\left|(\mathbf{B}'_j)^{\top}\mathbf{a}\right|, \ j\in \mathbb{N}_m\setminus  \{k_i:i\in\mathbb{N}_l\},
\end{equation}
\begin{equation}\label{cond3_B_full}
(\mathbf{B}'_{j})^{\top}\mathbf{a}={0}, \ j\in\mathbb{N}_{n}\setminus\mathbb{N}_m.
\end{equation}
\end{corollary}
\begin{proof}
Note that  
$\mathcal{N}(\mathbf{B}^\top)=\left(\mathcal{R}(\mathbf{B})\right)^\bot=\{\mathbf{0}\},$ when $\mathrm{rank}(\mathbf{B})=m$. It follows that  $\mathbf{b}$ in Theorem \ref{sparse_B} is the zero vector.
Thus, the desired result can be obtained directly from Theorem \ref{sparse_B}.
\end{proof}

As an application of Corollary \ref{sparse_B_full}, we specialize it to the $\ell_1$ SVM classification model 
with the hinge loss function. Suppose that  $D:=\{(\mathbf{x}_j,y_j):j\in\mathbb{N}_n\}\subset\mathbb{R}^d\times\{-1,1\}$ is the given data and $K$ is a reproducing kernel on $\mathbb{R}^d$. Set $\mathbf{Y}:=\mathrm{diag}(y_j: j\in\mathbb{N}_n)$ and $\mathbf{K}':=[\mathbf{K} \ \mathbf{1}_n]$ with $\mathbf{K}:=[K(\mathbf{x}_j,\mathbf{x}_k):j,k\in\mathbb{N}_n]$. The $\ell_1$ SVM classification model
with the hinge loss function has the form 
\begin{equation}\label{SVM_hinge_loss}
\min\left\{\bm{\phi}(\mathbf{Y}\mathbf{K}'\mathbf{u})
+\lambda\|\mathbf{B}\mathbf{u}\|_1:\mathbf{u}\in\mathbb{R}^{n+1}\right\},
\end{equation}
where $\bm{\phi}$ is defined by \eqref{sum of hinge loss} and $\mathbf{B}:=[\mathbf{I}_n \ \mathbf{0}]\in\mathbb{R}^{n\times(n+1)}$. By introducing a univariate function $\phi(x):=\max\{1-x,0\}$,  $x\in\mathbb{R}$, we represent $\bm{\phi}$ as 
\begin{equation}\label{sum of hinge loss1} 
\bm{\phi}(\mathbf{x})=\sum_{j\in\mathbb{N}_n}\phi(x_j), \ 
\mathbf{x}:=[x_j:j\in\mathbb{N}_n]\in\mathbb{R}^n.
\end{equation}


\begin{corollary}\label{general-example1}
Suppose that $(\mathbf{x}_j,y_j)\in\mathbb{R}^d\times\{1,-1\}$, $j\in\mathbb{N}_n$, and $K$ is a reproducing kernel on $\mathbb{R}^d$. Let function $\bm{\phi}$, matrices $\mathbf{K}$, $\mathbf{K}^{'}$, $\mathbf{Y}$ and $\mathbf{B}$ be defined as above.  Then the regularization problem \eqref{SVM_hinge_loss} with $\lambda>0$ has a solution $\mathbf{u}^{*}$ with $\mathbf{B}\mathbf{u}^*:=\sum_{i\in\mathbb{N}_l}z^*_{k_i}\mathbf{e}_{k_i}\in \Omega_l$ for some $l\in\mathbb{Z}_{n+1}$ if and only if there exists $\mathbf{c}:=[c_j:j\in\mathbb{N}_n]\in\mathbb{R}^n$ with $c_j\in\partial\phi((\mathbf{Y}\mathbf{K}'\mathbf{u}^*)_j)$, $j\in\mathbb{N}_n$, such that 
\begin{equation}\label{lambda_general-example1_1}
\lambda=-(\mathbf{Y}\mathbf{K}_{k_i})^{\top}\mathbf{c} \mathrm{sign}(z^*_{k_i}),\ i\in\mathbb{N}_l,
\end{equation}
\begin{equation}\label{lambda_general-example1_2}
\lambda\geq\left|(\mathbf{Y}\mathbf{K}_{j})^{\top}\mathbf{c}\right|, \ j\in \mathbb{N}_n\setminus  \{k_i:i\in\mathbb{N}_l\},    
\end{equation}
\begin{equation}\label{lambda_general-example1_3}
\mathbf{y}^{\top}\mathbf{c}={0}.
\end{equation}
\end{corollary}
\begin{proof}
It is clear that the fidelity term $\bm{\psi}$ defined by $\bm{\psi}(\mathbf{u}):=\bm{\phi}(\mathbf{Y}\mathbf{K}'\mathbf{u})$, $\mathbf{u}\in\mathbb{R}^{n+1}$ is convex on $\mathbb{R}^{n+1}$ and the matrix $\mathbf{B}$ has full row rank. Hence, the hypotheses of Corollary \ref{sparse_B_full} with $m,n$ being replaced by $n,n+1$, respectively, are satisfied.  Then by Corollary \ref{sparse_B_full}, the regularization problem \eqref{SVM_hinge_loss} has a solution $\mathbf{u}^{*}$ with $\mathbf{B}\mathbf{u}^*:=\sum_{i\in\mathbb{N}_l}z^*_{k_i}\mathbf{e}_{k_i}\in \Omega_l$ for some $l\in\mathbb{Z}_{n+1}$ if and only if there exists $\mathbf{a}\in\partial\bm{\psi}(\mathbf{u}^*)$ such that 
\eqref{cond1_B_full}, \eqref{cond2_B_full} and \eqref{cond3_B_full} with $m,n$ being replaced by $n,n+1$, respectively, hold.

It remains to verify that in this case \eqref{cond1_B_full}, \eqref{cond2_B_full} and \eqref{cond3_B_full} reduce to \eqref{lambda_general-example1_1}, \eqref{lambda_general-example1_2} and \eqref{lambda_general-example1_3}, respectively. We first describe the subdifferential of the fidelity term $\bm{\psi}$. By the chain rule of the subdifferential, we have that for all $\mathbf{u}\in\mathbb{R}^{n+1}$
\begin{equation}\label{subdiff_psi}
\partial\bm{\psi}(\mathbf{u})=(\mathbf{Y}\mathbf{K}')^\top \partial\bm{\phi}(\mathbf{Y}\mathbf{K}'\mathbf{u}).
\end{equation}
It follows from equation \eqref{sum of hinge loss1} that for all $\mathbf{x}:=[x_j:j\in\mathbb{N}_n]\in\mathbb{R}^n$
\begin{equation}\label{subdiff_phi}
\partial \bm{\phi}(\mathbf{x})=\left\{\mathbf{c}:=[c_j:j\in\mathbb{N}_n]\in\mathbb{R}^n:c_j\in\partial\phi(x_j), j\in\mathbb{N}_n\right\}.
\end{equation}
Substituting equation  \eqref{subdiff_phi} with $\mathbf{x}:=\mathbf{Y}\mathbf{K}'\mathbf{u}$ into equation \eqref{subdiff_psi}, we obtain that 
\begin{equation*}\label{subdiff_psi1}
\partial\bm{\psi}(\mathbf{u})=\{(\mathbf{Y}\mathbf{K}')^\top\mathbf{c}:\mathbf{c}:=[c_j:j\in\mathbb{N}_n]\in\mathbb{R}^n, c_j\in\partial\phi((\mathbf{Y}\mathbf{K}'\mathbf{u})_j),\ j\in\mathbb{N}_n\}.
\end{equation*}
We next represent the matrix $\mathbf{B}'$ defined by \eqref{psi_widetilde_B}. Note that $\mathbf{B}$ has the singular value decomposition  $\mathbf{B}=\mathbf{U}\mathbf{\Lambda}\mathbf{V}^{\top}$ with $\mathbf{U}:=\mathbf{I}_n$, $\mathbf{\Lambda}:=\mathbf{B}$ and $\mathbf{V}:=\mathbf{I}_{n+1}$. Then by definition \eqref{psi_widetilde_B} we get that $\mathbf{B}'=\mathbf{I}_{n+1}$.
Substituting the representations of  $\mathbf{B}'$ and the subdifferential of $\bm{\psi}$
into \eqref{cond1_B_full}, \eqref{cond2_B_full} and \eqref{cond3_B_full} with noting that $(\mathbf{B}'_j)^{\top}(\mathbf{Y}\mathbf{K}')^\top=(\mathbf{Y}\mathbf{K}_j)^{\top}$ for all $j\in\mathbb{N}_n$ and $(\mathbf{B}'_{n+1})^{\top}(\mathbf{Y}\mathbf{K}')^\top=\mathbf{y}^{\top}$, we get the desired conditions \eqref{lambda_general-example1_1}, \eqref{lambda_general-example1_2} and \eqref{lambda_general-example1_3}.
\end{proof}


We consider the $\ell_1$ SVM regression model with the $\epsilon$-insensitive loss function 
\begin{equation}\label{SVM_epsilon_insensitive_loss}
\min\left\{\bm{\phi}_{\mathbf{y},\epsilon}(\mathbf{K}'\mathbf{u})+\lambda\|\mathbf{B}\mathbf{u}\|_1:\mathbf{u}\in\mathbb{R}^{n+1}\right\},
\end{equation}
where $\mathbf{K}^{'}$, $\mathbf{B}$ are defined as in the classification model and $\bm{\phi}_{\mathbf{y},\epsilon}$ is defined by \eqref{sum of epsilon-insensitive loss}. The function  $\bm{\phi}_{\mathbf{y},\epsilon}$ is additively separable with the form $\bm{\phi}_{\mathbf{y},\epsilon}(\mathbf{x})=\sum_{j\in\mathbb{N}_n}\phi_{y_j,\epsilon}(x_j),\   \mathbf{x}:=[x_j:j\in\mathbb{N}_n]\in\mathbb{R}^n$,
where 
$\phi_{y,\epsilon}(t):=\mathrm{max}\{|y-t|-\epsilon,0\},\ t\in\mathbb{R}$.
By arguments similar to those used in the proof of Corollary \ref{general-example1}, we get the following characterization of the sparsity of the solution of the regularization problem  \eqref{SVM_epsilon_insensitive_loss} under the transform $\mathbf{B}$.

\begin{corollary}\label{general-example2}
Suppose that $(\mathbf{x}_j,y_j)\in\mathbb{R}^d\times\mathbb{R}$, $j\in\mathbb{N}_n$, and $K$ is a reproducing kernel on $\mathbb{R}^d$. Let function $\bm{\phi}_{\mathbf{y},\epsilon}$, matrices $\mathbf{K}$,  $\mathbf{K}^{'}$ and $\mathbf{B}$ be defined as above. Then the regularization problem \eqref{SVM_epsilon_insensitive_loss} with $\lambda>0$ has a solution $\mathbf{u}^{*}$ with $\mathbf{B}\mathbf{u}^*:=\sum_{i\in\mathbb{N}_l}z^*_{k_i}\mathbf{e}_{k_i}\in \Omega_l$ for some $l\in\mathbb{Z}_{n+1}$ if and only if there exists $\mathbf{c}:=[c_j:j\in\mathbb{N}_n]\in\mathbb{R}^n$ with $c_j\in\partial\phi_{y_j,\epsilon}((\mathbf{K}'\mathbf{u}^*)_j)$, $j\in\mathbb{N}_n$,  such that 
\begin{equation*}\label{lambda_general-example2_1}
\lambda=-(\mathbf{K}_{k_i})^{\top}\mathbf{c} \mathrm{sign}(z^*_{k_i}),\ i\in\mathbb{N}_l,
\end{equation*}
\begin{equation}\label{lambda_general-example2_2}
\lambda\geq\left|(\mathbf{K}_j)^{\top}\mathbf{c}\right|, \ j\in \mathbb{N}_n\setminus  \{k_i:i\in\mathbb{N}_l\},
\end{equation}
\begin{equation*}\label{lambda_general-example2_3}
\mathbf{1}_n^{\top}\mathbf{c}={0}.
\end{equation*}
\end{corollary}

We now return to the general regularization problem \eqref{optimization_problem_nm}.
When the fidelity term $\bm{\psi}$ is differentiable, Theorem \ref{sparse_B}
has the following simple form.

\begin{corollary}\label{choice_sparsity_smooth_B}
Suppose that $\bm{\psi}$ is a differentiable and convex function on $\mathbb{R}^n$ and $\mathbf{B}$ is a real $m\times n$ matrix with the singular value decomposition \eqref{psi_SVD}. Let $\mathbf{B}'$ be defined by \eqref{psi_widetilde_B}. Then the regularization problem \eqref{optimization_problem_nm} with $\lambda>0$ has a solution $\mathbf{u}^{*}$ with $\mathbf{B}\mathbf{u}^*:=\sum_{i\in\mathbb{N}_l}z^*_{k_i}\mathbf{e}_{k_i}\in\Omega_l$ for some $l\in\mathbb{Z}_{m+1}$ if and only if there exists  $\mathbf{b}\in\mathcal{N}(\mathbf{B}^{\top})$ such that
\begin{equation}\label{cond1_B_diff}
\lambda=-((\mathbf{B}'_{k_i})^{\top}\nabla\bm{\psi}(\mathbf{u}^*)+b_{k_i}) \mathrm{sign}(z^*_{k_i}),\ i\in\mathbb{N}_l,
\end{equation}
\begin{equation}\label{cond2_B_diff}
\lambda\geq\left|(\mathbf{B}'_j)^{\top}\nabla\bm{\psi}(\mathbf{u}^*)+b_j\right|, \ j\in \mathbb{N}_m\setminus  \{k_i:i\in\mathbb{N}_l\},
\end{equation}
\begin{equation}\label{cond3_B_diff}
(\mathbf{B}'_{j})^{\top}\nabla\bm{\psi}
(\mathbf{u}^*)={0}, \ j\in\mathbb{N}_{m+n-r}\setminus\mathbb{N}_m.
\end{equation}
\end{corollary}
\begin{proof}
Since $\bm{\psi}$ is differentiable, its subdifferential at $\mathbf{u}^*$ is the singleton $\nabla\bm{\psi}(\mathbf{u}^*)$. Substituting this into \eqref{cond1_B}, \eqref{cond2_B} and \eqref{cond3_B}, we obtain \eqref{cond1_B_diff}, \eqref{cond2_B_diff} and \eqref{cond3_B_diff}, respectively. 
\end{proof}

The next result concerns the case that the transform matrix $\mathbf{B}$ has full row rank. 

\begin{corollary}\label{choice_sparsity_mooth_B_full}
Suppose that $\bm{\psi}$ is a differentiable and convex function on $\mathbb{R}^n$ and $\mathbf{B}$ is a real $m\times n$ matrix having full row rank. Let $\mathbf{B}'$ be defined by \eqref{psi_widetilde_B}. Then the regularization problem \eqref{optimization_problem_nm} with $\lambda>0$ has a solution $\mathbf{u}^{*}$ with $\mathbf{B}\mathbf{u}^*:=\sum_{i\in\mathbb{N}_l}z^*_{k_i}\mathbf{e}_{k_i}\in\Omega_l$ for some $l\in\mathbb{Z}_{m+1}$ if and only if 
\begin{equation}\label{cond1_B_diff_full}
\lambda=-(\mathbf{B}'_{k_i})^{\top}\nabla\bm{\psi}(\mathbf{u}^*) \mathrm{sign}(z^*_{k_i}),\ i\in\mathbb{N}_l,
\end{equation}
\begin{equation}\label{cond2_B_diff_full}
\lambda\geq\left|(\mathbf{B}'_j)^{\top}\nabla\bm{\psi}(\mathbf{u}^*)\right|, \ j\in \mathbb{N}_m\setminus  \{k_i:i\in\mathbb{N}_l\},
\end{equation}
\begin{equation}\label{cond3_B_diff_full}
(\mathbf{B}'_{j})^{\top}\nabla\bm{\psi}
(\mathbf{u}^*)={0}, \ j\in\mathbb{N}_{n}\setminus\mathbb{N}_m.
\end{equation}
\end{corollary}

\begin{proof}
Conditions \eqref{cond1_B_diff_full}, \eqref{cond2_B_diff_full} and \eqref{cond3_B_diff_full} can be obtained from Corollary \ref{sparse_B_full} by noting that the subdifferential of $\bm{\psi}$ at $\mathbf{u}^*$ is the singleton $\nabla\bm{\psi}(\mathbf{u}^*)$. 
\end{proof}

In the remaining part of this section, we apply Corollary \ref{choice_sparsity_mooth_B_full} to three specific models. We first consider the total-variation signal denoising model 
\begin{equation}\label{total-variation-signal-denoising}
\min \left\{\frac{1}{2}\|\mathbf{u}-\mathbf{x}\|_2^2+\lambda\|\mathbf{D}^{(1)}\mathbf{u}\|_1:\mathbf{u}\in\mathbb{R}^n\right\}.
\end{equation}
where $\mathbf{D}^{(1)}$ is the $(n-1)\times n$ first order difference matrix. Suppose that $\mathbf{D}^{(1)}$ has the singular value decomposition $\mathbf{D}^{(1)}=\mathbf{U}\Lambda\mathbf{V}^{\top}$ and $\mathbf{D}^{(1)'}:=\mathbf{V}\Lambda'\mathbf{U}'$ is defined by \eqref{psi_widetilde_B}. It follows from \cite{shepard1990singular} that $\mathbf{V}_{n}=\frac{\sqrt{n}}{n}\mathbf{1}_n$, which together with the definition of $\mathbf{D}^{(1)'}$ leads to $\mathbf{D}^{(1)'}_{n}=\frac{\sqrt{n}}{n}\mathbf{1}_n$.

\begin{corollary}\label{general-smooth-example0}
Let $\mathbf{D}^{(1)}$ be the  $(n-1)\times n$ first order difference matrix and $\mathbf{D}^{(1)'}$ be defined as above. Then the regularization problem \eqref{total-variation-signal-denoising} with $\lambda>0$ has a solution $\mathbf{u}^{*}$ with $\mathbf{D}^{(1)}\mathbf{u}^*:=\sum_{i\in\mathbb{N}_l}z^*_{k_i}\mathbf{e}_{k_i}\in\Omega_l$ for some $l\in\mathbb{Z}_{n}$ if and only if 
\begin{equation}\label{lambda_general-smooth-example01}
\lambda=(\mathbf{D}^{(1)'}_{k_i})^{\top}(\mathbf{x}-\mathbf{u}^{*}) \mathrm{sign}(z^*_{k_i}),\ i\in\mathbb{N}_l,
\end{equation}
\begin{equation}\label{lambda_general-smooth-example02}
\lambda\geq\left|(\mathbf{D}^{(1)'}_j)^{\top}(\mathbf{u}^{*}-\mathbf{x})\right|, \ j\in \mathbb{N}_{n-1}\setminus  \{k_i:i\in\mathbb{N}_l\},
\end{equation}
\begin{equation}\label{lambda_general-smooth-example03}
\mathbf{1}_n^{\top}(\mathbf{u}^{*}-\mathbf{x})={0}.
\end{equation}
\end{corollary}
\begin{proof}
Since the fidelity term $\bm{\psi}:=\frac{1}{2}\|\mathbf{u}-\mathbf{x}\|_2^2$ is differentiable and convex and the matrix $\mathbf{D}^{(1)}$ has full row rank, Corollary \ref{choice_sparsity_mooth_B_full} guarantees that the total-variation signal denoising model has a solution $\mathbf{u}^*$ with $\mathbf{D}^{(1)}\mathbf{u}^*:=\sum_{j\in\mathbb{N}_l}z_{k_i}^*\mathbf{e}_{k_i}\in \Omega_l$ for some $l\in\mathbb{Z}_{n}$ if and only if there hold \eqref{cond1_B_diff_full}, \eqref{cond2_B_diff_full} and \eqref{cond3_B_diff_full} with $\mathbf{B}'$ and $m$ being replaced by $\mathbf{D}^{(1)'}$ and $n-1$, respectively.  Substituting  $\nabla\bm{\psi}(\mathbf{u}^*)=\mathbf{u}^*-\mathbf{x}$ and $\mathbf{D}^{(1)'}_n=\frac{\sqrt{n}}{n}\mathbf{1}_n$ into \eqref{cond1_B_diff_full}, \eqref{cond2_B_diff_full} and \eqref{cond3_B_diff_full}, we get their equivalent representations as \eqref{lambda_general-smooth-example01}, \eqref{lambda_general-smooth-example02} and \eqref{lambda_general-smooth-example03}.
\end{proof}

A parameter choice strategy for the most sparse solution under the transform $\mathbf{D}^{(1)}$ is provided in following  remark. We denote by $\widetilde{\mathbf{D}}^{(1)'}$ the matrix composed of the first $n-1$ columns of $\mathbf{D}^{(1)'}$.
\begin{remark}\label{general-smooth-example0_most_sparse}
The regularization problem \eqref{total-variation-signal-denoising} with $\lambda>0$ has a solution 
$\mathbf{u}^*$ with $\mathbf{D}^{(1)}\mathbf{u}^*=\mathbf{0}$ if and only if
\begin{equation}\label{lambda_general-smooth-example0_most_sparse}
\lambda\geq\left\| (\widetilde{\mathbf{D}}^{(1)'})^{\top}\mathbf{x}\right\|_{\infty}.
\end{equation}
Moreover,
the solution $\mathbf{u}^*$ with $\mathbf{D}^{(1)}\mathbf{u}^*=\mathbf{0}$ has the form  $\mathbf{u}^*:=\frac{1}{n}\mathbf{1}_n^{\top}\mathbf{x}\mathbf{1}_n$.
\end{remark}
\begin{proof}
It follows from Corollary \ref{general-smooth-example0} with $l=0$ that the total-variation signal denoising model has a solution $\mathbf{u}^*$ satisfying $\mathbf{D}^{(1)}\mathbf{u}^*=\mathbf{0}$ if and only if there hold
 \begin{equation}\label{B0_condition1}
 \lambda\geq \left\|(\widetilde{\mathbf{D}}^{(1)'})^{\top}(\mathbf{u}^*-\mathbf{x})\right\|_{\infty}
 \ \mbox{and}\
 \mathbf{1}_n^{\top}
 (\mathbf{u}^*-\mathbf{x})=0.
 \end{equation}
Suppose that $\mathbf{D}^{(1)}\mathbf{u}^*=\mathbf{0}$. We obtain $\mathbf{u}^{*}$ by solving two eqautions $\mathbf{D}^{(1)}\mathbf{u}^*=\mathbf{0}$ and $\mathbf{1}_n^{\top} (\mathbf{u}^*-\mathbf{x})=0$. By Lemma \ref{lemma: change variable}, the vector $\mathbf{u}^*$ satisfying the first equation can be represented as $\mathbf{u}^*=\mathbf{D}^{(1)'}\llbracket\mathbf{0},v^*\rrbracket$ for some
$v^*\in\mathbb{R}$, which together with  $\mathbf{D}^{(1)'}_n=\frac{\sqrt{n}}{n}\mathbf{1}_n$,  leads to $\mathbf{u}^*=\frac{v^*\sqrt{n}}{n}\mathbf{1}_n$.
Substituting this representation into the second equation yields that $v^*=\frac{\sqrt{n}}{n}\mathbf{1}_n^{\top}\mathbf{x}$, which further leads to  $\mathbf{u}^*=\frac{1}{n}(\mathbf{1}_n^{\top}\mathbf{x})\mathbf{1}_n$. By employing the above representation of $\mathbf{u}^*$, we rewrite the inequality in condition \eqref{B0_condition1} as $\lambda\geq \left\|\frac{1}{n}(\mathbf{1}_n^{\top}\mathbf{x})(\widetilde{\mathbf{D}}^{(1)'})^{\top}\mathbf{1}_n-(\widetilde{\mathbf{D}}^{(1)'})^{\top}\mathbf{x}\right\|_{\infty}.$ 
It suffices to show  $(\widetilde{\mathbf{D}}^{(1)'})^{\top}\mathbf{1}_n=\mathbf{0}$. 
By the definition of $\mathbf{D}^{(1)'}$, we have  $(\mathbf{D}^{(1)'})^{\top}\mathbf{V}_n=(\mathbf{U}')^{\top}\Lambda'\mathbf{V}^{\top}\mathbf{V}_n.$
Substituting $\mathbf{V}^{\top}\mathbf{V}_n=\llbracket\mathbf{0},1\rrbracket$ into the above equation, we get $(\mathbf{D}^{(1)'})^{\top}\mathbf{V}_n=(\mathbf{U}')^{\top}\Lambda'\llbracket\mathbf{0},1\rrbracket$, which further yields $(\mathbf{D}^{(1)'})^{\top}\mathbf{V}_n=\llbracket\mathbf{0},1\rrbracket$. That is, $(\widetilde{\mathbf{D}}^{(1)'})^{\top}\mathbf{V}_n=\mathbf{0}.$
Noting $\mathbf{V}_n=\frac{\sqrt{n}}{n}\mathbf{1}_n$, we obtain $(\widetilde{\mathbf{D}}^{(1)'})^{\top}\mathbf{1}_n=\mathbf{0}$. Thus, we rewrite the inequality in condition \eqref{B0_condition1} as inequality \eqref{lambda_general-smooth-example0_most_sparse}.

Conversely, suppose that condition \eqref{lambda_general-smooth-example0_most_sparse} holds. By setting $\mathbf{u}^*:=\frac{1}{n}(\mathbf{1}_n^{\top}\mathbf{x})\mathbf{1}_n$, we conclude that condition \eqref{B0_condition1} holds. That is, $\mathbf{u}^*$ is a solution  of the total-variation signal denoising model with $\mathbf{D}^{(1)}\mathbf{u^*}=\mathbf{0}$.
\end{proof}

The second one is the $\ell_1$ SVM models for classification/regression with the squared loss function defined by \eqref{square_loss}. As pointed out in section 2, these models can be formulated as   
\begin{equation}\label{SVM_square_loss}
\min
\left\{\frac{1}{2}\|\mathbf{K}'\mathbf{u}-\mathbf{y}\|_{2}^{2}+\lambda\|\mathbf{B}\mathbf{u}\|_1:\mathbf{u}\in\mathbb{R}^{n+1}\right\},
\end{equation}
where $\mathbf{K}'$ is the augmented kernel matrix and $\mathbf{B}:=[\mathbf{I}_n \ \mathbf{0}]\in\mathbb{R}^{n\times(n+1)}.$

\begin{corollary}\label{general-smooth-example1}
Suppose that $(\mathbf{x}_j,y_j)\in\mathbb{R}^d\times\mathbb{R}$, $j\in\mathbb{N}_n$, and $K$ is a reproducing kernel on $\mathbb{R}^d$. Let matrices $\mathbf{K}'$ and $\mathbf{B}$ be defined as above and set $\mathbf{y}:=[y_j:j\in\mathbb{N}_n]$. Then the regularization problem \eqref{SVM_square_loss} with $\lambda>0$ has a solution $\mathbf{u}^*$ with $\mathbf{B}\mathbf{u}^*:=\sum_{i\in\mathbb{N}_l}z^*_{k_i}\mathbf{e}_{k_i}\in \Omega_l$ for some $l\in\mathbb{Z}_{n+1}$ if and only if there hold
\begin{equation}\label{lambda_general-smooth-example11}
\lambda=(\mathbf{K}_{k_i})^{\top}\left(\mathbf{y}-\mathbf{K}'\mathbf{u}^{*}\right)\mathrm{sign}({z}_{k_i}^{*}), \  
i\in\mathbb{N}_l,
\end{equation}
\begin{equation}\label{lambda_general-smooth-example12}
\lambda\geq\left|(\mathbf{K}_j)^{\top}\left(\mathbf{K}'\mathbf{u}^{*}-\mathbf{y}\right)\right|, 
\  j\in \mathbb{N}_n\setminus \{k_i:i\in\mathbb{N}_l\},
\end{equation}
\begin{equation}\label{lambda_general-smooth-example13}
\mathbf{1}_{n}^{\top}\left(\mathbf{K}'\mathbf{u}^{*}-\mathbf{y}\right)=0.   
\end{equation}
\end{corollary}
\begin{proof}
We prove this corollary by employing Corollary \ref{choice_sparsity_mooth_B_full}.
Note that the fidelity term $\bm{\psi}$ defined by $\bm{\psi}(\mathbf{u}):=\frac{1}{2}\|\mathbf{K}'\mathbf{u}-\mathbf{y}\|_2^2$, $ \mathbf{u}\in\mathbb{R}^{n+1}$, is differentiable and convex and the matrix $\mathbf{B}$ has full row rank.  By Corollary \ref{choice_sparsity_mooth_B_full} we have that the regularization problem \eqref{SVM_square_loss} has a solution $\mathbf{u}^*$ with $\mathbf{B}\mathbf{u}^*:=\sum_{i\in\mathbb{N}_l}z^*_{k_i}\mathbf{e}_{k_i}\in \Omega_l$ for some $l\in\mathbb{Z}_{n+1}$ if and only if \eqref{cond1_B_diff_full}, \eqref{cond2_B_diff_full} and \eqref{cond3_B_diff_full} hold with $m, n$ being replaced by $n, n+1$, respectively. According to the definition of $\bm{\psi}$, the gradient of $\bm{\psi}$ at $\mathbf{u}^*$ has the form 
$\nabla\bm{\psi}(\mathbf{u}^*)=(\mathbf{K}')^\top\left(\mathbf{K}'\mathbf{u}^{*}-\mathbf{y}\right).$
Substituting the above equation into \eqref{cond1_B_diff_full}, \eqref{cond2_B_diff_full} and \eqref{cond3_B_diff_full}, with noting that $(\mathbf{B}'_j)^{\top}(\mathbf{K}')^\top=(\mathbf{K}_j)^{\top}$, $j\in\mathbb{N}_n$, and $(\mathbf{B}'_{n+1})^{\top}(\mathbf{K}')^\top=\mathbf{1}_{n}^{\top}$, conditions  \eqref{cond1_B_diff_full}, \eqref{cond2_B_diff_full} and \eqref{cond3_B_diff_full} can be represented as \eqref{lambda_general-smooth-example11}, \eqref{lambda_general-smooth-example12} and \eqref{lambda_general-smooth-example13}, respectively.
\end{proof}

When the solution has the most sparsity under the transform $\mathbf{B}$, the characterization stated in Corollary \ref{general-smooth-example1} reduces to a simple form. 
\begin{remark}\label{general-smooth-example1_most_sparse}
The regularization problem  \eqref{SVM_square_loss} with $\lambda>0$ has a solution $\mathbf{u}^*$ with $\mathbf{B}\mathbf{u^*}=\mathbf{0}$ if and only if 
\begin{equation}\label{lambda_general-smooth-example1_most_sparse}
\lambda\geq\left\|\mathbf{K}^{\top}\left(\frac{1}{n}\mathbf{1}_{n}^{\top}\mathbf{y}\mathbf{1}_{n}-\mathbf{y}\right)\right\|_{\infty}. 
\end{equation}
Moreover, the solution $\mathbf{u}^*$ with $\mathbf{B}\mathbf{u^*}=\mathbf{0}$ has the form  $\mathbf{u}^*:=\llbracket\mathbf{0},\frac{1}{n}\mathbf{1}_{n}^{\top}\mathbf{y}\rrbracket$.
\end{remark}
\begin{proof}
Corollary \ref{general-smooth-example1} with $l=0$ shows that the regularization problem \eqref{SVM_square_loss} has a solution $\mathbf{u}^*$ with $\mathbf{B}\mathbf{u^*}=\mathbf{0}$ if and only if there hold
\begin{equation}\label{lambda_general-smooth-example1_most_sparse1}
\lambda\geq\|\mathbf{K}^{\top}(\mathbf{K}'\mathbf{u}^{*}-\mathbf{y})\|_{\infty}\ \mbox{and}\ 
\mathbf{1}_{n}^{\top}\left(\mathbf{K}'\mathbf{u}^{*}-\mathbf{y}\right)=0. 
\end{equation}
On one hand, if $\mathbf{B}\mathbf{u}^*=\mathbf{0}$, then we get that  $\mathbf{u}^*=\llbracket\mathbf{0},b^*\rrbracket$ for some $b^*\in\mathbb{R}$, which together with the equality in condition \eqref{lambda_general-smooth-example1_most_sparse1} leads to $b^*=\frac{1}{n}\mathbf{1}_{n}^{\top}\mathbf{y}$. That is,  $\mathbf{u}^*=\llbracket\mathbf{0},\frac{1}{n}\mathbf{1}_{n}^{\top}\mathbf{y}\rrbracket$. Substituting this representation of  $\mathbf{u}^*$ into the inequality in condition \eqref{lambda_general-smooth-example1_most_sparse1}, we obtain inequality \eqref{lambda_general-smooth-example1_most_sparse}. On the other hand, if inequality \eqref{lambda_general-smooth-example1_most_sparse} holds, by setting $\mathbf{u}^*:=\llbracket\mathbf{0},\frac{1}{n}\mathbf{1}_{n}^{\top}\mathbf{y}\rrbracket$, we conclude that condition \eqref{lambda_general-smooth-example1_most_sparse1} holds. That is, $\mathbf{u}^*$ is a solution with $\mathbf{B}\mathbf{u^*}=\mathbf{0}$ of the regularization problem \eqref{SVM_square_loss}.
\end{proof}

The third example concerns the $\ell_1$-regularized logistic regression model.  Associated with given data $D:=\{(\mathbf{x}_j,y_j): j\in\mathbb{N}_n\}$, we set $\mathbf{X}:=[\mathbf{x}_j:j\in \mathbb{N}_{n}]^{\top}$, augmented to  $\mathbf{X}':=[\mathbf{X}\ \mathbf{1}_{n}]$,  $\mathbf{Y}:=\mathrm{diag}(y_j: j\in\mathbb{N}_n)$ and $\mathbf{B}:=[\mathbf{I}_d\ \mathbf{0}]\in\mathbb{R}^{d\times{(d+1)}}$. The $\ell_1$-regularized logistic regression model can be represented as
\begin{equation}\label{rewritten_l1_LRM}
\mathrm{min}\left\{\bm{\phi}(\mathbf{Y}\mathbf{X}'\mathbf{u})
+\lambda\|\mathbf{Bu}\|_1
:\mathbf{u}\in\mathbb{R}^{d+1}\right\},
\end{equation}
where $\bm{\phi}$ is defined by \eqref{logistic_regression}.


\begin{corollary}\label{general-smooth-example2}
Suppose that $(\mathbf{x}_j,y_j)\in\mathbb{R}^d\times\{-1,1\}, j\in\mathbb{N}_n$. Let function $\phi$, matrices $\mathbf{Y}$, $\mathbf{X}'$ and $\mathbf{B}$ be defined as above. Set  $\mathbf{y}:=[y_j:j\in\mathbb{N}_n]$. Then the regularization problem \eqref{rewritten_l1_LRM} with $\lambda>0$ has a solution $\mathbf{u}^*$ and $\mathbf{B}\mathbf{u}^*:=\sum_{i\in\mathbb{N}_l}z_{k_i}^*\mathbf{e}_{k_i}\in \Omega_l$ for some $l\in\mathbb{Z}_{d+1}$ if and only if there hold
\begin{equation}\label{lambda_general-smooth-example21}
\lambda=\frac{1}{n}(\mathbf{Y}\mathbf{X}_{k_i})^{\top}\mathbf{c}_{\mathbf{u}^*}\mathrm{sign}(z^*_{k_i}),\ i\in\mathbb{N}_l,
\end{equation}
\begin{equation}\label{lambda_general-smooth-example22}
\lambda\geq\frac{1}{n}\left|(\mathbf{Y}\mathbf{X}_{j})^{\top}\mathbf{c}_{\mathbf{u}^*}\right|,\ j\in\mathbb{N}_d\setminus\{k_i:i\in\mathbb{N}_l\},
\end{equation}
\begin{equation}\label{lambda_general-smooth-example23}
\mathbf{y}^{\top}\mathbf{c}_{\mathbf{u}^*}=0,
\end{equation}
where $\mathbf{c}_{\mathbf{u}^*}:=[\left(1+\mathrm{exp}((\mathbf{Y}\mathbf{X}'\mathbf{u}^*)_j)\right)^{-1}:j\in\mathbb{N}_{n}]\in\mathbb{R}^n$.
\end{corollary}
\begin{proof}
Corollary \ref{choice_sparsity_mooth_B_full} ensures that regularization problem \eqref{rewritten_l1_LRM} has a solution $\mathbf{u}^*$ with $\mathbf{B}\mathbf{u}^*:=\sum_{j\in\mathbb{N}_l}z_{k_i}^*\mathbf{e}_{k_i}\in \Omega_l$ for some $l\in\mathbb{Z}_{d+1}$ if and only if \eqref{cond1_B_diff_full}, \eqref{cond2_B_diff_full} and \eqref{cond3_B_diff_full} hold with $m:=d$, $n:=d+1$. By the chain rule of the gradient, we obtain that 
$
\nabla\bm{\psi}(\mathbf{u}^*)
=(\mathbf{Y}\mathbf{X}')^{\top}\nabla\bm{\phi}
(\mathbf{Y}\mathbf{X}'\mathbf{u}^*)=-\frac{1}{n}(\mathbf{Y}\mathbf{X}')^{\top}
\mathbf{c}_{\mathbf{u}^*}.    
$
As pointed out in the proof of Corollary \ref{general-example1}, we have that $\mathbf{B}'=\mathbf{I}_{d+1}$ and thus,  $(\mathbf{B}_j')^{\top}(\mathbf{Y}\mathbf{X}')^{\top}=(\mathbf{Y}\mathbf{X}_j)^{\top}$, $j\in\mathbb{N}_d$ and $(\mathbf{B}_{d+1}')^{\top}(\mathbf{Y}\mathbf{X}')^{\top} =\mathbf{y}^{\top}$. 
Substituting these equations into \eqref{cond1_B_diff_full}, \eqref{cond2_B_diff_full} and \eqref{cond3_B_diff_full} leads to \eqref{lambda_general-smooth-example21}, \eqref{lambda_general-smooth-example22} and \eqref{lambda_general-smooth-example23}.
\end{proof}

When the most sparse solution under the transform $\mathbf{B}$ is desired, we have the parameter choice strategy described in the next remark. We denote by $n_{+}$ and  $n_{-}$ the numbers of data 
with output $y_j=1$ and $y_j=-1$, respectively and set $\mathbf{c}:=[\left(1+(n_{+}/n_{-})^{y_j}\right)^{-1}:j\in\mathbb{N}_{n}]\in\mathbb{R}^n$.

\begin{remark}\label{general-smooth-example2_most_sparse}
The regularization problem  \eqref{rewritten_l1_LRM} with $\lambda>0$ has a solution 
$\mathbf{u}^*$ with $\mathbf{B}\mathbf{u}^*=\mathbf{0}$ if and only if
\begin{equation}\label{lambda_general-smooth-example2_most_sparse}
\lambda\geq\frac{1}{n}\left\|(\mathbf{Y}\mathbf{X})^{\top}\mathbf{c}\right\|_{\infty}.
\end{equation}
Moreover,
the solution $\mathbf{u}^*$ with $\mathbf{B}\mathbf{u^*}=\mathbf{0}$ has the form  $\mathbf{u}^*:=\llbracket\mathbf{0},\mathrm{ln}(n_{+}/n_{-})\rrbracket$.
\end{remark}
\begin{proof}
By employing Corollary \ref{general-smooth-example2} with $l=0$, we get that the regularization problem \eqref{rewritten_l1_LRM}
has a solution $\mathbf{u}^*$ with $\mathbf{B}\mathbf{u}^*=\mathbf{0}$ if and only if 
\begin{equation}\label{lambda_general-smooth-example2_most_sparse1}
\lambda\geq\frac{1}{n}\left\|(\mathbf{Y}\mathbf{X})^{\top}\mathbf{c}_{\mathbf{u}^*}\right\|_{\infty}
\ \mbox{and}\
\mathbf{y}^{\top}\mathbf{c}_{\mathbf{u}^*}=0.
\end{equation}
If $\mathbf{B}\mathbf{u}^*=\mathbf{0}$, then we have that  $\mathbf{u}^*=\llbracket\mathbf{0},b^*\rrbracket$ for some $b^*\in\mathbb{R}$. Substituting $\mathbf{u}^*=\llbracket\mathbf{0},b^*\rrbracket$ into the equality in \eqref{lambda_general-smooth-example2_most_sparse1} and noting that  $\mathbf{Y}\mathbf{X}'\llbracket\mathbf{0},b^*\rrbracket=b^*\mathbf{y}$, we obtain that $n_{+}(1+\exp(b^{*}))^{-1}-n_{-}(1+\exp(-b^{*}))^{-1}=0$, which further yields that $b^*=\mathrm{ln}(n_{+}/n_{-})$. Hence, we get that  $\mathbf{u}^*:=\llbracket\mathbf{0},\mathrm{ln}(n_{+}/n_{-})\rrbracket$. This together with the inequality in \eqref{lambda_general-smooth-example2_most_sparse1} leads to inequality \eqref{lambda_general-smooth-example2_most_sparse}. Conversely, if inequality \eqref{lambda_general-smooth-example2_most_sparse} holds, we obtain that \eqref{lambda_general-smooth-example2_most_sparse1} holds by setting $\mathbf{u}^*:=\llbracket\mathbf{0},\mathrm{ln}(n_{+}/n_{-})\rrbracket$. This yields that $\mathbf{u}^*$ is a solution of the regularization problem \eqref{rewritten_l1_LRM} and $\mathbf{B}\mathbf{u^*}=\mathbf{0}$.
\end{proof}

We comment that the  characterizations about sparsity of a solution of  \eqref{rewritten_l1_LRM} under the transform $\mathbf{B}$, stated in Corollary \ref{general-smooth-example2} and Remark \ref{general-smooth-example2_most_sparse}, were established in \cite{koh2007interior}.

\section{Parameter Choices for Alleviating the Ill-Posedness and Promoting  Sparsity of the Regularized Solutions}

As we pointed out earlier, the purpose of imposing the $\ell_1$ regularization is two-folds: alleviating the ill-posedness and promoting sparsity of a regularized solution. In this section, we demonstrate how the regularization parameter $\lambda$ can be chosen to achieve both of these by considering a Lasso regularized model. 

We consider the Lasso regularized prediction model. Specifically, we aim at a prediction $\mathbf{u}\in \mathbb{R}^n$ from a given response vector $\mathbf{x}\in\mathbb{R}^p$, via the equation $\mathbf{A}\mathbf{u}=\mathbf{x}$. Here, we assume that the predictor matrix  $\mathbf{A}\in\mathbb{R}^{p\times n}$ is  $\mathcal{S}$-block separable, that is, it satisfies condition \eqref{suff_nece_Lasso_block_separable} with respect to the partition $\mathcal{S}:=\left\{S_{1},S_{2},\ldots, S_{d}\right\}$ of the set $\mathbb{N}_n$. Suppose that the response vector $\mathbf{x}$ contains noise. That is, instead of  $\mathbf{x}$, we obtain a noisy response $\mathbf{x}^{\delta}\in\mathbb{R}^p$ with a given noise level $\delta$, which satisfies $\|\mathbf{x}^{\delta}-\mathbf{x}\|_2\leq \delta$. We employ the following Lasso regularized model \eqref{lasso} to recover $\mathbf{u}$ from $\mathbf{x}^\delta$
\begin{equation}\label{lasso_noise_data}
\min \left\{\frac{1}{2}\|\mathbf{Au}-\mathbf{x}^\delta\|_2^2+\lambda\|\mathbf{u}\|_1:\mathbf{u}\in\mathbb{R}^n\right\}.
\end{equation}
We are interested in choices of the regularization parameter $\lambda$ that balances the error of the regularized solution $\mathbf{u}_\lambda^\delta$ and its block sparsity. Here, the error is compared to the minimal norm solution $\tilde{\mathbf{u}}$ of the prediction problem which is defined by
\begin{equation}\label{mini_norm}
\tilde{\mathbf{u}}:=\argmin\{\|\mathbf{u}\|_1:\mathbf{A}\mathbf{u}=\mathbf{x},\mathbf{u}\in\mathbb{R}^n\}.
\end{equation}
The minimal norm problem itself is a recent research topic of great interest \cite{Cai2010,Cheng2021,Gilbert2017,Li-Micchelli-Xu2020,Zhang2015necessary}.
For the purpose of estimating the error between $\mathbf{u}_\lambda^\delta$ and $\tilde{\mathbf{u}}$, it is desirable to require the minimal norm problem \eqref{mini_norm} has a unique solution. We briefly review the uniqueness result of  $\tilde{\mathbf{u}}$ presented recently in  \cite{Gilbert2017,Zhang2015necessary}. 
To this end, we denote by $J$ the support of $\tilde{\mathbf{u}}$ and by $J^c$ the complement of $J$ in $\mathbb{N}_n$. We then introduce vector $\mathbf{v}:=[{\rm sign}(\tilde{u}_j):j\in J]$, and two matrices $\mathbf{A}':=[\mathbf{A}_j:j\in J]$ and $\mathbf{A}'':=[\mathbf{A}_j:j\in J^c]$. It is known from \cite{Gilbert2017,Zhang2015necessary} that the minimal norm problem \eqref{mini_norm} has a unique solution $\tilde{\mathbf{u}}$ if and only if $\mathbf{A}\tilde{\mathbf{u}}=\mathbf{x}$, $\mathbf{A}'$ has full column rank and there exists $\mathbf{y}\in\mathbb{R}^p$ such that $(\mathbf{A}')^{\top}\mathbf{y}=\mathbf{v}$ and $\|(\mathbf{A}'')^{\top}\|_{\infty}<1$. In this section, we simply assume that the minimal norm problem \eqref{mini_norm} has a unique solution.

Below, we state an error estimate between  $\mathbf{u}_\lambda^\delta$ and $\tilde{\mathbf{u}}$ which specializes  a general argument established in \cite{grasmair2008sparse,grasmair2011necessary} to the regularization problem \eqref{lasso_noise_data}.

\begin{lemma}\label{estimate_lemma}
Suppose that  $\mathbf{A}\in\mathbb{R}^{p\times n}$, the minimal norm problem \eqref{mini_norm} with $\mathbf{x}\in\mathbb{R}^p$ has a unique solution $\tilde{\mathbf{u}}$, and for $\delta>0$, $\mathbf{x}^\delta\in\mathbb{R}^p$ satisfies $\|\mathbf{x}^\delta-\mathbf{x}\|_2\leq\delta$. Let $\mathbf{u}_{\lambda}^\delta$ be a solution of the regularization problem \eqref{lasso_noise_data}.
Then there exist $\beta_1$, $\beta_2>0$ such that for all $\delta, \lambda>0$,
\begin{equation}\label{estimate_error}
\|\mathbf{u}_{\lambda}^\delta-\tilde{\mathbf{u}}\|_2\leq\frac{\lambda\beta_2^2}{2\beta_1}+\frac{\delta^2}{2\lambda\beta_1}+\frac{\beta_2\delta}{\sqrt{2}\beta_1}.
\end{equation}
\end{lemma}



We are ready to present our results. Suppose that  $\mathcal{S}:=\left\{S_{1},S_{2},\ldots, S_{d}\right\}$ is a partition of the set $\mathbb{N}_n$ and $\mathbf{A}\in\mathbb{R}^{p\times n}$ satisfies condition \eqref{suff_nece_Lasso_block_separable}. We introduce a sequence of numbers
\begin{equation}\label{sequence_a_j}
a_j^\delta:=\|(\mathbf{A}_{(j)})^{\top}\mathbf{x}^\delta\|_{\infty},\ \mbox{for all}\ j\in\mathbb{N}_d,
\end{equation}
and rearrange them in a nondecreasing order:
$a_{k_1}^\delta\leq a_{k_2}^\delta\leq\cdots\leq a_{k_d}^\delta$ with distinct $k_i\in\mathbb{N}_d,$ $ i\in\mathbb{N}_d$. 

\begin{theorem}\label{choice_error}
Suppose that  $\mathbf{A}\in\mathbb{R}^{p\times n}$ satisfies condition \eqref{suff_nece_Lasso_block_separable}, the minimal norm problem \eqref{mini_norm} with $\mathbf{x}\in\mathbb{R}^p$ has a unique solution $\tilde{\mathbf{u}}$, and for $\delta>0$, $\mathbf{x}^\delta\in\mathbb{R}^p$ satisfies $\|\mathbf{x}^\delta-\mathbf{x}\|_2\leq\delta$.

(i) If $\lambda:=a_{k_{d-l}}^\delta$ for a given  $l\in\mathbb{Z}_{d+1}$, then the regularization problem \eqref{lasso_noise_data} has a sparse solution $\mathbf{u}_{\lambda}^\delta$ with the $\mathcal{S}$-block sparsity of level $\leq l$ satisfying the error bound 
\begin{equation}\label{estimate_error_sparsity}
\|\mathbf{u}_{\lambda}^\delta-\tilde{\mathbf{u}}\|_2\leq\frac{a_{k_{d-l}}^\delta\beta_2^2}{2\beta_1}+\frac{\delta^2}{2a_{k_{d-l}}^\delta\beta_1}+\frac{\beta_2\delta}{\sqrt{2}\beta_1}
\end{equation}
for some constants $\beta_1, \beta_2>0$ independent of $\delta$.

(ii) If $\lambda:=C\delta$ for a constant $C>0$, then the regularization problem \eqref{lasso_noise_data} has a sparse solution $\mathbf{u}_{\lambda}^{\delta}$ with the $\mathcal{S}$-block sparsity of level $l$ satisfying the error bound
\begin{equation}\label{error_estimate_Cdelta}
\|\mathbf{u}_{\lambda}^{\delta}-\tilde{\mathbf{u}}\|_2\leq
C'\delta,
\end{equation}
where $l\in\mathbb{Z}_{d+1}$ satisfies $a_{k_{d-l}}^\delta\leq C\delta< a_{k_{d-l+1}}^\delta$ and $C':=\frac{C^2\beta_2^2+\sqrt{2}C\beta_2+1}{2C\beta_1}$.
\end{theorem}

\begin{proof}
According to the hypothesis that condition \eqref{suff_nece_Lasso_block_separable} is satisfied, we prove this theorem by employing  Corollary \ref{block_smooth_example} with $\mathbf{x}$ being replaced by $\mathbf{x}^\delta$. Moreover, since we assume that the minimal norm interpolation problem \eqref{mini_norm} has a unique solution $\tilde{\mathbf{u}}$, Lemma \ref{estimate_lemma}  ensures that there exist $\beta_1$, $\beta_2>0$ independent of $\lambda, \delta$ such that the error estimation \eqref{estimate_error} holds.

We first prove statement (i). Since the regularization parameter is chosen as $\lambda:=a_{k_{d-l}}^\delta$, according to the nondescreasing order $a_{k_1}^\delta\leq a_{k_2}^\delta\leq\cdots\leq a_{k_d}^\delta$ of the sequence $a_{k_j}$, $j\in\mathbb{N}_{d}$,
we have that $\lambda\geq a_{k_j}^{\delta}$, for all $j\in\mathbb{N}_{d-l}$.
Appealing to Corollary \ref{block_smooth_example},
we conclude that the regularization problem \eqref{lasso_noise_data} with the so chosen parameter $\lambda$ has a sparse solution $\mathbf{u}_{\lambda}^\delta$ with the $\mathcal{S}$-block sparsity of level $\leq l$. The error bound  \eqref{estimate_error_sparsity} of the regularized solution $\mathbf{u}_{\lambda}^\delta$  is obtained by
substituting $\lambda=a_{k_{d-l}}^\delta$ into the right hand side of estimate \eqref{estimate_error} in  Lemma \ref{estimate_lemma}.

We next show statement (ii). 
Substituting $\lambda=C\delta$ into the right hand side of the estimate \eqref{estimate_error} of Lemma \ref{estimate_lemma} with straightforward computation leads to the error bound \eqref{error_estimate_Cdelta}. In addition, Corollary \ref{block_smooth_example} ensures that the regularized solution $\mathbf{u}_{\lambda}^\delta$ has
the $\mathcal{S}$-block sparsity of level $l$, where $l$ satisfies $a_{k_{d-l}}^\delta\leq C\delta< a_{k_{d-l+1}}^\delta$.
\end{proof}

We may obtain a special result when the matrix $\mathbf{A}$ is an orthogonal matrix of order $n$. In this case, condition \eqref{suff_nece_Lasso_block_separable} holds for the nature partition $\mathcal{S}:=\{\mathcal{S}_1, \dots, \mathcal{S}_n\}$ of the set $\mathbb{N}_n$ and the sequence defined by \eqref{sequence_a_j} can be rearranged as 
$a_{k_1}^\delta\leq a_{k_2}^\delta\leq\cdots\leq a_{k_n}^\delta$ with $k_i\in\mathbb{N}_n,$ $ i\in\mathbb{N}_n$. 

\begin{corollary}\label{choice_error_orthogonal}
Suppose that  $\mathbf{A}\in\mathbb{R}^{n\times n}$ is an orthogonal matrix and for $\delta>0$, $\mathbf{x}^\delta\in\mathbb{R}^p$ satisfies $\|\mathbf{x}^\delta-\mathbf{x}\|_2\leq\delta$.

(i) If $\lambda:=a_{k_{n-l}}^\delta$ for a given  $l\in\mathbb{Z}_{n+1}$, then the regularization problem \eqref{lasso_noise_data} has a sparse solution $\mathbf{u}_{\lambda}^\delta$ with sparsity of level $\leq l$ satisfying the error bound 
\begin{equation*}
\|\mathbf{u}_{\lambda}^\delta-\tilde{\mathbf{u}}\|_2\leq\frac{a_{k_{n-l}}^\delta\beta_2^2}{2\beta_1}+\frac{\delta^2}{2a_{k_{n-l}}^\delta\beta_1}+\frac{\beta_2\delta}{\sqrt{2}\beta_1}
\end{equation*}
for some constants $\beta_1, \beta_2>0$ independent of $\delta$.

(ii) If $\lambda:=C\delta$ for a constant $C>0$, then the regularization problem \eqref{lasso_noise_data} has a sparse solution $\mathbf{u}_{\lambda}^{\delta}$ with sparsity of level $l$ satisfying the error bound
\begin{equation*}
\|\mathbf{u}_{\lambda}^{\delta}-\tilde{\mathbf{u}}\|_2\leq
C'\delta,
\end{equation*}
where $l\in\mathbb{Z}_{n+1}$ satisfies $a_{k_{n-l}}^\delta\leq C\delta< a_{k_{n-l+1}}^\delta$ and $C':=\frac{C^2\beta_2^2+\sqrt{2}C\beta_2+1}{2C\beta_1}$.
\end{corollary}
\begin{proof}
Note that the orthogonal matrix $\mathbf{A}$ satisfies condition \eqref{suff_nece_Lasso_block_separable} for the nature partition $\mathcal{S}:=\{\mathcal{S}_1, \dots, \mathcal{S}_n\}$ of $\mathbb{N}_n$. It follows from
the invertability of $\mathbf{A}$ that the minimal norm interpolation problem \eqref{mini_norm} has a unique solution. That is, the hypothesis of Theorem \ref{choice_error} is satisfied. Hence, the desired results of this corollary follows directly from Theorem \ref{choice_error} with the nature partition $\mathcal{S}$ of $\mathbb{N}_n$. 
\end{proof}

Theorem \ref{choice_error} and Corollary \ref{choice_error_orthogonal} provide parameter choice strategies which balance sparsity of the corresponding regularized solutions and their error bounds. Item (ii) of Theorem \ref{choice_error} and Corollary \ref{choice_error_orthogonal} shows that the proposed parameter choice strategy generates a regularized solution which enjoys the same error bound given in \cite{grasmair2008sparse,grasmair2011necessary} and as well as sparsity of a prescribed level.

Theorem \ref{choice_error} may be extended to general cases when matrix  $\mathbf{A}$ may not satisfy condition \eqref{suff_nece_Lasso_block_separable} and/or the transform matrix $\mathbf{B}$ is not an identity. In such cases, by combining  Corollary \ref{general-smooth-example} (or Theorem \ref{sparse_B}) and Lemma \ref{estimate_lemma}, we may obtain results similar to those in Theorem \ref{choice_error} on choices of the regularization parameter that balances the approximation error and sparsity of the regularized solution. We leave details of further development to the interested readers.

\section{Numerical Experiments}
In this section, we conduct numerical experiments 
to verify theoretical results obtained in the last three sections. The numerical results demonstrate that the proposed parameter choices can balance sparsity of the regularized solution and its approximation accuracy. 
Specifically, we test the results stated in  Theorems 
\ref{choice_sparsity_block_nonsmooth}, \ref{block_smooth_example1}, \ref{sparse_B} and \ref{choice_error}. All the experiments are performed with Matlab R2018a on an Intel Core I9 (8-core) with 5.0 GHz and 32 Gb RAM. 

In our numerical computation, the regularization problems are solved by employing the Fixed Point Proximity Algorithm (FPPA) developed in \cite{argyriou2011efficient,li2015multi,micchelli2011proximity}, which we review below. Suppose that $f:\mathbb{R}^n\to \overline{\mathbb{R}}:= \mathbb{R}\cup\{+\infty\}$ is a convex function, with $\mathrm{dom}(f):=\{\mathbf{x}\in\mathbb{R}^n:f(\mathbf{x})<+\infty\}\neq{\emptyset}.$ The proximity operator $\text{prox}_{f}:\mathbb{R}^n\to\mathbb{R}^n$ of $f$ is defined for $\mathbf{x}\in\mathbb{R}^n$ by
\begin{equation*}
\text{prox}_{f}(\mathbf{x}):=\argmin\left\{\frac{1}{2}\|\mathbf{u}-\mathbf{x}\|_2^2+f(\mathbf{u}):\mathbf{u}\in\mathbb{R}^n\right\}.
\end{equation*}
Suppose that   $\bm{\varphi}:\mathbb{R}^n\to\overline{\mathbb{R}}$  and $\bm{\omega}:\mathbb{R}^m\to\overline{\mathbb{R}}$ are two convex functions which may not be differentiable, and a matrix $\mathbf{C}\in\mathbb{R}^{m\times n}$. We solve the optimization problem 
$\min\{\bm{\varphi}(\mathbf{u})+\bm{\omega}(\mathbf{C}\mathbf{u}):\mathbf{u}\in\mathbb{R}^n\}$ by using the following FPPA.
For given positive constants $\beta$, $\rho$ and initial points $\mathbf{u}^0$, $\mathbf{v}^0$, the FPPA is described as
\begin{equation}\label{FPPA}
\left\{\begin{array}{l}
\mathbf{u}^{k+1}=\operatorname{prox}_{\beta\bm{\varphi}}\left(\mathbf{u}^{k}-\beta \mathbf{C}^{\top} \mathbf{v}^{k}\right), \\
\mathbf{v}^{k+1}=\rho\left(\mathcal{I}-\operatorname{prox}_{\frac{1}{\rho} \bm{\omega}}\right)\left(\frac{1}{\rho} \mathbf{v}^{k}+\mathbf{C}\left(2 \mathbf{u}^{k+1}-\mathbf{u}^{k}\right)\right).
\end{array}\right.
\end{equation}
According to \cite{li2015multi}, iteration \eqref{FPPA} converges if
$\beta \rho<1/\|\mathbf{C}\|_2^2$.
We remark that $\beta$ and $\rho$ involved in \eqref{FPPA} are algorithm parameters which play important roles, their appropriate choices making the algorithm to converge and potentially accelerating the convergence. As different choices of the pair of $\beta$ and $\rho$ under convergence condition $\beta \rho<1/\|\mathbf{C}\|_2^2$ will merely affect the convergence speed of sequence $\mathbf{u}^{k}$ generated by \eqref{FPPA} approaching to the solution $\mathbf{u}^*$, in the experiments to be presented below we will run the algorithm with a relatively large number of iterations instead of putting much effort on tuning $\beta$ and $\rho$ since finding optimal choices of these two parameter is not a focus of this study. The number of iterations depends on the specific problem. 
In each of the numerical examples, by $\mathbf{u}^\infty$ we denote the numerical approximation of solution $\mathbf{u}^*$ after the iterations converge. For convenience, we use $``\mathrm{SL}"$ and $``\mathrm{BSL}"$ to denote the sparsity level and the block sparsity level of $\mathbf{u}^\infty$ (or $\mathbf{B}\mathbf{u}^\infty$), respectively.

\subsection{Image denoising by using DWT}
In this subsection, we test the result in Theorem \ref{choice_sparsity_block_nonsmooth} by considering the image denoising model with an orthogonal discrete wavelet transform.  Given a noisy image $\mathbf{x}\in\mathbb{R}^{n^2}$ and a one-dimensional orthogonal discrete wavelet transform matrix $\mathbf{W}\in\mathbb{R}^{n\times n}$. We consider the image denoising model \cite{donoho1995adapting}
\begin{equation}\label{Image_denoising}
\min\left\{\frac{1}{2}\|\mathbf{u}-\mathbf{x}\|_2^2+\lambda\|\mathbf{B}\mathbf{u}\|_1:\mathbf{u}\in\mathbb{R}^{n^2}\right\}, \end{equation}
where $\mathbf{B}:=\mathbf{W}\otimes\mathbf{W}$ and $\otimes$ denotes the Kronecker product.
Note that the matrix $\mathbf{B}$, as a Kronecker product of two orthogonal matrices, is also orthogonal. By simple change of variables $\mathbf{v}=\mathbf{B}\mathbf{u}$, we identify \eqref{Image_denoising} as the Lasso regularized model \eqref{lasso} with $p=n$, and $n$, $\mathbf{A}$, $\mathbf{u}$ being replaced by $n^2$, $\mathbf{B}^{\top}$, $\mathbf{v}$, respectively. Applying Theorem \ref{choice_sparsity_block_nonsmooth} to the Lasso regularized model \eqref{lasso} with an orthogonal matrix leads to Corollary \ref{block_smooth_example} with the nature partition  $\mathcal{S}$ of $\mathbb{N}_n$. 

The experiment is conducted on gray scale test image ‘Cameraman' with size $256\times 256$. We use $\mathbf{f}:=[f_j:j\in\mathbb{N}_{n^2}]$ with $n:=256$ for the original image. The noisy image is modeled as $\mathbf{x}:=\mathbf{f}+\bm{\eta}$ with noise $\bm{\eta}$ iid $N(0,\sigma^2)$ being Gaussian noise at level $\sigma=20$. The discrete wavelet transform matrix $\mathbf{W}$ is the Daubechies wavelet with the vanishing moments $\mathrm{N}=4$ and the coarsest resolution level $\mathrm{L}=4$. We choose seven different values of the parameter $\lambda$ according to Corollary \ref{block_smooth_example}
from the set $\{b_j:j\in\mathbb{N}_{n^2}\}$ with  $b_j:=|((\mathbf{B}^{\top})_{j})^{\top}\mathbf{x}|$, $j\in\mathrm{N}_{n^2}$. Specifically, we rearrange the sequence $b_j, j\in\mathbb{N}_{n^2}$, in a nondecreasing order: $b_{k_1}\leq b_{k_2}\leq \cdots \leq b_{k_{n^2}}$ with distinct $k_i\in\mathbb{N}_{n^2}, i\in\mathbb{N}_{n^2}$. Numerically we pick $\lambda=b_{k_j}$, for $j=23515$, $35941$, $45295$, $53490$, $61229$, $65091$, $65536$ with indices of $j$ selected randomly. The minimization problem \eqref{lasso} with each value of $\lambda$ is solved by the FPPA with $\bm{\varphi}:=\lambda\|\cdot\|_1$, $\bm{\omega}:=\frac{1}{2}\|\cdot-\mathbf{x}\|_2^2$ and $\mathbf{C}:=\mathbf{B}^{\top}$.

\begin{table}[ht]
\caption{\label{image_DWT_sparsity} Numerical results for image denoising model by 100 iterations}
\vspace{0.2cm}
\begin{indented}
\lineup
\item[]\begin{tabular}{@{}cc|c|c|c|c|c|c|c}
\hline\hline
&$\lambda$ &$10.1895$ &$16.7430$ &$23.0684$ &$31.3891$ &$49.8117$ &$228.6029$ &$3707.6947$\\ \hline 
&$\mathrm{SL}$  &$42021$  &$29595$ &$20241$ &$12046$
&$4307$ &$445$  &$0$ \\   \hline 
&$\mathrm{PSNR}$   
&$25.3330$  &$26.5095$ &$26.8584$ &$26.4546$ &$24.7376$ &$18.8337$  &$5.5824$ \\   \hline \hline 
\end{tabular}
\end{indented}
\end{table}



\begin{figure}[H]
\centering
\subfigure[]{
\begin{minipage}[ht]{0.3\linewidth}
\centering
\includegraphics[width=4.6cm,height=3.6cm]{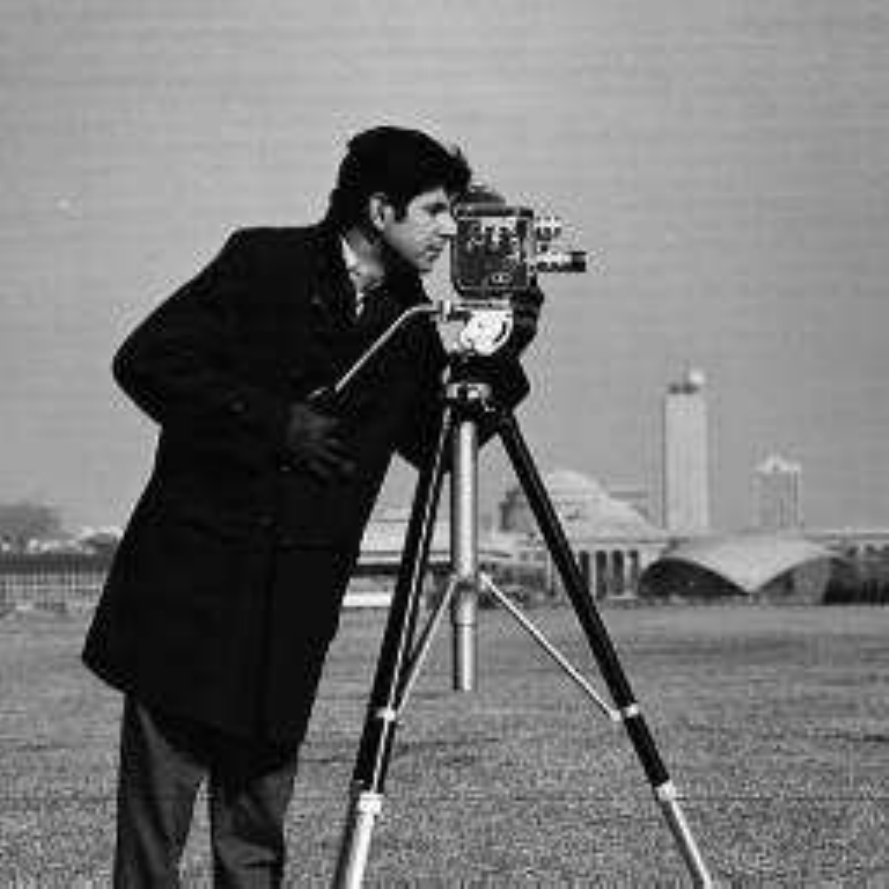}
\end{minipage}%
}%
\subfigure[]{
\begin{minipage}[ht]{0.3\linewidth}
\centering
\includegraphics[width=4.6cm,height=3.6cm]{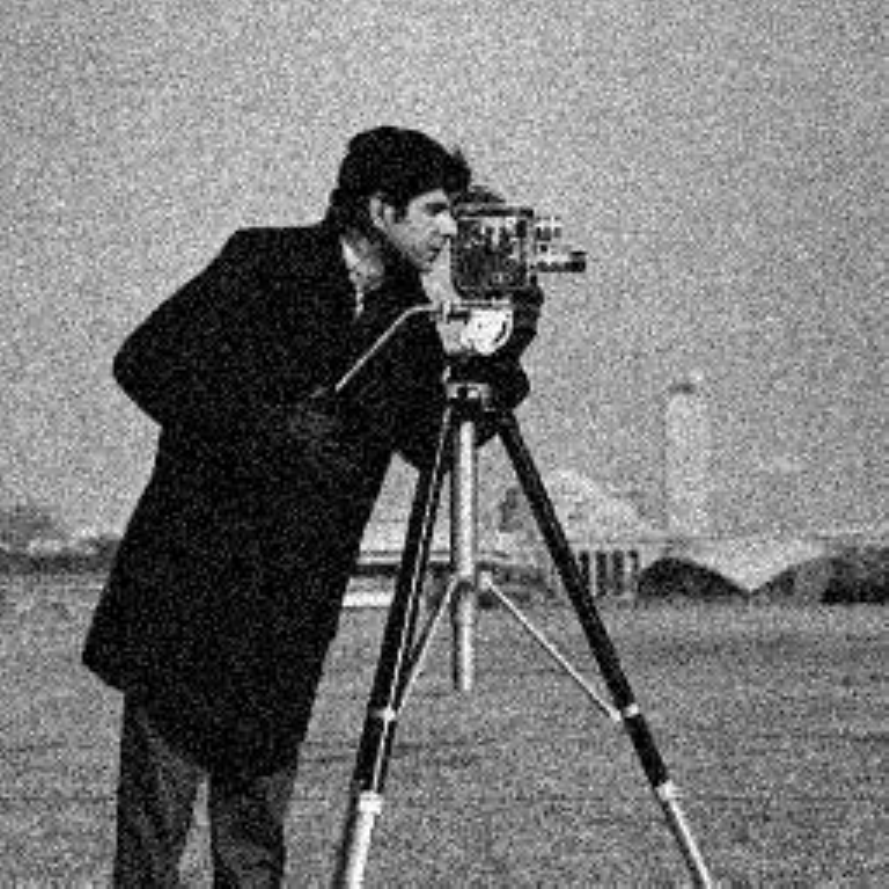}
\end{minipage}
}%
\subfigure[]{
\begin{minipage}[ht]{0.3\linewidth}
\centering
\includegraphics[width=4.6cm,height=3.6cm]{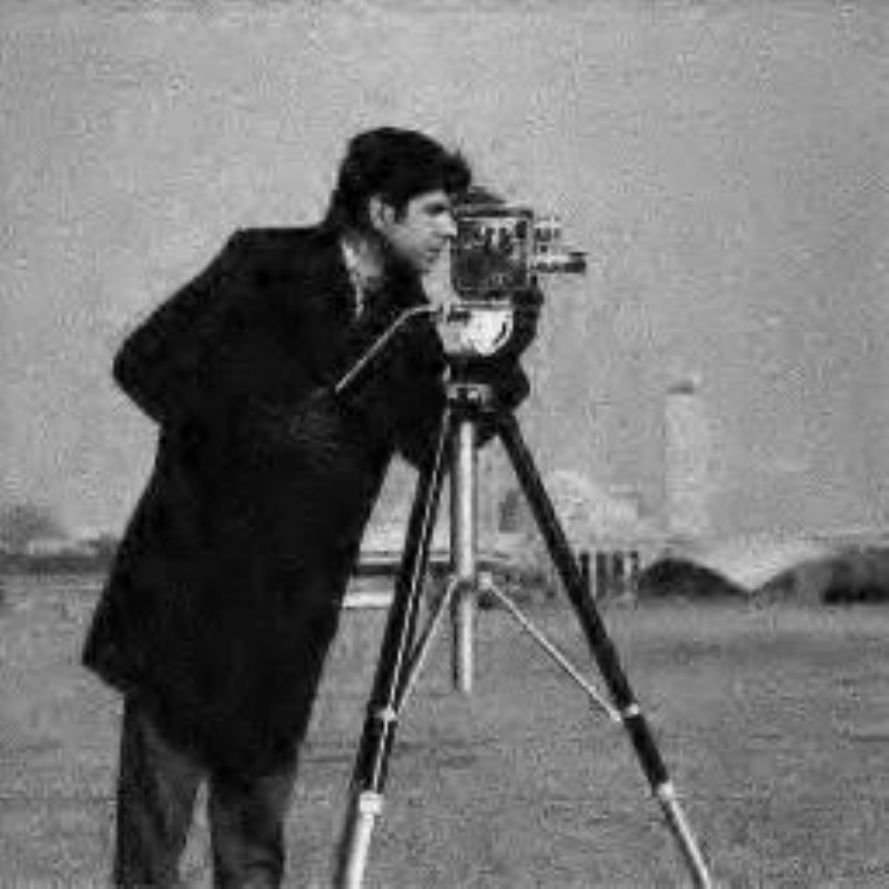}
\end{minipage}%
}%
\centering
\caption{(a) The original image of ‘Cameraman’; (b) the noisy
image of ‘Cameraman’ with Gaussian noise at level $\sigma=20$; (c) the denoised image of ‘Cameraman’  with $\lambda=23.0684$ (SL=20241, $\mathrm{PSNR}=26.8584$).}
\label{figure_Original_Noised}
\end{figure}

We report in Table \ref{image_DWT_sparsity} the selected values of parameter $\lambda$, the sparsity levels $\mathrm{SL}$ of the numerical approximation $\mathbf{v}^\infty$ and the PSNR values of the denoised images $\mathbf{B}^{\top}\mathbf{v}^{\infty}$, where the PSNR value is defined by  $\mathrm{PSNR}:=20 \mbox{log}_{10}(255/\|\mathbf{f}-\mathbf{B}^{\top}\mathbf{v}^{\infty}\|_2)$, and show in Figure \ref{figure_Original_Noised} the original image, the noisy image, and the denoised image with $\lambda=23.0684$. These numerical results confirm the theoretical result presented in Corollary \ref{block_smooth_example}.
Note that Corollary \ref{block_smooth_example} guarantees that the solutions $\mathbf{v}^*$ of problem \eqref{lasso} with the selected values of the parameter $\lambda$ have the sparsity levels $l=n^2-j$, that is, $l=42021$, $29595$, $20241$, $12046$, $4307$, $445$, $0$, respectively, which coincide exactly with the sparsity levels of the numerical solutions  $\mathbf{v}^\infty$  reported in Table \ref{image_DWT_sparsity}. 

\subsection{Signal denoising by the group Lasso regularized model}\label{Group_lasso_signal}
In this subsection, we test the result in Theorem \ref{block_smooth_example1} by considering the group Lasso regularized model \eqref{group_lasso} with $\mathbf{A}\in\mathbb{R}^{n\times n}$ whose columns form an orthogonal wavelet basis. 

We consider recovering the Doppler signal function
\begin{equation}\label{Doppler_signal}
f(t):=\sqrt{t(1-t)}\mathrm{sin}\left(\frac{2.1\pi}{t+1.05}\right), \ t\in[0,1], \end{equation}
from its noisy data by employing the group Lasso regularized model \eqref{group_lasso}. 
Let $n:=4096$. We generate sample points $t_j, j\in\mathbb{N}_{n}$, on uniform grid in $[0,1]$ with step size $\mathrm{h}=1/(n-1)$ and consider recovering the signal $\mathbf{f}:=[f(t_j):j\in\mathbb{N}_{n}]$ from a noisy signal $\mathbf{x}:= \mathbf{f}+\bm{\eta}$, where $\bm{\eta}$ is an additive white Gaussian noise with the signal-to-noise ratio $\mathrm{SNR}=7$. The matrix $\mathbf{A}$ comes from the Daubechies wavelet with $\mathrm{N}=6$ and $\mathrm{L}=3$. By introducing a partition $\mathcal{S}:=\left\{S_{1},S_{2},\ldots, S_{10}\right\}$ of the index set $\mathbb{N}_{n}$ with the cardinality $n_1=2^3$ and $n_{j}=2^{j+1},j\in\mathbb{N}_{10}\setminus\{1\}$, we decompose the matrix $\mathbf{A}$ into $10$ sub-matrices defined by  $\mathbf{A}_{(j)}:=[\mathbf{A}_k:k\in{S}_j]\in\mathbb{R}^{n\times n_j}$, $j\in\mathbb{N}_{10}$. We choose seven different values of the parameter $\lambda$ according to Theorem \ref{block_smooth_example1} from the set $\{a_j:j\in\mathbb{N}_{10}\}$ with $a_j:=\|(\mathbf{A}_{(j)})^{\top}\mathbf{x}\|_2/{\sqrt{n_j}}$, $j\in\mathbb{N}_{10}$. Specifically, we rearrange the sequence $a_j,j\in\mathbb{N}_{10}$, in a nondecreasing order: $a_{k_1}\leq a_{k_2}\leq \cdots \leq a_{k_{10}}$ with distinct $k_i\in\mathbb{N}_{10}, i\in\mathbb{N}_{10}$. We then choose $\lambda=a_{k_j}$, for $j=0$, $2$, $4$, $5$, $7$, $9$, $10$. The  problem \eqref{group_lasso} 
with each value of $\lambda$ is solved by the FPPA with $\bm{\varphi}(\mathbf{u}):=\lambda\sum_{j\in\mathbb{N}_{10}}\sqrt{n_j}\|\mathbf{u}_j\|_2$, $\mathbf{u}\in\mathbb{R}^n$, $\bm{\omega}:=\frac{1}{2}\|\cdot-\mathbf{x}\|_2^2$ and $\mathbf{C}:=\mathbf{A}$.

\begin{table}[ht]
\caption{\label{Group_lasso_sparsity_1m} Numerical results for signal denoising with the group Lasso regularized model by $1000$ iterations}
\vspace{0.2cm}
\begin{indented}
\lineup
\item[]\begin{tabular}{@{}cc|c|c|c|c|c|c|c}
\hline\hline
&$\lambda$  &$0.1312$ &$0.1400$  &$0.1541$ &$0.2175$ &$1.6193$ &$3.1193$ &$5.0126$\\ \hline 
&$\mathrm{BSL}$  &$10$ &$8$  &$6$ &$5$ 
&$3$ &$1$  &$0$ \\   \hline 
&$\mathrm{MSE}$  &$0.0017$ &$0.0018$  &$0.0020$ &$0.0031$ 
&$0.0296$ &$0.0563$  &$0.0858$ \\   \hline \hline 
\end{tabular}
\end{indented}
\end{table}

The selected values of parameter $\lambda$, the $\mathcal{S}$-block sparsity levels $\mathrm{BSL}$ of the numerical approximation $\mathbf{u}^\infty$ and the $\mathrm{MSE}$ values of the denoised signals $\mathbf{A}\mathbf{u}^{\infty}$ are reported in Table \ref{Group_lasso_sparsity_1m}, where the $\mathrm{MSE}$ value is defined by  $\mathrm{MSE}:=\frac{1}{n}\|\mathbf{f}-\mathbf{A}\mathbf{u}^{\infty}\|_2^2$. These  numerical results confirm the theoretical result stated in Theorem \ref{block_smooth_example1}. In particular, the $\mathcal{S}$-block sparsity levels $\mathrm{BSL}$ of the numerical approximation $\mathbf{u}^\infty$ match exactly with those of the solutions $\mathbf{u}^*$, which are given by $l=10-j$ (that is, $l=10$, $8$, $6$, $5$, $3$, $1$, $0$) guaranteed by  Theorem \ref{block_smooth_example1}. Moreover, the approximation errors of the solutions corresponding to the selected values of $\lambda$ exhibit increase as  the values of $\lambda$ become larger.

\subsection{Total-variation signal denoising}\label{TV_signal_denoising}

In this experiment, we test the result in Theorem \ref{sparse_B} by considering the total-variation signal denoising model \eqref{total-variation-signal-denoising}. Note that Theorem \ref{sparse_B} applied to this model leads to Corollary \ref{general-smooth-example0}. Due to the choice of $\mathbf{B}$ the resulting minimization problem \eqref{total-variation-signal-denoising} is neither separable nor block separable. Hence, the choice of parameter $\lambda$ described in Corollary \ref{general-smooth-example0} depends on the unknown solution $\mathbf{u}^*$. Unlike the experiments presented in the last two subsections, in this experiment, we test the necessary condition described by the inequalities \eqref{lambda_general-smooth-example02} in Corollary \ref{general-smooth-example1} for the problem \eqref{total-variation-signal-denoising} to have a solution $\mathbf{u}^{*}$ with a certain sparsity level under the transform $\mathbf{D}^{(1)}$. Specifically, for a chosen value of parameter $\lambda$, we solve the corresponding regularization problem \eqref{total-variation-signal-denoising} and obtain a numerical approximation $\mathbf{u}^{\infty}$. We then verify the pair of the chosen $\lambda$ value and the corresponding solution  $\mathbf{u}^{\infty}$ satisfy inequalities \eqref{lambda_general-smooth-example02}, and the sparsity level of $\mathbf{D}^{(1)}\mathbf{u}^{\infty}$. In addition, we test the parameter choice strategy described in Remark \ref{general-smooth-example0_most_sparse} for the case when the solution has the most sparsity under the transform $\mathbf{D}^{(1)}$.

\begin{table}[ht]
\caption{\label{table: transformation B} Numerical results for total-variation signal denoising model by $50000$ iterations}
\vspace{0.2cm}
\begin{indented}
\lineup
\item[]\begin{tabular}{@{}cc|c|c|c|c|c|c|c}
\hline\hline      
&$\lambda$              
&$0.1$ &$0.2$ &$0.5$  &$6$ &$30$ &$102$ &$270.1717$\\ \hline 
&$\gamma$
&$0.1000$ &$0.1997$  &$0.5000$ &$6.0000$ &$29.9983$ &$101.9980$ &$270.1716$\\ \hline 
&$\mathrm{SL}$  
&$1488$ &$744$ &$355$  &$160$ &$67$ &$12$ &$0$ \\  \hline 
&$\mathrm{MSE}$  
&$0.0042$ &$0.0020$ &$0.0012$  &$0.0076$ &$0.0297$ &$0.0638$ &$0.0835$ \\ \hline \hline 
\end{tabular}
\end{indented}
\end{table}

Again, we consider recovering the Doppler signal function defined by \eqref{Doppler_signal} from its noisy data. The original signal $\mathbf{f}$ and the noisy signal $\mathbf{x}$ are chosen in the same way as in subsection \ref{Group_lasso_signal}. According to Remark \ref{general-smooth-example0_most_sparse}, when $\lambda \geq\lambda_{\max}:=\| (\widetilde{\mathbf{D}}^{(1)'})^{\top}\mathbf{x}\|_{\infty} \ (\approx 270.1717)$, the corresponding regularized solution is the zero vector under the transform $\mathbf{D}^{(1)}$. For this reason, we choose seven different values of the parameter $\lambda$ in the interval $(0,\lambda_{\max}]$ and solve the minimization problem  \eqref{total-variation-signal-denoising} with each of such values by employing the FPPA with $\bm{\varphi}:=\frac{1}{2}\|\cdot-\mathbf{x}\|_2^2$, $\bm{\omega}:=\lambda\|\cdot\|_1$ and $\mathbf{C}:=\mathbf{D}^{(1)}$. 
Associated with the numerical approximation $\mathbf{u}^{\infty}$, we identify $l$ distinct integers $k_i\in \mathbb{N}_{n-1}$ so that $\mathbf{D}^{(1)}\mathbf{u}^{\infty}:=\sum_{i\in\mathbb{N}_l}z^{\infty}_{k_i}\mathbf{e}_{k_i}$, $z^{\infty}_{k_i}\in\mathbb{R}\setminus{\{0\}}$, $i\in\mathbb{N}_l$. That is, $\mathbf{D}^{(1)}\mathbf{u}^{\infty}$ has only $l$ nonzero components. We then compute the number $\gamma:=\mathrm{max}\left\{\left|(\mathbf{D}^{(1)'}_j)^{\top}
(\mathbf{u}^\infty-\mathbf{x})\right|:j\in \mathbb{N}_{n-1}\setminus\{k_i:i\in\mathbb{N}_l\}\right\}$ and verify indeed that $\lambda>\gamma$, which in turn implies that inequalities \eqref{lambda_general-smooth-example02} in Corollary \ref{general-smooth-example0} are satisfied.
We report in Table \ref{table: transformation B} the selected values of parameter $\lambda$, the values of $\gamma$, the sparsity levels $\mathrm{SL}$ of $\mathbf{D}^{(1)}\mathbf{u}^\infty$ and the $\mathrm{MSE}$ values of the denoised signals $\mathbf{u}^\infty$, where the $\mathrm{MSE}$ value is defined by $\mathrm{MSE}:=\frac{1}{n}\|\mathbf{f}-\mathbf{u}^{\infty}\|_2^2$. These numerical results indeed confirm inequalities \eqref{lambda_general-smooth-example02} in Corollary \ref{general-smooth-example0} and the theoretical result stated in Remark \ref{general-smooth-example0_most_sparse}. In particular, numerical results presented in the last column of Table \ref{table: transformation B} verify the estimate presented in Remark \ref{general-smooth-example0_most_sparse}.

\subsection{$\ell_1$ SVM classification and regression with the squared loss function}\label{MNIST_database}

In this subsection,  we consider the $\ell_1$ SVM models \eqref{SVM_square_loss} for classification and regression with the squared loss function and test the theoretical estimate of Theorem \ref{sparse_B} applied to this model  (Corollary \ref{general-smooth-example1}). Like the case in subsection \ref{TV_signal_denoising}, the choice of parameter $\lambda$ described in Corollary \ref{general-smooth-example1} depends on the unknown solution $\mathbf{u}^*$ and hence we test the necessary condition for a solution $\mathbf{u}^*$ of the minimization problem \eqref{SVM_square_loss} to have sparsity of a certain level under the transform $\mathbf{B}$ (that is, the inequalities \eqref{lambda_general-smooth-example12} in Corollary \ref{general-smooth-example1}). Specifically, we obtain a numerical solution $\mathbf{u}^{\infty}$ by solving the minimization problem \eqref{SVM_square_loss} with a 
chosen value of parameter $\lambda$. We then verify that the pair of the chosen $\lambda$ value and the corresponding solution $\mathbf{u}^{\infty}$ satisfy inequalities \eqref{lambda_general-smooth-example12}, and the sparsity level of $\mathbf{B}\mathbf{u}^{\infty}$. For a special case of Corollary \ref{general-smooth-example1}, we test the choice of the parameter proposed in Remark \ref{general-smooth-example1_most_sparse} which leads to the zero solution $\mathbf{u}^\infty$ under the transform $\mathbf{B}$. 

In the first case, we consider the $\ell_1$ SVM classification model \eqref{SVM_square_loss}. The dataset that we use for classification is the handwriting digits from MNIST database \cite{lecun1998gradient}, which is originally composed of $60,000$ training samples and $10,000$ testing samples of the digits $``0"$ through $``9"$. We consider the binary classification problem with two digits $``7"$ and $``9"$ and take $8,141$ training samples and $2,037$ testing samples of these two digits from the database. The reason for which we choose these two particular digits is that it has been recognized that their handwriting is not easy to distinguish.  
The kernel we choose for model \eqref{SVM_square_loss} is the Gaussian Kernel defined by 
\begin{equation}\label{Gaussian_Kernel}
K(x,y):=\mathrm{exp}\left(-\frac{\|x-y\|_2^2}{2\mu^2}\right),
\ \ x,y\in \mathbb{R}^n,
\end{equation}
with $\mu=4.8$ and $n=8141$. Let $\mathbf{y}\in\{-1,1\}^{n}$ be the given vector storing labels of training data in which $-1$ and $1$ represent the digits $``7"$ and $``9"$ respectively. Remark \ref{general-smooth-example1_most_sparse} guarantees that when $\lambda\geq\lambda_{\mathrm{max}}:=\left\|\mathbf{K}^{\top}\left(\frac{1}{n}\mathbf{1}_{n}^{\top}\mathbf{y}\mathbf{1}_{n}-\mathbf{y}\right)\right\|_{\infty} (\approx 487.7454)$, the corresponding regularization problem \eqref{SVM_square_loss} has a solution with $\mathbf{B}\mathbf{u}^*=\mathbf{0}$. Accordingly, we choose seven different values of the parameter $\lambda$ in the interval $(0,\lambda_{\max}]$. The minimization problem \eqref{SVM_square_loss} with each of such values is solved by the FPPA with $\bm{\varphi}:=\lambda\|\cdot\|_1\circ\mathbf{B}$, $\bm{\omega}:=\frac{1}{2}\|\cdot-\mathbf{x}\|_2^2$ and $\mathbf{C}:=\mathbf{K}'$. 

\begin{table}[ht]
\caption{\label{table_sparsity_5} Numerical results for $\ell_1$ SVM classification model with the square loss function by $30,000$ iterations}
\vspace{0.2cm}
\begin{indented}
\lineup
\item[]\begin{tabular}{@{}cc|c|c|c|c|c|c|c} 
\hline\hline 
&$\lambda$          
 &$0.5$  &$0.7$ &$1.4$ &$1.8$ &$4.0$ &$6.0$ &$487.7454$ \\ \hline        
 &$\gamma$
 &$0.4985$ &$0.6982$ &$1.4000$ &$1.7990$ &$3.9956$ &$5.9990$ &$487.7454$\\ \hline 
&$\mathrm{SL}$  
&$1762$ &$880$ &$481$ &$340$ &$187$ &$142$ &$0$\\ \hline
&$\mathrm{TRA}$ &$99.46\%$  &$99.19\%$ &$98.61\%$ &$98.34\%$ &$97.57\%$ &$97.16\%$ &$50.67\%$\\ \hline        
&$\mathrm{TEA}$              
 &$98.77\%$ &$98.82\%$ &$98.48\%$ &$98.38\%$ &$97.79\%$ &$97.40\%$ &$50.47\%$\\        
\hline \hline   
\end{tabular}
\end{indented}
\end{table}

By using the numerical approximation $\mathbf{u}^{\infty}$, we identify $l$ distinct integers $k_i\in\mathbb{N}_{n}$ so that $\mathbf{B}\mathbf{u}^{\infty}:=\sum_{i\in\mathbb{N}_l}z^{\infty}_{k_i}\mathbf{e}_{k_i}\in\Omega_l$. We compute the numbers 
$\gamma:=\max\left\{\left|(\mathbf{K}_j)^{\top}\left(\mathbf{K}'\mathbf{u}^{\infty}-\mathbf{y}\right)\right|, j\in \mathbb{N}_n\setminus \{k_i:i\in\mathbb{N}_l\}\right\}$ and verify indeed that $\lambda\geq\gamma$. The selected values of parameter $\lambda$, the values of $\gamma$, the sparsity levels $\mathrm{SL}$ of $\mathbf{B}\mathbf{u}^\infty$, the values of accuracy on training datasets $\mathrm{TRA}$ and on testing datasets $\mathrm{TEA}$ are reported in Table \ref{table_sparsity_5}, where the accuracy is measured by labels that are correctly predicted by the model. These numerical results confirm  inequalities \eqref{lambda_general-smooth-example12} in Corollary \ref{general-smooth-example1}.  
In particular, the numerical results presented in the last column of Table \ref{table_sparsity_5} coincide with the theoretical result presented in Remark \ref{general-smooth-example1_most_sparse}.


\begin{table}[ht]
\caption{\label{table_sparsity_8} Numerical results for $\ell_1$ SVM regression model with square loss by $50,000$ iterations}
\vspace{0.2cm}
\begin{indented}
\lineup
\item[]\begin{tabular}{@{}cc|c|c|c|c|c|c|c}  
\hline\hline      
&$\lambda$              
&$0.002$ &$0.005$  &$0.02$ &$0.1$ &$1$  &$5$ &$37.5378$\\ \hline  
&$\gamma$
&$0.0020$ &$0.0049$ &$0.0198$ &$0.0997$ &$0.9981$ &$4.9911$ &$37.5378$\\ \hline
&$\mathrm{LS}$  
&$513$ &$306$ &$106$  &$28$ &$12$ &$5$ &$0$\\ \hline 
&$\mathrm{TRM}$ 
&$0.0129$ &$0.0132$  &$0.0139$ &$0.0148$ &$0.0179$  &$0.0196$ &$0.0512$\\ \hline        
&$\mathrm{TEM}$              
&$0.0146$ &$0.0148$ &$0.0151$  &$0.0153$ &$0.0171$ &$0.0190$ &$0.0513$\\
\hline \hline   
\end{tabular}
\end{indented}
\end{table}

In the second case, we consider the $\ell_1$ SVM regression model \eqref{SVM_square_loss}.   
The benchmark dataset is ``Mg"  \cite{chang2011libsvm} with 1,385 instances and each instance has 6 features. We take 1,000 instances as training data and 385 instances as testing data. The kernel involving in model \eqref{SVM_square_loss} is chosen as the Gaussian Kernel defined by 
\eqref{Gaussian_Kernel} with $\mu=1.07$ and $n=1,000$.  Let $\mathbf{y}\in\mathbb{R}^n$ be the given labels. According to Remark \ref{general-smooth-example1_most_sparse}, when $\lambda\geq\lambda_{\mathrm{max}}$($\approx 37.5378$) defined as above, the corresponding regularization problem \eqref{SVM_square_loss} has a solution with $\mathbf{B}\mathbf{u}^*=\mathbf{0}$. Then we select seven different values of parameter $\lambda$ in $(0,\lambda_{\max}]$. The minimization problem \eqref{SVM_square_loss} with each value of $\lambda$ is solved by the FPPA with the same $\bm{\varphi}$, $\bm{\omega}$ and $\mathbf{C}$ as those in the previous case. 
We compute the numbers 
$\gamma$ as in the first case and verify indeed that $\lambda\geq\gamma$. The MSE values of the prediction $\tilde{\mathbf{y}}:=\mathbf{K}'\mathbf{u}^\infty$ is defined by $\mathrm{MSE}:=\frac{1}{n}\|\mathbf{y}-\tilde{\mathbf{y}}\|^2_2$. The $\mathrm{MSE}$ on training and testing datasets are denoted by $\mathrm{TRM}$ and $\mathrm{TEM}$ for simplicity. We report in Table \ref{table_sparsity_8} the selected values of parameter $\lambda$, the values of $\gamma$, the sparsity levels $\mathrm{SL}$ of  $\mathbf{B}\mathbf{u}^\infty$ and the  $\mathrm{TRM}$ values and the $\mathrm{TEM}$ values. The numerical results presented in Table \ref{table_sparsity_8} confirm the inequalities \eqref{lambda_general-smooth-example12} in Corollary \ref{general-smooth-example1} and the theoretical result stated in Remark \ref{general-smooth-example1_most_sparse}.

\subsection{$\ell_1$ SVM classification with the hinge loss function and regression with the $\epsilon$-insensitive loss function}
In this subsection, we test the result in Theorem \ref{sparse_B} by considering the $\ell_1$ SVM classification model \eqref{SVM_hinge_loss} and the $\ell_1$ SVM regression model \eqref{SVM_epsilon_insensitive_loss}. Applying Theorem \ref{sparse_B} to these two models leads to Corollary \ref{general-example1} and Corollary \ref{general-example2}, respectively. Noting that for these two $\ell_1$ SVM models, the choice of parameter $\lambda$ in Corollaries \ref{general-example1} and \ref{general-example2} also depends on the unknown solution $\mathbf{u}^*$. As the cases discussed in \ref{TV_signal_denoising} and \ref{MNIST_database}, we test the necessary condition described by the inequalities \eqref{lambda_general-example1_2} in Corollary \ref{general-example1} and the inequalities \eqref{lambda_general-example2_2} in Corollary \ref{general-example2} for the problem \eqref{SVM_hinge_loss} and \eqref{SVM_epsilon_insensitive_loss}, respectively, having a solution $\mathbf{u}^{*}$ with sparsity of a prescribed level under the transform $\mathbf{B}$. Specifically, for a given parameter $\lambda$, by solving the minimization problems \eqref{SVM_hinge_loss} and \eqref{SVM_epsilon_insensitive_loss}, we obtain the corresponding numerical solutions $\mathbf{u}^{\infty}$.
We then verify the pair of the chosen $\lambda$ value and the corresponding solution  $\mathbf{u}^{\infty}$ satisfy inequalities \eqref{lambda_general-example1_2}, and the sparsity level of $\mathbf{B}\mathbf{u}^{\infty}$. We also test the similar relation for the inequalities \eqref{lambda_general-example2_2} in Corollary \ref{general-example2}

In the first experiment, we consider the $\ell_1$ SVM classification model \eqref{SVM_hinge_loss} with the hinge loss function. The dataset we use is MNIST database with digits ``7" and ``9'' as mentioned in subsection \ref{MNIST_database}. The experiment uses $512$ training samples and $2037$ testing samples. The kernel involves in model \eqref{SVM_hinge_loss} is chosen as the Gaussian Kernel defined by \eqref{Gaussian_Kernel} with $\mu=4$ and $n=512$. We choose seven different values of parameter $\lambda$. The minimization problem 
\eqref{SVM_hinge_loss} with each value of $\lambda$ is solved by the FPPA with $\bm{\varphi}:=\lambda\|\cdot\|_1\circ\mathbf{B}$, $\bm{\omega}:=\bm{\phi}$ and $\mathbf{C}:=\mathbf{Y}\mathbf{K}'$. 

\begin{table}[ht]
\caption{$\ell_1$ SVM classification model with hinge loss (512 training dataset) by millions of iterations}
\vspace{0.2cm}
\begin{indented}
\lineup
\item[]\begin{tabular}{@{}cc|c|c|c|c|c|c|c}  
\hline\hline      
&$\lambda$              
&$0.1$ &$0.2$&$1$&$2$ &$4$&$10$ &$27.9851$\\ \hline 
&$\gamma$
&$0.0543$ &$0.1458$&$0.8785$ &$1.9399$ 
&$3.7844 $  &$ 9.6159$  &$27.9850$ \\ \hline 
&$\mathrm{SL}$  
& $151$ & $151$ & $111$ & $56$ & $37$ & $15$ &$0$\\  \hline 
&$\mathrm{TRA}$          
&$100\%$ &$100\%$&$97.27\%$  &$95.51\%$
&$92.97\%$ &$81.45\%$&$50.00\%$\\ \hline
&$\mathrm{TEA}$        
&$96.71\%$ & $96.66\%$ & $95.68\%$&$93.47\%$ &$91.36\%$ & $80.85\%$ & $50.47\%$\\
\hline \hline
\end{tabular}
\label{table_sparsity_41}
\end{indented}
\end{table}

Associated with the numerical approximation $\mathbf{u}^{\infty}$, we identify $l$ distinct integers $k_i\in\mathbb{N}_{n}$ so that $\mathbf{B}\mathbf{u}^{\infty}:=\sum_{i\in\mathbb{N}_l}z^{\infty}_{k_i}\mathbf{e}_{k_i}$, $z^{\infty}_{k_i}\in\mathbb{R}\setminus{\{0\}}$, $l\in\mathbb{N}_l$. That is, $\mathbf{B}\mathbf{u}^{\infty}$ has $l$ nonzero components. We compute the numbers 
$\gamma:=\mathrm{max}\left\{
\mathrm{min}\left\{
|(\mathbf{Y}\mathbf{K}_j)^{\top}\mathbf{c}|:\mathbf{c}\in\partial\bm{\phi}(\mathbf{Y}\mathbf{K}'\mathbf{u}^{\infty})\right\}: 
\ j\in \mathbb{N}_n\setminus  \{k_i:i\in\mathbb{N}_l\}\right\}$ and verify indeed that $\lambda\geq\gamma$. The selected value of parameter $\lambda$, the values of $\gamma$, the sparsity levels $\mathrm{SL}$ of $\mathbf{B}\mathbf{u}^\infty$, the $\mathrm{TRA}$ values and the  $\mathrm{TEA}$ values are reported in Table  \ref{table_sparsity_41}. These numerical results confirm the inequality \eqref{lambda_general-example1_2} in Corollary \ref{general-example1}.
Note that if $\lambda$ is sufficiently large, the vector $\mathbf{B}\mathbf{u}^*$ has the sparsity of level $0$. The value of the parameter $\lambda$ listed in the last column of Table \ref{table_sparsity_41} such that the solution has most sparsity under the transform $\mathbf{B}$. Testing accuracy being worse is caused by small number of training samples which can not capture significant features.

We repeat the first experiment by using $8,141$ training samples and $2,037$ testing samples. The kernel involves in model \eqref{SVM_hinge_loss} is  chosen as the Gaussian Kernel defined by \eqref{Gaussian_Kernel} with $\mu=4$ and $n=8,141$.
The selected value of parameter $\lambda$, the sparsity levels $\mathrm{SL}$ of $\mathbf{B}\mathbf{u}^\infty$, the $\mathrm{TRA}$ values and the  $\mathrm{TEA}$ values are reported in Table \ref{table_sparsity_4}. Observing form Table \ref{table_sparsity_4}, we get that the sparsity level of $\mathbf{B}\mathbf{u}^{\infty}$ is smaller as the value of parameter $\lambda$ increases, the corresponding $\mathrm{TRA}$ and $\mathrm{TEA}$ values become lower. 

\begin{table}[ht]
\caption{$\ell_1$ SVM classification model with hinge loss (8,141 training dataset) by 30,000 iterations}
\vspace{0.2cm}
\begin{indented}
\lineup
\item[]\begin{tabular}{@{}cc|c|c|c|c|c|c|c}  
\hline\hline      
&$\lambda$              
&$0.1$ &$0.2$ &$1$ &$2$ &$4$ &$10$ &$435.0694$  \\ \hline 
&$\mathrm{SL}$  
&$552$ &$481$ &$167$ &$92$ &$56$ &$34$ &$0$ \\  \hline 
&$\mathrm{TRA}$              
&$99.99\%$ &$99.99\%$  &$99.08\%$
&$98.17\%$ &$97.53\%$ &$96.30\%$&$50.67\%$ \\ \hline        
&$\mathrm{TEA}$         
&$98.72\%$ &$98.77\%$ &$98.38\%$
&$98.09\%$ &$97.45\%$ &$96.27\%$ &$50.47\%$\\        
\hline \hline    
\end{tabular}
\label{table_sparsity_4}
\end{indented}
\end{table}


In the second experiment, we consider the $\ell_1$ SVM regression model \eqref{SVM_epsilon_insensitive_loss} with the $\epsilon$-insensitive loss function. The dataset we use is ``Mg'' as mentioned in subsection \ref{MNIST_database}. The kernel involves in model \eqref{SVM_epsilon_insensitive_loss} is chosen as the Gaussian Kernel defined by \eqref{Gaussian_Kernel} with $\mu=1.5$ and $n=1,000$.
We take  $\epsilon=10^{-4}$ involved in $\epsilon$-insensitive loss for $\ell_1$ SVM regression model \eqref{SVM_epsilon_insensitive_loss}. We choose seven different values of parameter $\lambda$. The minimization problem \eqref{SVM_epsilon_insensitive_loss} with each value of $\lambda$ is solved by the FPPA with $\bm{\varphi}:=\lambda\|\cdot\|_1\circ\mathbf{B}$, $\bm{\omega}:=\bm{\phi}_{\mathbf{y},\epsilon}$ and $\mathbf{C}:=\mathbf{K}'$. 

\begin{table}[ht]
\caption{Numerical results of $\ell_{1}$ SVM regression model with $\epsilon$-insensitive loss by millions of iterations}
\vspace{0.2cm}
\begin{indented}
\lineup
\item[]\begin{tabular}{@{}cc|c|c|c|c|c|c|c}  
\hline\hline      
&$\lambda$              
&$0.01$ &$0.4$  &$1$ &$2$ &$4$ &$18.00$ &$135.8091$ \\ \hline        
 &$\gamma$
 &$0.005$ &$0.1045$ &$0.8980$ &$1.7528$ &$3.2839$ &$17.8510$ &$135.8091$ \\ \hline 
&$\mathrm{SL}$  
&$305$ &$33$ &$19$ &$11$ &$7$ &$5$ &$0$ \\ \hline 
&$\mathrm{TRM}$ 
&$0.0145$ &$0.0165$ &$0.0177$ &$0.0185$ &$0.0200$ &$0.0207$ &$0.0530$
\\ \hline        
&$\mathrm{TEM}$              
 &$0.0157$ &$0.0162$ &$0.0170$ &$0.0177$ &$0.0193$ &$0.0202$
 &$0.0530$ \\        
\hline \hline  
\end{tabular}
\label{table_sparsity_7}
\end{indented}
\end{table}

Associated with the numerical approximation $\mathbf{u}^{\infty}$, we identify $l$ distinct integers $k_i\in\mathbb{N}_{n}$ so that $\mathbf{B}\mathbf{u}^{\infty}:=\sum_{i\in\mathbb{N}_l}z^{\infty}_{k_i}\mathbf{e}_{k_i}$, $z^{\infty}_{k_i}\in\mathbb{R}\setminus{\{0\}}$, $l\in\mathbb{N}_l$. That is, $\mathbf{B}\mathbf{u}^{\infty}$ has $l$ nonzero components. We compute the numbers 
$\gamma:=\mathrm{max}\left\{\mathrm{min}\left\{|(\mathbf{K}_j)^{\top}\mathbf{c}|:\mathbf{c}\in\partial\bm{\phi}_{\mathbf{y},\epsilon}(\mathbf{K}'\mathbf{u}^{\infty})\right\}: \ j\in \mathbb{N}_n\setminus  \{k_i:i\in\mathbb{N}_l\}\right\}.$ The selected values of parameter $\lambda$, the values of $\gamma$, the sparsity levels $\mathrm{SL}$ of $\mathbf{B}\mathbf{u}^\infty$, the $\mathrm{TRM}$ values and the  $\mathrm{TEM}$ values are reported in Table  \ref{table_sparsity_7}. Similar to the first experiment that the numerical results reported in Table \ref{table_sparsity_41} coincides with inequality \eqref{lambda_general-example1_2} in Corollary \ref{general-example1}, these numerical results reported in Table  \ref{table_sparsity_7} confirm inequalities \eqref{lambda_general-example2_2} in Corollary \ref{general-example2}. 

\subsection{Signal denoising by using DWT and DCT}
In this subsection, we test the result in Theorem \ref{choice_error} by considering the signal denoising model \eqref{lasso_noise_data} with matrix $\mathbf{A}\in\mathbb{R}^{n\times n}$ being chosen as an orthogonal discrete wavelet transform and the discrete cosine transform separably. Note that Theorem \ref{choice_error} applied to these cases leads to Corollary \ref{choice_error_orthogonal}. We will confirm that the parameter strategies proposed in Corollary \ref{choice_error_orthogonal} can balance sparsity of the regularized solution and its approximate accuracy. For each chosen value of parameter $\lambda$, the resulting Lasso regularized model \eqref{lasso_noise_data} is solved by the FPPA with $\bm{\varphi}:=\lambda\|\cdot\|_1$, $\bm{\omega}:=\frac{1}{2}\|\cdot-\mathbf{x}^{\delta}\|_2^2$ and $\mathbf{C}:=\mathbf{A}$. 

In this experiment, we again consider recovering the Doppler signal function $f$ defined by \eqref{Doppler_signal}. As in subsection \ref{Group_lasso_signal}, we take $n:=4096$ and generate sample points $t_j, j\in\mathbb{N}_{n}$ and the original signal $\mathbf{f}$. The noisy signal is modeled as $\mathbf{x}^\delta:= \mathbf{f}+\bm{\eta}$, where $\bm{\eta}$ is the additive white Gaussian noise and the noise level $\delta:=\|\bm{\eta}\|_2$.


\begin{table}[ht]
\caption{\label{wavelet_sparsity_1m} Numerical results for the signal denoising model (with $\delta= 8.4682$) using the DB6 wavelet transform with 1,000 iterations}
\vspace{0.2cm}
\begin{indented}
\lineup
\item[]\begin{tabular}{@{}cc|c|c|c|c|c|c|c}
\hline\hline
&$\lambda$  &$0.1381$ &$0.1872$  &$0.2251$ &$0.3343$ &$0.6606$ &$3.9509$ &$7.1753$\\ \hline 
&$\mathrm{SL}$  &$1257$ &$718$  &$420$ &$102$ 
&$31$ &$7$  &$0$ \\   \hline 
&$\mathrm{MSE}$  &$0.0029$ &$0.0018$  &$0.0015$ &$0.0019$ 
&$0.0048$ &$0.0497$  &$0.0858$ \\   \hline \hline 
\end{tabular}
\end{indented}
\end{table}

In the first case, we choose the discrete wavelet transform matirx $\mathbf{A}$ as the Daubechies wavelet with $\mathrm{N}=6$ and $\mathrm{L}=4$. 
We test both Items (i) and (ii) of Theorem \ref{choice_error}. We first test 
Item (i). We choose the parameter $\lambda$ according to Statement (i) of 
Corollary \ref{choice_error_orthogonal}. Specifically, we define the sequence 
$a_j^{\delta}, j\in\mathbb{N}_{n}$, by \eqref{sequence_a_j} with $d=n$, 
rearrange the sequence in a nondecreasing order: $a_{k_1}^{\delta}\leq 
a_{k_2}^{\delta}\leq \cdots \leq a_{k_{n}}^{\delta}$ with distinct 
$k_i\in\mathbb{N}_{n}, i\in\mathbb{N}_{n}$, and pick 
$\lambda=a_{k_{n-j}}^\delta$ for $j=1257$, $718$, $420$, $102$, $31$, $7$, $0$ 
with indices of $j$ selected randomly.
We report in Table \ref{wavelet_sparsity_1m} the selected values of parameter $\lambda$, the sparsity levels $\mathrm{SL}$ of the numerical approximation  $\mathbf{u}^{\infty}$ and the MSE values of the denoised signals $\mathbf{u}^{\infty}$ with the noise level $\delta=8.4682$, where the MSE value is defined as subsection \ref{TV_signal_denoising}. These numerical results confirm the theoretical result stated in the Statement (i) of Corollary \ref{choice_error_orthogonal}, which guarantees the denoised signals have sparsity of levels $\leq l$ with $l=1257$, $718$, $420$, $102$, $31$, $7$, $0$, corresponding to the chosen values of $\lambda$. Moreover, the $\mathrm{MSE}$ values of the denoised signals depending on $\lambda$ exhibit a pattern like $\lambda+1/\lambda$, which is essentially described in Statement (i) Corollary \ref{choice_error_orthogonal} as an upper bound of the approximation error.

\begin{table}[ht] 
\caption{\label{estimate_error_DWT}
Numerical results for the signal denoising model (with $\lambda:=1.2\delta$) using the DB6 wavelet transform with 1,000 iterations}
\vspace{0.2cm}
\begin{indented}
\lineup
\item[] 
\begin{tabular}{@{}cc|c|c|c|c|c|c|c}
\hline\hline
&$\delta$  &1.9e-9 &6.0e-7  &1.9e-4 &1.9e-2 &6.0e-1 &1.9 &6.0\\ \hline 
&$\mathrm{SL}$  &$1049$ &$583$  &$274$ &$119$ 
&$29$ &$15$  &$0$ \\   \hline 
&$\mathrm{MSE}$ &1.4e-18 &7.8e-14  &3.8e-9 &1.7e-5 
&5.5e-3 &2.6e-2  &8.6e-2 \\ 
\hline 
&$\mathrm{MSE}/\delta$ &7.4e-10 &1.3e-7  &2.0e-5 &8.9e-4
&9.2e-3 &1.4e-2  &1.4e-2 \\
\hline \hline 
\end{tabular}
\end{indented}
\end{table}

We now test Statement (ii) of Corollary \ref{choice_error_orthogonal}. According to Statement (ii) of Corollary \ref{choice_error_orthogonal}, we choose $\lambda:=C\delta$ with $C=1.2$ and seven different noise levels shown in the first row of Table  \ref{estimate_error_DWT}. We report numerical results of this case in Table \ref{estimate_error_DWT} where the noise levels $\delta$, the sparsity levels $\mathrm{SL}$ of the numerical approximation $\mathbf{u}^{\infty}$, the $\mathrm{MSE}$ values of the denoised signals $\mathbf{u}^{\infty}$ and $\mathrm{MSE}/\delta$ are listed.
These numerical results verify the theoretical finding presented in Statement (ii) of Corollary \ref{choice_error_orthogonal} which guarantees the regularized solution has sparsity of a prescribed level and a bound $C'\delta$ for its approximation error. We comment on the numerical values of $\mathrm{MSE}/\delta$ listed in the last row of Table \ref{estimate_error_DWT}. Even though from Corollary \ref{choice_error_orthogonal}, $C'$ should be a constant independent of $\delta$, the numerical results for this particular example show that $C'$ decreases as $\delta$ become smaller. Perhaps this may be an indication of a superconvergence phenomenon for this example or it hints that the estimate in Corollary \ref{choice_error_orthogonal} may be improved for certain special cases. We plan to further understand this phenomenon.

For the parameter choice strategy $\lambda:=C\delta$ proposed in Statement (ii) of Corollary \ref{choice_error_orthogonal}, we also consider two values of $\delta$ ($\delta=1.9\times 10^{-9}, 1.9\times 10^{-4}$), and four values of $C$ ($C=0.12, 1.2, 12, 120$). Numerical results for this case can be found in Table \ref{estimate_error_DWT2} where the noise level $\delta$, the constants $C$, the sparsity levels $\mathrm{SL}$ of $\mathbf{u}^{\infty}$ and the $\mathrm{MSE}$ values of the denoised signals $\mathbf{u}^{\infty}$ are reported. They are consistent with the estimate shown in Statement (ii) of Corollary \ref{choice_error_orthogonal}.



\begin{table}[ht]
\caption{\label{estimate_error_DWT2} Numerical results for the signal denoising model (with $\lambda:=C\delta$) using the DB6 wavelet transform with 1,000 iterations}
\vspace*{0.2cm}
\begin{indented}
\lineup
\item[]\begin{tabular}{@{}c|c|c|c|c|c|c|c|cc}
\hline\hline
$\delta$
&\multicolumn{4}{c|}{1.9e-9} &\multicolumn{4}{c}{1.9e-4}
\\
\hline
C &0.12 &1.2 &12 &120 &0.12 &1.2  &12  &120 \\
\hline
$\mathrm{SL}$  &1300 &1049 &833 &662 &385 &274 &187 &118\\
\hline
$\mathrm{MSE}$ &1.8e-20 &1.4e-18 &1.1e-16 &8.8e-15 &5.2e-11 &3.8e-9 &2.6e-7 &1.7e-5\\
\hline\hline
\end{tabular}
\end{indented}
\end{table}

The second case repeats the first case with $\mathbf{A}$ being replaced by the discrete cosine transform. 
Numerical results for this case are reported in Tables \ref{DCT_sparsity_1m}, \ref{estimate_error_DCT} and \ref{estimate_error_DCT2}, which again confirm Statements (i) and (ii) of Corollary \ref{choice_error_orthogonal}. Note that for this case the MSE/$\delta$ values shown in the last row of Table \ref{estimate_error_DCT} behave more like a constant.


\begin{table}[ht]
\caption{\label{DCT_sparsity_1m}
Numerical results for the signal denoising model (with $\delta = 8.4682$) using the discrete cosine transform with 1,000 iterations}
\vspace{0.2cm}
\begin{indented}
\lineup
\item[]\begin{tabular}{@{}cc|c|c|c|c|c|c|c}
\hline\hline
&$\lambda$  &$0.1435$ &$0.1960$ &$0.2426$ &$0.3820$ &$1.0869$ &$4.0588$ &$8.2623$\\ \hline 
&$\mathrm{SL}$  &$1257$ &$718$  &$420$ &$102$ 
&$31$ &$7$  &$0$ \\   \hline 
&$\mathrm{MSE}$  &$0.0040$ &$0.0033$  &$0.0034$ &$0.0049$ 
&$0.0159$ &$0.0593$  &$0.0858$ \\ \hline \hline 
\end{tabular}
\end{indented}
\end{table}



\begin{table}[ht] 
\caption{\label{estimate_error_DCT}
Numerical results for the signal denoising model (with $\lambda := 1.4\delta$) using the discrete cosine transform with 1,000 iterations}
\vspace{0.2cm}
\begin{indented}
\lineup
\item[] \begin{tabular}{@{}cc|c|c|c|c|c|c|c}
\hline\hline
&$\delta$  &6.0e-4 &1.9e-2  &6.0e-2 &1.9e-1 &6.0e-1 &1.9 &6.0\\ \hline 
&$\mathrm{SL}$  &$1172$  &$558$  &$292$ &$115$ 
&$27$ &$11$  &$0$ \\   \hline 
&$\mathrm{MSE}$ &2.3e-7 &1.1e-4  &6.9e-4 &3.2e-3 
&1.8e-2 &4.1e-2  &8.6e-2 \\   \hline 
&$\mathrm{MSE}/\delta$ &3.8e-4 &5.8e-3  &1.2e-2 &1.7e-2 
&3.0e-2 &2.2e-2  &1.4e-2 \\      \hline \hline 
\end{tabular}
\end{indented}
\end{table}



\begin{table}[ht]
\caption{\label{estimate_error_DCT2}
Numerical results for the signal denoising model (with $\lambda := C\delta$) using the discrete cosine transform with 1,000 iterations}
\vspace*{0.2cm}
\begin{indented}
\lineup
\item[]\begin{tabular}{@{}cc|c|c|c|c|c|c|c|cc}
\hline\hline
&$\delta$ &\multicolumn{4}{c|}{6.0e-4} &\multicolumn{4}{c}{1.9e-1}
\\
\hline
&C &0.14 &1.4 &14 &140 &0.14 &1.4  &14  &140 \\
\hline
&$\mathrm{SL}$  &2178 &1172 &725 &291 &557 &115 &11 &0\\
\hline
&$\mathrm{MSE}$ &4.7e-9 &2.3e-7 &1.4e-5 &6.9e-4 &1.1e-4 &3.2e-3 &4.1e-2 &8.6e-2\\
\hline\hline
\end{tabular}
\end{indented}
\end{table}


\section{Conclusion}

We have studied choice strategies of the regularization parameter for the regularization problem with a $\ell_1$ norm regularization. The strategies are proposed to balance sparsity of a regularized solution and its approximation accuracy compared to the corresponding minimal solution. The ingredient used in developing the strategies is the connection of the choice of the parameter with the sparsity level of the regularized solution and 
with the approximation error. Much effort of this paper is given to understanding the connection between the choice of the parameter and the sparsity level of the regularized solution. We have also demonstrated how such understanding is combined with an error bound of the regularized solution to obtain a strategy for choices of the parameter to balance its sparsity and approximation accuracy. We have conducted substantial numerical experiments to test the proposed strategies. Numerical results of various application models confirm our theoretical estimates. More extensive applications of the proposed strategies will be our future research projects.

\ack
R. Wang is supported in part by the National Key Research and Development Program of China (grants no. 2020YFA0714100 and 2020YFA0713600) and by the Natural Science Foundation of China under grant 12171202; Y. Xu is supported in part by the US National Science Foundation under grant DMS-1912958 and by the US National Institutes of Health under grant R21CA263876.

\section*{References}

\end{document}